\documentclass[10pt,reqno]{amsproc}
\usepackage{MnSymbol}
\usepackage{stmaryrd}
\usepackage{stmaryrd}
\usepackage{fullpage}
\usepackage{graphicx}
\usepackage{subfigure}
\usepackage{psfrag}
\usepackage[usenames, dvips]{color}
\usepackage{algorithm}
\usepackage{algorithmic}
\usepackage{float}
\usepackage{enumerate}
\usepackage{cite}
\usepackage{rotating}
\usepackage[small]{caption}
\usepackage{morefloats}
\newtheorem{theorem}{Theorem}

\newtheorem{remark}{Remark}

\title[A multi-time-step coupling method]{On multi-time-step monolithic coupling algorithms 
for elastodynamics}

\author{S.~Karimi} \author{K.~B.~Nakshatrala \\
{\small Department of Civil and Environmental Engineering, 
University of Houston, Houston, Texas 77204--4003.}\\
{\small Correspondence to: \textbf{\emph{e-mail:}} 
knakshatrala@uh.edu, \textbf{\emph{Phone:}}~+1-713-743-4418}}

\begin{document}
%=========================;
%  Abstract and Keywords  ;
%=========================;
\begin{abstract}
  We present a way of constructing multi-time-step 
  monolithic coupling methods for elastodynamics. 
  The governing equations for constrained multiple 
  subdomains are written in dual Schur form and 
  enforce the continuity of velocities at system 
  time levels. The resulting equations will be in 
  the form of differential-algebraic equations. 
  To crystallize the ideas we shall employ Newmark 
  family of time-stepping schemes. The proposed 
  method can handle multiple subdomains, and 
  allows different time-steps as well as different 
  time stepping schemes from the Newmark family in 
  different subdomains.
  We shall use the energy method to assess the numerical 
  stability, and quantify the influence of perturbations 
  under the proposed coupling method. Two different notions
  of energy preservation are introduced and employed to
  assess the performance of the proposed method.
  Several numerical examples are presented to illustrate 
  the accuracy and stability properties of the proposed 
  method. We shall also compare the proposed multi-time-step 
  coupling method with some other methods available 
  in the literature.
\end{abstract}

\maketitle

%=================================;
%  Include all the sections here  ;
%=================================;

%*********************************************;
%                                             ;
%  NAME                                       ;
%    S1_Monolithic_Intro.tex                  ;
%                                             ;
%*********************************************;
\section{INTRODUCTION AND MOTIVATION}
\label{Sec:Monolithic_Introduction}

%==============================================;
%  Coupled problem, importance and difficulty  ;
%==============================================;
Coupled problems (such as fluid-structure interaction, 
structure-structure interaction and thermal-structure 
interaction) have been the subject of intense research 
in recent years in both computational mechanics and 
applied mathematics. The report compiled by the Blue 
Ribbon Panel on Simulation-Based Engineering Science 
emphasizes that the ability to solve coupled problems 
will be vital to accelerate the advances in engineering 
and science through simulation 
\cite{NSF_Blue_Ribbon_Panel_2006}.
Developing stable and accurate numerical strategies 
for coupled problems can be challenging due to several 
reasons. These problems may involve multiple temporal 
scales and different spatial scales. One may have to 
deal with different types of equations for different 
aspects of physics, which could be coupled nonlinear 
equations. It is noteworthy that there exists neither 
a complete mathematical theory (for existence, uniqueness, 
and sharp estimates) nor a comprehensive computational 
framework to solve any given coupled problem. Some of 
the current research efforts are targeted towards 
resolving the aforementioned issues. Other research 
efforts are towards developing linear and nonlinear 
solvers, parallel frameworks, and tools for heterogeneous 
computing environments (including GPU-based computing) 
for coupled problems.

Herein, we shall present a numerical approach that can 
handle moderate disparity in temporal scales. We shall 
take elastodynamics as the benchmark problem, as it 
serves two purposes. This problem is important in its 
own right. In addition, the problem serves as a model 
problem to develop numerical algorithms for fluid-structure 
interaction problems, which can be much more involved than 
a problem typically encountered in elastodynamics. In a 
fluid-structure interaction simulation, in addition to a 
coupling algorithm, robust mesh motion algorithms, data 
transfer algorithms to interpolate data across mismatching 
meshes, and stable solvers for fluids and solids are needed. 

It is now well-recognized that neither implicit 
nor explicit time-stepping schemes will be totally 
advantageous to meet all the desired features in a 
numerical simulation (e.g., see the discussion in 
references \cite{Gravouil_Combescure_IJNME_2001_v50_p199,
Nakshatrala_Hjelmstad_Tortorelli_IJNME_2008_v75_p1385}). 
Many factors (which include mesh, physical properties of 
the subdomain, accuracy, stability, total time of interest) 
affect the choice of the time-stepping scheme(s) 
\cite{Hairer_Norsett_Wanner_vol1}. It is sometimes 
much more economical to adopt different time-steps 
and/or time-stepping schemes in different subdomains.
To this end mixed methods and multi-time-step methods 
have been developed. 

%=================================================;
%  Subsection: Multi-time-step and mixed methods  ;
%=================================================;
\subsection{Multi-time-step and mixed methods} 
Mixed methods refer to a class of algorithms 
that employ different time-stepping schemes 
in different subdomains. Some early efforts 
on mixed methods are 
\cite{1979_Belytschko_Yen_Mullen_CMAME_v17_p259,
Belytschko_Liu_ComputStruc_1982_v12_p445 
,Belytschko_Mullen_FEMNM_1977_v2_p697 
,Belytschko_Mullen_IJNME_1978_12_p1575 
,Hughes_IJNME_1980_v15_p1413
,Hughes_Liu_JAM_1978_v45_p371
,Hughes_Pister_Taylor_CMAME_1979_v17_p159}. 
The use of different time-steps in different 
subdomains is referred to as multi-time-stepping 
or subcycling. Some representative works in this 
direction are 
\cite{Belytschko_Smolinski_Neal_IJNME_1988_v26_p349,
Daniel_CM_1997_v20_p272,
Nakshatrala_Nakshatrala_Tortorelli_IJNME_2009_v78_p1387}.
But many of the prior efforts on mixed methods 
and multi-time-stepping suffer from one or more 
of the following deficiencies: 
(i) The method cannot handle multiple subdomains. 
(ii) The method may not be accurate for disparate 
material properties, and for highly graded meshes. 
(iii) The method may suffer from very stringent 
stability limits, which may not be practical to 
meet realistic problems.
(iv) The accuracy and stability depend on the preferential 
  treatment of certain subdomains. For example, in the 
  application of the conventional staggered coupling 
  method, one domain is made to advance before 
  another. The accuracy and stability depends on 
  the choice of the subdomain that has to advance 
  first \cite{Akkasale_MSThesis_TAMU_2011}. 

We conjecture that the main source of the 
aforementioned numerical deficiencies is 
due to the fact that the prior works tried 
to develop coupling methods for transient 
problems by extending the strategies that 
were successful in developing partitioned 
schemes for static problems. However, it 
should be emphasized that designing coupling 
algorithms or partitioned schemes for transient 
problems require special attention compared to 
static problems. 
The governing equations for both undecomposed 
and decomposed static problems are algebraic 
equations. In the case of transient problems, 
the governing equations of an undecomposed 
problem are Ordinary Differential Equations (ODEs) 
whereas the governing equations of a decomposed 
problem are Differential-Algebraic Equations 
(DAEs).

Many of the prior works just 
employed the time-stepping schemes that are 
primarily developed for ODEs to construct 
partitioned schemes. However, it is well-known 
in the numerical analysis literature that care 
should be taken in applying popular time 
integrating schemes developed for ODEs to 
solve DAEs. The title of Petzold's seminal work 
\cite{Petzold_SIAMJSciStatComp_1982_v3_p367} 
-- ``Differential/algebraic equations are not 
ODEs'' -- succinctly summarizes this fact. This 
viewpoint was also taken in references 
\cite{Nakshatrala_Hjelmstad_Tortorelli_IJNME_2008_v75_p1385,
Nakshatrala_Prakash_Hjelmstad_JCP_2009_v228_p7957} 
to develop coupling methods for first-order transient 
systems.

\emph{This paper aims to develop a coupling method that 
allows different time-steps and different time 
integrators in different parts of the computational 
domain, which will be achieved using the results 
from the theory of differential-algebraic equations 
(e.g., Ascher and Petzold \cite{Ascher_Petzold}).} 
In recent years, the trend is to use dual Schur 
approach to develop multi-time-step coupling algorithms 
for second-order transient systems. A notable work in 
this direction is by Gravouil and Combescure (e.g., 
\cite{Gravouil_Combescure_IJNME_2001_v50_p199}, which 
we shall refer to as the GC method. 
Based on the GC method, Pegon and Magonette
developed a parallel inter-field method (the PM method), 
reference \cite{Bursi_Pegon_IJNME_2008_v75} is devoted to 
analysis of this method. 
Bursi et al. extended the PM method by employing
the generalized $\alpha$-method in 
\cite{Bursi_Pegon_JCAM_2010_v234}. Real time 
partitioned time-integration using the LSRT methods
has been of interest recently in
\cite{Jia_Bursi_IJNME_v87_2011}. Mahjoubi and Krenk 
proposed a multi-time-step
coupling method using state-pace time integration in
\cite{2010_Mahjoubi_Krenk_IJNME_v83_p1700}, a more
general presentation of which appears in 
\cite{2011_Mahjoubi_Gravouil_Combescure_Greffet_CMAME_v200_p1069}.
Another work that is relevant to the current paper 
is by Prakash and Hjelmstad 
\cite{Prakash_Hjelmstad_IJNME_2004_v61_p2183}, 
which we shall refer to as the PH method. It 
is worth to critically review the GC and PH 
methods. 

%============================================================;
%  Subsection: A critical analysis of the GC and PH methods  ;
%============================================================;
\subsubsection{A critical analysis of the GC and PH methods}
The GC method is a multi-time-step coupling method 
for structural problems based on Newmark family of 
time integrators. The GC coupling method is built 
based on the following assumptions:
%----------------------------;
%  Assumptions in GC method  ;
%----------------------------;
\begin{enumerate}[(GC1)]
\item Enforcing the continuity of velocity on the interface
at the fine time-steps.
\item Linear interpolation of interface velocities.
\item Linear interpolation of Lagrange multiplier 
within the coarsest time-step. 
\end{enumerate} 
The GC method is shown to exhibit excessive 
numerical damping (for example, see reference 
\cite{Prakash_Hjelmstad_IJNME_2004_v61_p2183} 
and the numerical results presented in Section 
\ref{Sec:S6_Monolithic_SDOF} of this paper). 
The PH method is based on a modification to 
the GC method, and is constructed based on 
the following assumptions: 
\begin{enumerate}[(PH1)]
\item Employed continuity of velocities along the 
  subdomain interface at coarse time-steps.
\item Linear interpolation of \emph{all} kinematic 
  variables (displacements, velocities, accelerations 
  of the nodes on the subdomain interface and in the 
  interior of the subdomains) within a coarse time-step. 
\item The method as it is presented in reference 
  \cite{Prakash_Hjelmstad_IJNME_2004_v61_p2183} 
  is valid only for two subdomains. 
\item The subdomain that has the largest time-step
has a more significant role in formulating the 
algorithm.
\end{enumerate}
In Section \ref{Sec:S4_Monolithic_Proposed}, we shall 
show that Assumption (PH2) is not consistent with the 
underlying physics and need not be consistent with 
the underlying numerical time-stepping scheme. It is 
also claimed that the PH method is energy preserving 
implying that the coupling does not affect the total 
physical energy of the system. In a subsequent section, 
we shall present various notions of energy preserving 
by a coupling algorithm, and show that the PH method 
is not energy preserving (on the contrary to what 
has been claimed in Reference 
\cite{Prakash_Hjelmstad_IJNME_2004_v61_p2183}). 

%================================================;
%  Subsection: Main contributions of this paper  ;
%================================================;
\subsection{Main contributions of this paper} 
The proposed coupling method is developed by selecting 
the ideal combination from the assumptions of the 
GC and PH methods, and thereby eliminating all the 
deficiencies that these two methods suffer from. 
% Specifically, the proposed coupling method is 
% based on the continuity of velocities at system 
% time-steps\footnote{System time-step is defined 
% in Section \ref{Sec:S4_Monolithic_Proposed}.} 
% and linear interpolation of Lagrange multipliers 
% within a system time-step. 
%
This paper has made several advancements in 
multi-time-step coupling of second-order 
transient systems, and some of the main 
ones are as follows:
\begin{enumerate}[(i)]
\item Developing a coupling method that 
can handle multiple subdomains, allows 
different time-steps in different subdomains, 
allows different time-stepping schemes under the 
Newmark family in different subdomains, and is 
stable and accurate.
\item A stability proof using the energy method to obtain 
  sufficient conditions for multi-time-step coupling is 
  presented. Unlike many of the earlier works, the contribution
  of interface and subdomains is taken into account to derive
  the stability criteria. Unlike the prior works on multi-time-step 
  coupling \cite{Gravouil_Combescure_IJNME_2001_v50_p199,
    Prakash_Hjelmstad_IJNME_2004_v61_p2183}, the proof is 
  constructed by taking into account the contributions 
  from all the subdomains and the interface, which is 
  the correct form.
\item Documented the deficiencies of backward difference 
  formulae (BDF) and implicit Runge-Kutta (IRK) schemes 
  (which are popular for solving differential-algebraic 
  equations) for solving second-order transient systems 
  with invariants (e.g., conservation of energy). 
\item New notions of energy preservation are 
	introduced and conditions under which the proposed method 
	satisfies any of those notions are also derived.  
\item A systematic study (both on the theoretical and 
  numerical fronts) on the effect of subcycling and 
  system time-step on the accuracy is presented. 
  Specifically, we have shown that subcycling 
  need not always improve accuracy. A criterion is 
  devised to guide whether subcycling will improve 
  accuracy or not. An attractive feature is that 
  this criterion can be calculated on the fly 
  during a numerical simulation.
\end{enumerate}

%=======================================;
%  Subsection: An outline of the paper  ;
%=======================================;  
\subsection{An outline of the paper} 
The remainder of this paper is organized as follows. 
Section \ref{Sec:S2_Monolithic_Newmark} briefly 
outlines Newmark family of time stepping schemes. 
Section \ref{Sec:S3_Monolithic_GE} presents the 
governing equations for multiple subdomains with 
a discussion on the numerical treatment of interface 
constraints. Section \ref{Sec:S4_Monolithic_Proposed} 
presents the proposed multi-time-step coupling method. 
A systematic theoretical analysis of the proposed 
coupling method (which includes stability analysis 
based on the energy method, influence of perturbations, 
bounds on interface drifts) is presented in Section 
\ref{Sec:S5_Monolithic_Stability}. In Section 
\ref{Sec:S6_Monolithic_SDOF}, some of the 
theoretical predictions are verified using 
a simple lumped parameter system. 
Section \ref{Sec:S7_Monolithic_Energy} discusses the 
conditions under which the multi-time-step coupling 
algorithm is energy conserving and the conditions 
under which it is energy preserving. 
Some deficiencies of employing backward difference 
formulae and implicit Runge-Kutta schemes for developing 
coupling algorithms for elastodynamics are discussed in 
Section \ref{Sec:S8_Monolithic_BDF_IRK}. Several 
representative numerical examples are presented in 
Section \ref{Sec:S9_Monolithic_NR} to illustrate the 
performance of the proposed coupling method. Conclusions 
are drawn in Section \ref{Sec:S10_Monolithic_Conclusions}.

%
%********************************************;
%                                            ;
%  NAME                                      ;
%    S2_Monolithic_Newmark.tex               ;
%                                            ;
%  WRITTEN BY                                ;
%    Kalyana Babu Nakshatrala                ;
%                                            ;
%********************************************;
\section{NEWMARK FAMILY OF TIME-STEPPING SCHEMES}
\label{Sec:S2_Monolithic_Newmark}
Consider a system of second-order ordinary differential 
equations of the following form: 
%-----------------------------------------;
%  Equation: System of second-order ODEs  ;
%-----------------------------------------;
\begin{align}
  \label{Eqn:Monolithic_system_ODEs}
  \boldsymbol{M} \ddot{\boldsymbol{u}}(t) + \boldsymbol{K} 
  \boldsymbol{u}(t) = \boldsymbol{f}(t) \quad t \in (0, T]
\end{align}
where $t$ denotes time, $T$ denotes the time interval of 
interest, $\boldsymbol{M}$ is a symmetric positive definite 
matrix, $\boldsymbol{K}$ is a symmetric positive semidefinite 
matrix, and a superposed dot denotes derivative with respect 
to the time. The above system of equations can arise from a 
semi-discrete finite element discretization of the governing 
equations in linear elastodynamics \cite{Hughes}. In this case, 
$\boldsymbol{M}$ is referred to as the mass matrix, $\boldsymbol{K}$ 
is the stiffness matrix, and $\boldsymbol{u}(t)$ is the nodal 
displacement vector. Of course, one has to augment the above 
equation with initial conditions, which, in the context of 
elastodynamics, will be the prescription of the initial 
displacement vector and the initial velocity vector.
One popular approach for solving equation 
\eqref{Eqn:Monolithic_system_ODEs} numerically
 is to employ a time-stepping scheme from the 
 Newmark family \cite{Newmark_JEMD_1959_v85_p67}. 
 We now present the Newmark time-stepping schemes 
 in the context of undecomposed problem (i.e., 
 the computational domain is not decomposed into 
 subdomains). In the subsequent sections, we shall 
 extend the presentation to multiple subdomains 
 with the possibility of using different time-steps
 and/or different time integrators under Newmark
 family in different subdomains. 

Let the time interval of interest $T$ be 
divided into $N$ sub-intervals such that
%-------------------------------;
%  Equation: Interval division  ;
%-------------------------------;
\begin{align}
  [0,T] = \bigcup_{n = 1}^{N} [t_{n-1},t_{n}]
\end{align}  
where $0 = t_0 < t_1 < \cdots < t_N = T$ are referred to 
as time levels. To make the presentation simple, we shall 
assume that the sub-intervals are uniform. That is, 
%-------------------------------;
%  Equation: Uniform time-step  ;
%-------------------------------;
\begin{align}
  \Delta t = t_{n} - t_{n-1} \quad \forall n = 1,\cdots, N
\end{align}
where $\Delta t$ is commonly referred to as the time-step. 
It should be, however, noted that the presentation can be 
easily extended to incorporate variable time-steps. 
%-----------------------------------;
%  Remark: Remark about time-steps  ;
%-----------------------------------;
\begin{remark}
  In our development of the proposed multi-time-step coupling 
  method, we shall use different kinds of time-steps (e.g., 
  subdomain time-step, system time-step). These time-steps 
  will be introduced in a subsequent section. For the present 
  discussion, such a distinction is not required, as for single 
  subdomain there is only one time-step. 
\end{remark}
We shall employ the following notation to denote 
displacement, velocity and acceleration nodal 
vectors at discrete time levels: 
%-----------------------------------------------------;
%  Equation: Displacement, velocity and acceleration  ;
%-----------------------------------------------------;
\begin{align}
\boldsymbol{d}^{(n)} = \boldsymbol{u}(t = t_n), \quad 
\boldsymbol{v}^{(n)} = \left. \frac{d \boldsymbol{u}}
{d t}\right\vert_{t = t_n}, \quad 
\boldsymbol{a}^{(n)} = \left. \frac{d^2 \boldsymbol{u}}
{d t^2}\right\vert_{t = t_n} 
\end{align}
Newmark family of time stepping schemes, which is 
a two-parameter family of time integrators, can be 
written as follows:
%----------------------------;
%  Equation: Newmark family  ;
%----------------------------;
\begin{subequations}
  \begin{align}
    \label{Eqn:Monolithic_Newmark_d}
    \boldsymbol{d}^{(n+1)} &= \boldsymbol{d}^{(n)} + 
    \Delta t \; \boldsymbol{v}^{(n)} + \frac{\Delta t^2}{2} 
    \left((1 - 2 \beta) \boldsymbol{a}^{(n)} + 2 \beta 
    \boldsymbol{a}^{(n+1)} \right) \\
    \label{Eqn:Monolithic_Newmark_v}
    \boldsymbol{v}^{(n+1)} &= \boldsymbol{v}^{(n)} + 
    \Delta t  \left((1 - \gamma) \boldsymbol{a}^{(n)} 
    + \gamma \boldsymbol{a}^{(n+1)} \right)
  \end{align}
\end{subequations}
where $\beta$ and $\gamma$ are user-specified parameters. 
A numerical solution at $(n+1)$-th time level can be obtained 
by simultaneously solving equations 
\eqref{Eqn:Monolithic_Newmark_d}--\eqref{Eqn:Monolithic_Newmark_v} 
with the following equation:
%------------------------------------------;
%  Equation: Discretized single subdomain  ;
%------------------------------------------;
\begin{align}
  \boldsymbol{M} \boldsymbol{a}^{(n+1)} + \boldsymbol{K} 
  \boldsymbol{d}^{(n+1)} = \boldsymbol{f}^{(n+1)}
\end{align}
where 
%--------------------------
%  Equation: f_n+1 vector  ;
%--------------------------;
\begin{align}
  \boldsymbol{f}^{(n+1)} := \boldsymbol{f}(t = t_{n+1})
\end{align}

It is well-known that one needs to choose $\gamma \geq 1/2$ 
for numerical stability \cite{Wood}. The time-stepping scheme will be 
unconditionally stable if $2 \beta \geq \gamma$, and will 
be conditionally stable if $2 \beta < \gamma$. 
Some popular time-stepping schemes under the Newmark 
family are the central difference scheme $(\gamma = 1/2, 
\beta = 0)$, the average acceleration scheme $(\gamma = 1/2, 
\beta = 1/4)$, and the linear acceleration scheme 
$(\gamma = 1/2, \beta = 1/6)$. The central difference scheme 
is also referred to as the velocity 
Verlet scheme, which is the case in the molecular dynamics 
literature (e.g., see reference \cite{Leimkuhler_Reich}).
The central difference scheme is explicit, second-order accurate, 
and conditionally stable. The average acceleration scheme is 
implicit, second-order accurate, and unconditionally stable. 
The linear acceleration scheme is implicit, second-order 
accurate, and conditionally stable. For further details 
on Newmark family of time-stepping schemes in the context 
of undecomposed problem, see references \cite{Hughes,Wood,Geradin_Rixen}.

%
%********************************************;
%                                            ;
%  Name                                      ;
%    S3_Monolithic_GE.tex                    ;
%                                            ;
%  Written By                                ;
%    Kalyana Babu Nakshatrala                ;
%                                            ;
%********************************************;
\section{GOVERNING EQUATIONS FOR MULTIPLE SUBDOMAINS}
\label{Sec:S3_Monolithic_GE}
We now write governing equations for multiple subdomains. We 
will also outline various ways to write subdomain interface 
conditions, and discuss their pros and cons. To this end, let 
us divide the domain $\Omega$ into $S$ non-overlapping subdomains, 
which will be denoted by $\Omega_1, \cdots, \Omega_S$. That is, 
%------------------------------------;
%  Equation: Divide into subdomains  ;
%------------------------------------;
\begin{align}
  \Omega = \bigcup_{i=1}^{S} \Omega_i \quad \mbox{and} 
  \quad \Omega_i \cap \Omega_j = \emptyset \quad \mbox{for} 
  \; i \neq j 
\end{align}
We shall assume that the meshes in the subdomains are 
conforming along the subdomain interface, as shown in 
Figure \ref{Fig:Monolithic_subdomains}. There are several 
ways to enforce the continuity along the interface, and 
hence, several ways to write the governing equations for 
multiple subdomains. Herein, we shall employ the dual Schur 
approach \cite{Toselli_DD}, which is also employed in the 
references that are relevant to this paper (i.e., references 
\cite{Gravouil_Combescure_IJNME_2001_v50_p199,
Prakash_Hjelmstad_IJNME_2004_v61_p2183}). 

We shall denote the number of displacement degrees-of-freedom 
in the $i$-th subdomain by $N_i$. The size of the velocity 
and acceleration nodal vectors in the $i$-th subdomain will 
also be $N_i \times 1$. The interface continuity conditions 
can be compactly written using signed Boolean matrices. A 
signed Boolean matrix is a matrix with entries either $0$, 
$-1$, or $+1$ such that each row has \emph{at most} one 
non-zero entry. Let us denote the total number of interface 
constraints by $N_C$. The size of the matrix $\boldsymbol{C}_i$ 
will be $N_C \times N_i$.

The governing equations for constrained multiple subdomains 
in a (time) continuous setting can be written as follows:
%--------------------------------------------;
%  Equation: Continuous GE for second-order  ;
%--------------------------------------------;
\begin{subequations}
\label{Eqn:Monolithic_GE_time_continuous}
  \begin{align}
  \label{Eqn:Monolithic_GE}
    &\boldsymbol{M}_i \ddot{\boldsymbol{u}}_i(t) +
    \boldsymbol{K}_i \boldsymbol{u}_i(t) = 
    \boldsymbol{f}_i(t) + \boldsymbol{C}_i^{\mathrm{T}} 
    \boldsymbol{\lambda}(t) \quad \forall i = 1, \cdots, S \\
    \label{Eqn:Monolithic_constraint_original}
    &\sum_{i = 1}^{S} \boldsymbol{C}_{i} \boldsymbol{u}_i 
    \left(t \right) = \boldsymbol{0}
  \end{align}
\end{subequations}
where the displacement vector of the $i$-th subdomain 
is denoted by $\boldsymbol{u}_i \left( t \right)$, and the 
external force applied to the $i$-th subdomain is denoted by 
$\boldsymbol{f}_i(t)$. The mass and stiffness matrices of the 
$i$-th subdomain are denoted by $\boldsymbol{M}_i$ 
and $\boldsymbol{K}_i$ respectively. In this paper, we shall assume that 
the matrices $\boldsymbol{M}_i$ are symmetric 
and positive definite, and the matrices $\boldsymbol{K}_i$ 
to be symmetric and positive semi-definite.
Equation \eqref{Eqn:Monolithic_constraint_original} is 
an algebraic constraint enforcing kinematic continuity 
of displacements along the subdomain interface. The 
vector $\boldsymbol{\lambda}$ is the vector of Lagrange 
multipliers arising due to the enforcement of constraints. 
The above equations should be augmented with appropriate 
initial conditions. A brief discussion on the derivation 
of the above equations can be found in Appendix. 
Equation \eqref{Eqn:Monolithic_GE_time_continuous} 
form a system of differential-algebraic equations. For 
the benefit of broader audience, we now briefly discuss 
differential-algebraic equations. 

%--------------------------------;
%  Remark: Remark about damping  ;
%--------------------------------;
\begin{remark}
  If one wants to including physical damping, 
  equation \eqref{Eqn:Monolithic_GE} should 
  be replaced with the following equation:
  %--------------------------------------------;
  %  Equation: Continuous GE for second-order  ;
  %--------------------------------------------;
  \begin{align}
    &\boldsymbol{M}_i \ddot{\boldsymbol{u}}_i(t) +
    \boldsymbol{D}_i \dot{\boldsymbol{u}}_i + 
    \boldsymbol{K}_i \boldsymbol{u}_i(t) = 
    \boldsymbol{f}_i(t) + \boldsymbol{C}_i^{\mathrm{T}} 
    \boldsymbol{\lambda}(t) \quad \forall i = 1, \cdots, S 
  \end{align}
  where $\boldsymbol{D}_i$ is the damping matrix 
  for the $i$-th subdomain. One can then easily 
  extend the proposed multi-time-step coupling 
  method to include contribution from physical 
  damping. However, a more challenging task is 
  to characterize the performance of the coupling 
  method due to damping. 
  This will depend on several issues like: whether 
  the damping is due to viscoelasticity, plasticity, 
  viscoplasticity or frictional contact? Whether the 
  damping matrix be modeled as Rayleigh damping (which 
  basically assumes that the damping matrix is a linear 
  combination of the mass matrix and the stiffness 
  matrix)? A systematic treatment of these issues 
  are beyond the scope of this paper, and will be 
  addressed in our future works. 
\end{remark}

%================================================;
%  Subsection: Differential-algebraic equations  ;
%================================================;
\subsection{Differential-algebraic equations}
A Differential-Algebraic Equation (DAE) is defined as an 
equation involving unknown functions and their derivatives. 
A DAE, in its most general form, can be written as follows: 
%---------------------------------;
%  Equation: DAE is general form  ;
%---------------------------------;
\begin{align}
  \label{Eqn:DAE_implicit}
  &\boldsymbol{g} \left( \dot{\boldsymbol{x}}(t),
  \boldsymbol{x}(t), t \right)= \boldsymbol{0} 
  \quad t \in (0,T] 
\end{align}
where the unknown function is denoted by $\boldsymbol{x}(t)$. 
A DAE of the form given by equation \eqref{Eqn:DAE_implicit} 
is commonly referred to as an implicit DAE. A quantity that 
is useful in the study of (smooth) differential-algebraic 
equations is the so-called differential index, which was first 
introduced by Gear \cite{Gear_SIAM_JNA_1990_v27_p1527} and 
further popularized by Petzold and Campbell \cite{Ascher_Petzold,
Brenan_Campbell_Petzold_DAE}. 
For a DAE of the form given by equation \eqref{Eqn:DAE_implicit}, 
differential index is the minimum number of times one has to 
differentiate with respect to the independent variable $t$ to 
be able to \emph{rewrite} equation \eqref{Eqn:DAE_implicit} in 
the following form:
\begin{align}
  \dot{\boldsymbol{x}}(t) = \boldsymbol{h}(\boldsymbol{x}(t),t)
\end{align} 
using only algebraic manipulations. It is commonly believed 
that the higher the differential index the greater will be 
the difficulty in obtaining stable numerical solutions. 
An important subclass of DAEs is titled as semi-explicit, 
which can be written as follows:
%--------------------------------;
%  Equation: Semi-explicit DAEs  ;
%--------------------------------;
\begin{subequations}
  \begin{align}
    \label{Eqn:DAE_semiexplicit}
    \dot{\boldsymbol{x}}(t) = \boldsymbol{p}(\boldsymbol{x}(t),t) \\
    \boldsymbol{0} = \boldsymbol{q}(\boldsymbol{x}(t),t)
  \end{align}
\end{subequations}

From the above discussion, it is evident that the DAE given by equations 
\eqref{Eqn:Monolithic_GE_time_continuous}
is a semi-explicit DAE with differential index 3. One way of solving 
a higher index DAE is to employ the standard index reduction technique 
to obtain a \emph{mathematically} equivalent DAE with lower differential 
index. It is noteworthy that index reduction can have deleterious effect 
on the stability and accuracy of numerical solutions (e.g., drift in the 
constraint). We now explore several mathematically equivalent 
forms of governing equations, which will have differential index ranging 
from 0 to 3.  

%===============================================;
%  Subsection: Subdomain interface constraints  ;
%===============================================;
\subsection{Subdomain interface constraints}
As stated earlier, dual Schur techniques for domain 
decomposition are of interest throughout this paper. 
One may write several types of continuity constraints 
resulting in semi-explicit DAEs of different differential 
indices. Note that in a continuous setting all these 
versions are mathematically equivalent. However, from 
a numerical point of view, their performance can be 
dramatically different. In fact, some may even exhibit 
instabilities. Some ways of constructing dual Schur 
methods are discussed below, which guide future 
research on constructing new multi-time-step 
coupling methods. 

%======================================;
%  Subsubsection: d-continuity method  ;
%======================================;
\textsf{$\boldsymbol{d}$-continuity method:} 
This method considers the original set of equations given 
by equations \eqref{Eqn:Monolithic_GE_time_continuous}. 
The method obtains $\left(\boldsymbol{u}_1 \left( t \right),\cdots,
\boldsymbol{u}_S\left( t \right), \boldsymbol{\lambda}(t) \right) $ for 
$t \in (0,T]$ by solving the following equations: 
%---------------------------------;
%  Equation: d-continuity method  ;
%---------------------------------;
\begin{subequations}
  \begin{align}
   \label{Eqn:Monolithic_dmethod_GE}
    &\boldsymbol{M}_{i} \ddot{\boldsymbol{u}}_{i}(t) 
    + \boldsymbol{K}_{i} \boldsymbol{u}_{i}(t) = 
    \boldsymbol{f}_{i}(t) + \boldsymbol{C}_{i}^{\mathrm{T}} 
    \boldsymbol{\lambda}\left( t \right) \quad \forall 
    i = 1,\cdots, S \\
    \label{Eqn:Monolithic_dmethod_constraint}
    &\sum_{i = 1}^{S} \boldsymbol{C}_{i} \boldsymbol{u}_{i}
    (t) = \boldsymbol{0} 
  \end{align}
\end{subequations}
The above equations 
\eqref{Eqn:Monolithic_dmethod_GE}--\eqref{Eqn:Monolithic_dmethod_constraint} 
form a system of DAEs of differential index \emph{three}. 
It has been discussed in the literature that the numerical 
solutions based on this method are prone to instabilities 
\cite{Geradin_Rixen,Geradin_Cardona}.

%======================================;
%  Subsubsection: v-continuity method  ;
%======================================;
\textsf{$\boldsymbol{v}$-continuity method:} This method 
obtains $\left( \boldsymbol{u}_1\left( t \right),\cdots,
\boldsymbol{u}_S\left( t \right),\boldsymbol{\lambda}(t) 
\right)$ for $ t\in (0,T]$ by solving the following 
equations: 
%---------------------------------;
%  Equation: v-continuity method  ;
%---------------------------------;
\begin{subequations}
\label{Eqn:vContinuity_GE}
  \begin{align}
  \label{Eqn:vContinuity_GE_subdomain}
    &\boldsymbol{M}_{i} \ddot{\boldsymbol{u}}_{i}(t) 
    + \boldsymbol{K}_{i} \boldsymbol{u}_{i}(t) = 
    \boldsymbol{f}_{i}(t) + \boldsymbol{C}_{i}^{\mathrm{T}}
    \boldsymbol{\lambda}\left( t \right) \quad \forall i = 1, \cdots, S\\
   \label{Eqn:vContinuity_GE_constraint}
    &\sum_{i = 1}^{S} \boldsymbol{C}_{i} \dot{\boldsymbol{u}}_{i}(t)
     = \boldsymbol{0} 
  \end{align}
\end{subequations}
The above equations form a system of DAEs of differential index 
\emph{two}. The $\boldsymbol{v}$-continuity method is of interest 
in this paper and in the previous works by Gravouil and Combescure 
\cite{Gravouil_Combescure_IJNME_2001_v50_p199}, and Prakash and 
Hjelmstad \cite{Prakash_Hjelmstad_IJNME_2004_v61_p2183}. This 
form of equations provides a simple but stable framework for 
seeking numerical solutions, and will form the basis for the 
proposed multi-time-step coupling method.

%======================================;
%  Subsubsection: a-continuity method  ;
%======================================;
\textsf{$\boldsymbol{a}$-continuity method:} This method 
obtains $\left(\boldsymbol{u}_1\left( t \right), \cdots, 
\boldsymbol{u}_S\left( t \right), \boldsymbol{\lambda}
\left( t \right)\right)$ for $t \in (0,T]$ by solving
 the following equations: 
%---------------------------------;
%  Equation: a-continuity method  ;
%---------------------------------;
\begin{subequations}
  \begin{align}
    &\boldsymbol{M}_{i} \ddot{\boldsymbol{u}}_{i}(t) 
    + \boldsymbol{K}_{i} \boldsymbol{u}_{i}(t) = 
    \boldsymbol{f}_{i}(t) + \boldsymbol{C}_{i}^{\mathrm{T}}
    \boldsymbol{\lambda}\left( t \right) \quad \forall 
    i = 1,\cdots,S\\
    &\sum_{i = 1}^{S} \boldsymbol{C}_{i} \ddot{\boldsymbol{u}}_{i}(t)
    = \boldsymbol{0} 
  \end{align}
\end{subequations}
The differential index of the above DAE is \emph{unity}. 
A drawback of this method is that there can be significant 
irrecoverable drift in the displacements without employing 
constraint stabilization or projection methods. The drift 
can be attributed to the fact that there is no explicit 
constraint on the continuity of displacements along the 
subdomain interface. We, therefore, do not employ this 
method in this paper.

%=======================================;
%  Subsection: Baumgarte stabilization  ;
%=======================================;
\textsf{Baumgarte stabilization method:} Under this method, 
kinematic constraint appears as a linear combination of the 
kinematic constraints under the $\boldsymbol{d}$-continuity, 
$\boldsymbol{v}$-continuity and $\boldsymbol{a}$-continuity 
methods. This method obtains $\left(\boldsymbol{u}_1(t),
\cdots,\boldsymbol{u}_S(t),\boldsymbol{\lambda}(t)\right)$ 
for $t \in (0,T]$ by solving the following equations: 
  \begin{subequations}
  \label{Eqn:DAE_Baumgarte}
    \begin{align}
      &\boldsymbol{M}_{i} \ddot{\boldsymbol{u}}_{i}(t) + \boldsymbol{K}_{i} 
      \boldsymbol{u}_{i}(t) = \boldsymbol{f}_{i}(t) + \boldsymbol{C}_{i}^{\mathrm{T}}
      \boldsymbol{\lambda}\left( t \right) \quad \forall i = 1,\cdots, S \\
      &\sum_{i=1}^{S} \boldsymbol{C}_{i} \ddot{\boldsymbol{u}}_{i}(t) 
      + \frac{\alpha}{\Delta t} \sum_{i=1}^{S} \boldsymbol{C}_{i} \dot{\boldsymbol{u}}_{i}(t) 
      + \frac{\beta}{\Delta t ^2} \sum_{i = 1}^{S} 
      \boldsymbol{C}_{i} \boldsymbol{u}_{i}(t) 
      = \boldsymbol{0} 
    \end{align}
  \end{subequations}
where $\alpha$ and $\beta$ are non-dimensional user-specified 
parameters. One can achieve damping in the drift displacements 
by choosing parameters satisfying the condition $\alpha^2 - 4 
\beta < 0$. This method was first proposed by Baumgarte in
\cite{Baumgarte_CMAME_1972_v1_p1} for constrained mechanical
systems. Note that in \cite{Baumgarte_CMAME_1972_v1_p1}, the
coefficients $\alpha$ and $\beta$ have dimensions of 
$\left[ T \right]^{-1}$ and $\left[ T \right]^{-2}$ respectively,
but in \eqref{Eqn:DAE_Baumgarte}, those coefficients are 
non-dimensionalized.
In Reference \cite{Nakshatrala_Prakash_Hjelmstad_JCP_2009_v228_p7957}, 
the Baumgarte stabilization method has been extended to first-order 
differential-algebraic equations, and the authors were able to derive 
sufficient conditions for stability using the energy method. 
To the best of the authors' knowledge deriving sufficient 
conditions for stability under the Baumgarte method for 
second-order differential-algebraic equations is still 
an open problem. Some notable efforts in this direction 
are \cite{1995_Yoon_Howe_Greenwood_ASME_JMD_v117_p446,
2008_Bauchau_JCND_v3_p1,flores2009investigation}.

%=============================================;
%  Subsection: Rewriting as a system of ODEs  ;
%=============================================;
\textsf{Rewriting as a system of ordinary differential equations:}
One can differentiate further, and rewrite the \textsf{a-continuity method} 
as a system of ordinary differential equations. From the definition 
of differential index, it is obvious that the differential index of 
the resulting governing equations will be \emph{zero}. The governing 
equations for this method take the following form:
%------------------------------;
%  Equation: A system of ODEs  ;
%------------------------------;
\begin{subequations}
  \begin{align}
    &\dot{\boldsymbol{u}_i} = \boldsymbol{v}_i \\
    &\dot{\boldsymbol{v}}_i = \boldsymbol{M}_i^{-1} 
    \left( \boldsymbol{f}_i + \boldsymbol{C}_i^{\mathrm{T}} 
    \boldsymbol{\lambda} - \boldsymbol{K}_i \boldsymbol{u}_i \right) \\
    &\boldsymbol{\lambda} = \left( \sum_{i=1}^{S} 
    \boldsymbol{C}_i \boldsymbol{M}_i^{-1} 
    \boldsymbol{C}_i^{\mathrm{T}}\right)^{-1} 
    \left[ \sum_{i = 1}^S \boldsymbol{C}_i \boldsymbol{M}_i^{-1}
      \left( \boldsymbol{K}_i \boldsymbol{u}_i - 
      \boldsymbol{f}_i\right) \right] 
 \end{align}
\end{subequations}
The main drawback of the above method is that there will be 
significant irrecoverable drift in the continuity of subdomain 
interface displacements and velocities. As advocated by Petzold 
in her famous paper \cite{Petzold_SIAMJSciStatComp_1982_v3_p367}, 
solving DAEs is much harder than solving systems of ODEs. Many 
of the popular integrators that are used for solving ODEs are 
not stable and accurate for solving DAEs.

%=========================================================;
%  Subsection: Rewriting as a system of first-order DAEs  ;
%=========================================================;
\textsf{Rewriting as a system of first-order differential-algebraic equations:}
Yet another approach is to rewrite the governing equations in 
first-order form, and then employ appropriate time-stepping 
schemes for solving first-order DAEs (e.g., backward difference 
formulae, implicit Runge-Kutta schemes). The first-order form 
can be achieved by introducing an auxiliary variable. The 
governing equations take the following form:
%-------------------------------;
%  Equations: First-order DAEs  ;
%-------------------------------;
\begin{subequations}
  \begin{align}
    &\dot{\boldsymbol{u}}(t) = \boldsymbol{v}(t) \\
    &\boldsymbol{M}_i \dot{\boldsymbol{v}}_i + 
    \boldsymbol{K}_i \boldsymbol{u}_i = \boldsymbol{f}_i(t) 
    + \boldsymbol{C}_i^{\mathrm{T}} \boldsymbol{\lambda} \\
    \label{Eqn:Monolithic_1st_order_u}
    &\sum_{i=1}^{S} \boldsymbol{C}_i \boldsymbol{u}_i = \boldsymbol{0}
  \end{align}
\end{subequations}
The differential index for the above system is \emph{three}. 
If one replaces the interface constraint equation 
\eqref{Eqn:Monolithic_1st_order_u} with either of 
the following:
\begin{align}
  \sum_{i = 1}^S \boldsymbol{C}_i \dot{\boldsymbol{u}}_i 
  = \boldsymbol{0} \quad \mathrm{or} \quad 
  \sum_{i = 1}^S \boldsymbol{C}_i \boldsymbol{v}_i = 
  \boldsymbol{0}
\end{align} 
then the differential index of the resulting differential-algebraic 
equations will be \emph{two}. If the interface constraint equation 
\eqref{Eqn:Monolithic_1st_order_u} is replaced with the following: 
\begin{align}
  \sum_{i = 1}^S \boldsymbol{C}_i \dot{\boldsymbol{v}}_i 
  = \boldsymbol{0}
\end{align}
then the resulting first-order DAEs will have index \emph{one}. 

In a subsequent section we shall show that the approach 
of rewriting the governing equations as first-order DAEs 
and then employing time-stepping schemes that are typically 
used for first-order transient systems is not accurate for 
elastodynamics. Hence, we do not employ such an approach 
to develop a multi-time-step coupling method. Instead, 
we consider the governing equations in second-order form 
and modify Newmark time-stepping schemes to be able to 
obtain stable and accurate results for resulting DAEs.
In the next section, we shall extend the 
$\boldsymbol{v}$-continuity to be able to 
employ different time-steps in different 
subdomains, and to couple explicit and 
implicit time-stepping schemes.

%
%************************************************;
%                                                ;
%  NAME                                          ;
%    S4_Monolithic_Proposed_coupling.tex         ;
%                                                ;
%  WRITTEN BY                                    ;
%    Saeed Karimi                                ;
%    Kalyana Babu Nakshatrala                    ;
%                                                ;
%************************************************;
\section{PROPOSED MULTI-TIME-STEP COUPLING METHOD}
\label{Sec:S4_Monolithic_Proposed}
The aim of this paper is to solve equations 
\eqref{Eqn:vContinuity_GE_subdomain}--\eqref{Eqn:vContinuity_GE_constraint} 
numerically by allowing each subdomain to have its own 
time-step and its own time integrator from the Newmark 
family of time stepping schemes. We first introduce 
notation that will help in presenting the proposed 
multi-time-step coupling method in a concise manner. 

%=====================================================;
%  Subsection: Notation for multi-time-step coupling  ;
%=====================================================;
\subsection{Notation for multi-time-step coupling}
Both the GC and PH methods are devised by introducing 
the coarsest time-step, which is the maximum of all 
the subdomain time-steps. This creates bias, at least 
in the mathematical setting, towards the subdomain 
that has the maximum time-step. 
Herein, we alleviate this drawback by introducing 
the notion of system time-step, which is greater 
than or equal to the coarsest time-step. Moreover, 
this approach allows for the possibility of all 
subdomains to subcycle, which is illustrated 
in a subsequent section. 
Figure \ref{Fig:Monolithic_multi_time_step_notation} 
gives a pictorial description of subdomain time-steps, 
system time-step, and the concept of subcycling. 
We shall define $\eta_i$ to be the ratio between 
system time-step ($\Delta t$) and the $i$-th 
subdomain time-step ($\Delta t_i$). That is,
%---------------------------------;
%  Equation: Definition of eta_i  ;
%---------------------------------;
\begin{align}
  \eta_i := \frac{\Delta t}{\Delta t_{i}}
\end{align}
For simplicity, we shall assume that $\eta_i$ is a (positive) 
integer.

We shall use the following notation to represent the 
value of a quantity of interest at subdomain time levels:
%-----------------------------------------------;
%  Equation: Quantity at subdomain time levels  ;
%-----------------------------------------------;
\begin{align}
  \Box_i^{\left( n + \frac{j}{\eta_i} \right)} \approx \Box_i 
  \left( t = n \Delta t + j \Delta t_i \right)
\end{align}
We shall employ the following notation to group the 
kinematic quantities:
%-----------------------------------------------------------------------;
%  Equation: \boldsymbol{X}_i, \mathbb{X}_i^n and \mathbb{X}^n vectors  ;
%-----------------------------------------------------------------------;
\begin{align}
  \boldsymbol{X}_{i}^{\left( n+ \frac{j}{\eta_i} \right)} := \left[ \begin{array}{c}	
      \boldsymbol{a}_{i}^{\left( n+ \frac{j}{\eta_i} \right)} \\
      \boldsymbol{v}_{i}^{\left( n+ \frac{j}{\eta_i} \right)} \\
      \boldsymbol{d}_{i}^{\left( n+ \frac{j}{\eta_i} \right)} \end{array}\right], \quad 
  \mathbb{X}_i^{(n+1)} := \left[\begin{array}{c}
      \boldsymbol{X}_{i}^{\left(n+\frac{1}{\eta_i}\right)} \\ 
      \boldsymbol{X}_{i}^{\left(n+\frac{2}{\eta_i}\right)} \\ 
      \vdots \\
      \boldsymbol{X}_{i}^{(n + 1)} \\ 
      \end{array} \right], \quad 
  \mathbb{X}^{(n+1)} := \left[\begin{array}{c}
      \mathbb{X}_{1}^{\left(n+1\right)} \\ 
      \mathbb{X}_{2}^{\left(n+1\right)} \\ 
      \vdots \\
      \mathbb{X}_{S}^{(n + 1)} \\ 
      \end{array} \right]
\end{align}
The vector $\boldsymbol{X}_i^{\left( n+ \frac{j}{\eta_i} \right)}$ 
contains all the kinematic unknowns for $i$-th subdomain over its 
subdomain time-step, $\mathbb{X}_i^{(n+1)}$ contains 
all the kinematic unknowns for $i$-th subdomain over 
a system time-step, and the vector $\mathbb{X}^{(n+1)}$ 
contains the kinematic unknowns of all subdomains over 
a system time-step. We define the following augmented 
subdomain signed Boolean matrices: 
%-----------------------------------;
%  Equation: \mathbb{C}_i matrices  ;
%-----------------------------------;
\begin{align}
  \mathbb{C}_i := \left[\begin{array}{c|c|c|c|c} 
      \underbrace{\mathcal{O}_i \quad \mathcal{O}_i 
        \quad \mathcal{O}_i}_{1} & 
      \underbrace{\mathcal{O}_i \quad \mathcal{O}_i 
        \quad \mathcal{O}_i}_{2} & 
      \cdots \; \cdots \; \cdots & 
      \underbrace{\mathcal{O}_i \quad \mathcal{O}_i 
        \quad \mathcal{O}_i}_{\eta_i - 1} & 
      \underbrace{\mathcal{O}_i \quad \boldsymbol{C}_i 
        \quad \mathcal{O}_i}_{\eta_i}
    \end{array} \right]
\end{align}
where the matrix $\mathcal{O}_i$ contains zeros of the same 
size as $\boldsymbol{C}_i$ (which is $N_C \times N_i$). It 
is evident that the size of $\mathbb{C}_i$ is $N_C \times 
3 \eta_i N_i$. The augmented signed Boolean matrix for the 
entire system is defined as follows:
%-------------------------------;
%  Equation: \mathbb{C} matrix  ;
%-------------------------------;
\begin{align}
  \mathbb{C} := \left[\begin{array}{cccc} 
      \mathbb{C}_1 & \mathbb{C}_2 & \cdots & \mathbb{C}_{S}
      \end{array} \right] 
\end{align}
The size of $\mathbb{C}$ is $N_C \times \left(\sum_{i=1}^{S} 
3 \eta_i N_i\right)$. The following augmented signed Boolean 
matrices will be useful in taking into account the effect of 
interface forces:
%-----------------------------------;
%  Equation: \mathbb{B}_i matrices  ;
%-----------------------------------;
\begin{align}
  \mathbb{B}_i^{\mathrm{T}} := \left[\begin{array}{c|c|c|c}
      -\frac{1}{\eta_i} \boldsymbol{C}_{i} \quad \mathcal{O}_i 
      \quad \mathcal{O}_i &
      -\frac{2}{\eta_i} \boldsymbol{C}_{i} \quad \mathcal{O}_i 
      \quad \mathcal{O}_i &
      \cdots \; \cdots \; \cdots & 
      -\frac{\eta_i}{\eta_i} \boldsymbol{C}_i \quad 
      \mathcal{O}_i \quad \mathcal{O}_i
    \end{array}\right]
\end{align}
The corresponding signed Boolean matrix for the entire 
system can be written as follows: 
%-------------------------------;
%  Equation: \mathbb{B} matrix  ;
%-------------------------------;
\begin{align}
  \mathbb{B} := \left[\begin{array}{c}
      \mathbb{B}_{1} \\
      \mathbb{B}_{2} \\
      \vdots \\
      \mathbb{B}_{S} 
    \end{array}\right]
\end{align}
We shall define the following augmented 
matrices for each subdomain:
%----------------------------------;
%  Equation: Matrices L_i and R_i  ;
%----------------------------------;
\begin{align}
    \label{Eqn:Monolithic_L_i_R_i}
    \mathbb{L}_{i} := \left[\begin{array}{ccc} 
        \boldsymbol{M}_{i} & \boldsymbol{O}_i & \boldsymbol{K}_{i} \\
        -\gamma_{i} \Delta t_{i} \boldsymbol{I}_i & \boldsymbol{I}_i & 
        \boldsymbol{O}_i \\
        -\beta_{i} \Delta t_{i}^{2} \boldsymbol{I}_i & 
        \boldsymbol{O}_i & \boldsymbol{I}_i 
      \end{array} \right] \quad
    \mathbb{R}_{i} := \left[\begin{array}{ccc}
        \boldsymbol{O}_i & \boldsymbol{O}_i & \boldsymbol{O}_i \\
         \left( 1 - \gamma_{i} \right)\Delta t_{i} \boldsymbol{I}_i & 
        \boldsymbol{I}_i & \boldsymbol{O}_i \\
        \left(\frac{1}{2} - \beta_{i} \right)\Delta t_{i}^{2} \boldsymbol{I}_i & 
        \Delta t_i \boldsymbol{I}_i & \boldsymbol{I}_i \end{array} \right]
\end{align}
where $\boldsymbol{O}_i$ denotes a matrix containing 
zeros of size $N_i \times N_i$, and $\boldsymbol{I}_i$ 
is the identity matrix of size $N_i \times N_i$.

%========================================;
%  Subsection: Multi-time-step coupling  ;
%========================================;
\subsection{Multi-time-step coupling} The proposed 
multi-time-step coupling method is developed based 
on the following assumptions:
\begin{enumerate}[(A)]
  \item Enforce the continuity of interface velocities 
    at system time-steps. 
  \item The corresponding Lagrange multipliers 
  (which will be interface reactions) are 
  calculated at system time-steps. (It should 
  be noted that the Lagrange multipliers are 
  unknowns,and will be a part of the solution.)
  \item The Lagrange multipliers are interpolated 
  linearly within system time-steps to approximate 
  their values at subdomain time-steps.
  \item The equilibrium equations for each subdomain 
    is enforced at its corresponding subdomain 
    time levels.
\end{enumerate}
with a requirement that the coupling method can handle 
arbitrary number of subdomains. 

Assumptions (B) and (C) take the following mathematical form:
%--------------------------------------------;
%  Equation: Linear interpolation of lambda  ;
%--------------------------------------------;
\begin{align}
  \label{Eqn:Lambda_Lin_Intrp}
  \boldsymbol{\lambda}^{\left( n+ \frac{j}{\eta_i} \right)} = 
  \left(1 - \frac{j}{\eta_i} \right) \boldsymbol{\lambda}^{\left( n \right)} 
  + \left(\frac{j}{\eta_i}\right) \boldsymbol{\lambda}^{ \left( n + 1 \right)}
\end{align}
where $\boldsymbol{\lambda}^{(n)}$ and $\boldsymbol{\lambda}^{\left( n+1 \right)}$ 
are Lagrange multipliers at system time levels.
Using equation \eqref{Eqn:Lambda_Lin_Intrp}, Assumption (D) 
takes the following form: 
%-------------------------------------;
%  Equation: Subdomain interpolation  ;
%-------------------------------------;
\begin{align}
  \label{Eqn:Subdomain_Intrp}
  \boldsymbol{M}_i \boldsymbol{a}_i^{\left( n+ \frac{j + 1}{\eta_i} \right)} + 
  \boldsymbol{K}_i \boldsymbol{d}_i ^ {\left( n+ \frac{j + 1}{\eta_i} \right)} -
  \frac{j+1}{\eta_i} \boldsymbol{C}_i^{\mathrm{T}} 
  \left( \boldsymbol{\lambda}^{\left( n + 1 \right)} - 
  \boldsymbol{\lambda}^{\left( n \right)}\right) = 
  \boldsymbol{f}_i ^ {\left( n+ \frac{j + 1}{\eta_i} \right)} + 
  \boldsymbol{C}_i^{\mathrm{T}} \boldsymbol{\lambda}^{\left( n \right)}
\end{align}
and the relations for the time-stepping schemes for 
the $i$-th subdomain take the following form: 
%--------------------------------------;
%  Equation: Subdomain Newmark family  ;
%--------------------------------------;
\begin{subequations}
  \begin{align}
    \label{Eqn:Newmark_d_subdomain}
    & \boldsymbol{d}_{i}^{\left( n+ \frac{j + 1}{\eta_i} \right)} = 
    \boldsymbol{d}_{i}^{\left( n+ \frac{j}{\eta_i} \right)} + \Delta t_i 
    \boldsymbol{v}_i^{\left( n+ \frac{j}{\eta_i} \right)} + 
    \frac{\left(\Delta t_i\right)^2}{2}
    \left(\left(1 - 2 \beta_i\right) \boldsymbol{a}_{i}^{\left( n+ \frac{j}{\eta_i} \right)} 
    + 2 \beta_i \boldsymbol{a}_{i}^{\left( n+ \frac{j + 1}{\eta_i} \right)}\right) \\
    \label{Eqn:Newmark_v_subdomain}
    & \boldsymbol{v}_{i}^{\left( n+ \frac{j + 1}{\eta_i} \right)} = 
    \boldsymbol{v}_{i}^{\left( n+ \frac{j}{\eta_i} \right)} + \Delta t_i 
    \left(\left(1 - \gamma_i\right) \boldsymbol{a}_{i}^{\left( n+ \frac{j}{\eta_i} \right)} 
    + \gamma_i\boldsymbol{a}_{i}^{\left( n+ \frac{j + 1}{\eta_i} \right)}\right)
  \end{align}
\end{subequations}
where $\beta_i$ and $\gamma_i$ are the Newmark parameters 
for the $i$-th subdomain.
Assumption (A) takes the following mathematical form: 
%--------------------------------------;
%  Equation: Continuity of velocities  ;
%--------------------------------------;
\begin{align}
  \sum_{i=1}^{S} \mathbb{C}_i \mathbb{X}_i^{\left( n+1 \right)} 
  = \boldsymbol{0}
\end{align}
Or, more compactly, 
%----------------------------------------------;
%  Equation: Compact continuity of velocities  ;
%----------------------------------------------;
\begin{align}
  \mathbb{C} \mathbb{X}^{\left( n+1 \right)} = \boldsymbol{0}
\end{align}

%===================================================================;
%  Subsubsection: Advance a subdomain over its subdomain time-step  ;
%===================================================================;
\subsubsection{Advance a subdomain over its subdomain time-step}
Using the above notation, the governing equations to advance 
the state of $i$-th subdomain over its time-step 
can be compactly written as follows:
%----------------------------------------------------;
%  Equation: Multi-time-step of v-continuity method  ;
%----------------------------------------------------;
\begin{align}
  \label{Eqn:Monolithic_Governing_compact_mat}
  \mathbb{L}_{i} \boldsymbol{X}_{i}^{ \left( n + \frac{j+1}{\eta_{i} }\right)} 
  - \left(\frac{j+1}{\eta_i}\right) \widetilde{\mathbb{C}}_{i}^{\mathrm{T}} 
  \left( \boldsymbol{\lambda}^{\left( n+1 \right)} - \boldsymbol{\lambda}^{\left( n \right)}\right)
  = \mathbb{P}_{i}^{\left( n+\frac{j+1}{\eta_{i}}\right)} 
  + \widetilde{\mathbb{C}}_{i}^{\mathrm{T}} 
  \boldsymbol{\lambda}^{\left( n \right)} 
  + \mathbb{R}_{i}\boldsymbol{X}_{i}^{ \left( n + \frac{j}{\eta_{i}} \right)}
\end{align}
where the following notation has been employed:
%--------------------------;
%  Equation: Big C matrix P vector ;
%--------------------------;
\begin{align}
  \widetilde{\mathbb{C}}_{i} := \left[\begin{array}{ccc}
      \boldsymbol{C}_{i} & \mathcal{O}_i & \mathcal{O}_i
    \end{array} \right] \quad
  \label{Eqn:Force}
  \mathbb{P}_{i}^{\left( n + \frac{j}{\eta_i} \right)} := 
  \left[ \begin{array}{c}
      \boldsymbol{f}_{i}^{\left( n + \frac{j}{\eta_i} \right)} \\
      \boldsymbol{0} \\
      \boldsymbol{0} \end{array} \right] 
\end{align}

%==============================================================;
%  Subsubsection: Advance a subdomain over a system time-step  ;
%==============================================================;
\subsubsection{Advance a subdomain over a system time-step}
The governing equations to advance a subdomain over a 
system time-step can be compactly written as follows: 
%---------------------------------------------------------;
%  Equation: Advance a subdomain over a system time-step  ; 
%---------------------------------------------------------;
\begin{align}
  \mathbb{Q}_i \mathbb{X}_i^{\left( n+1 \right)} + \mathbb{B}_i 
  \left(\boldsymbol{\lambda}^{\left( n+1 \right)} - \boldsymbol{\lambda}^{\left( n \right)}\right) 
  = \mathbb{F}_i^{\left( n+1 \right)}
\end{align}
where the matrix $\mathbb{Q}_i$ is defined as follows:
%-----------------------------------;
%  Equation: \mathbb{Q}_i matrices  ;
%-----------------------------------;
\begin{align}
  \mathbb{Q}_i := \left[ 
    \begin{array}{c c c c}
      \mathbb{L}_i &  &  &  \\
      -\mathbb{R}_i & \mathbb{L}_i &  &  \\
      &  \ddots &  \ddots & \\
      &  & -\mathbb{R}_i & \mathbb{L}_i \\  
    \end{array} \right]
\end{align}

%=================================================================;
%  Subsubsection: Advance all subdomains over a system time-step  ;
%=================================================================;
\subsubsection{Advance all subdomains over a system time-step}
We now write the governing equations to advance all the 
subdomains from (system) time level $t_n$ to $t_{n+1}$ (i.e., 
advance all subdomains by a system time-step) in a compact 
form. The mathematical statement takes the following form: 
Find $\mathbb{X}^{(n+1)}$ and $\boldsymbol{\lambda}^{(n+1)}$ 
by solving the following system of linear equations: 
%-----------------------------------------------;
%  Equation: Global system of linear equations  ;
%-----------------------------------------------;
\begin{align}
  \left[\begin{array}{cc}
      \mathbb{A} & \mathbb{B} \\
      \mathbb{C} & \mathbb{O}
  \end{array} \right]
  \left[\begin{array}{c}
  \mathbb{X}^{(n+1)} \\
  \boldsymbol{\lambda}^{(n+1)} - \boldsymbol{\lambda}^{(n)}
  \end{array}\right]
  =  \left[\begin{array}{c}
  \mathbb{F}^{(n+1)} \\
  \boldsymbol{0}
  \end{array}\right]
\end{align}
where the matrix $\mathbb{A}$ is defined as follows:
%--------------------------------;
%  Equation: \mathcal{A} matrix  ;
%--------------------------------;
\begin{align}
  \label{Eqn:Monolithic_global_A}
  \mathbb{A} := \left[ \begin{array}{cccc}
      \mathbb{Q}_1 &  &  &  \\
      %---------------------------------------------------	 
      & \mathbb{Q}_2 &  & \\
      %---------------------------------------------------	 
      & & \ddots &  \\
      %---------------------------------------------------
      & & & \mathbb{Q}_S 
    \end{array} \right]
\end{align}
and the following notation is employed: 
%----------------------------------------;
%  Equation: \mathcal{F}^{(n+1)} vector  ;
%----------------------------------------;
\begin{align}
  \mathbb{F}^{(n+1)} := 
  \left[ \begin{array}{c}
      \mathbb{F}_{1}^{(n+1)} \\
      \mathbb{F}_{2}^{(n+1)} \\
      \vdots \\
      \mathbb{F}_{S}^{(n+1)}
    \end{array} \right] \quad
  \boldsymbol{\mathbb{F}}_i^{(n+1)} := \left[ \begin{array}{c}
      \mathbb{P}_{i}^{\left(n+\frac{1}{\eta_i}\right)} +
      \widetilde{\mathbb{C}}_{i}^{\mathrm{T}}\boldsymbol{\lambda}^{(n)} 
      + \mathbb{R}_{i} \boldsymbol{X}_{i}^{(n)} \\ 
      \mathbb{P}_{i}^{\left(n+\frac{2}{\eta_i}\right)} + 
      \widetilde{\mathbb{C}}_{i}^{\mathrm{T}}\boldsymbol{\lambda}^{(n)} \\ 
      \vdots \\
      \mathbb{P}_{i}^{(n+1)} + \widetilde{\mathbb{C}}_{i}^{\mathrm{T}}
      \boldsymbol{\lambda^{(n)}} 
    \end{array} \right]
\end{align}
   
%=====================================;
%  Subsection: Comments on PH method  ;
%=====================================;
\subsection{Comments on the derivation of the PH method 
in Reference \cite{Prakash_Hjelmstad_IJNME_2004_v61_p2183}}
One main assumption in deriving the PH method is 
that the acceleration, velocity and displacement 
\emph{all} vary linearly with time within a system 
time-step. It should be emphasized that such an 
assumption is \emph{not self-consistent}. Moreover, 
this assumption need not be 
consistent with the underlying time stepping scheme. To 
wit, the assumption made in deriving the PH method takes 
the following mathematical form: 
%---------------------------------------;
%  Equation: Assumption on a, v, and d  ;
%---------------------------------------;
\begin{subequations}
  \begin{align}
    \label{Eqn:Monolithic_PH_acc} 
    &\boldsymbol{a}_i^{\left( n + \frac{j}{\eta_i} \right)} = \left(1 - \frac{j}{\eta_i}\right) 
    \boldsymbol{a}_i^{\left( n \right)} + \frac{j}{\eta_i} \boldsymbol{a}_i^{\left( n + 1 \right)} \\
    &\boldsymbol{v}_i^{\left( n + \frac{j}{\eta_i} \right)} = \left(1 - \frac{j}{\eta_i}\right) 
    \boldsymbol{v}_i^{\left( n \right)} + \frac{j}{\eta_i} \boldsymbol{v}_i^{\left( n + 1 \right)} \\
    \label{Eqn:Monolithic_PH_disp} 
    &\boldsymbol{d}_i^{\left( n + \frac{j}{\eta_i} \right)} = \left(1 - \frac{j}{\eta_i}\right) 
    \boldsymbol{d}_i^{\left( n \right)} + \frac{j}{\eta_i} \boldsymbol{d}_i^{\left( n + 1 \right)}
  \end{align}
\end{subequations}
Let us consider equation \eqref{Eqn:Monolithic_PH_acc}, 
which can be interpreted as follows: 
%-----------------------------------;
%  Equation: Linear variation of a  ;
%-----------------------------------;
\begin{align}
  \label{Eqn:Monolithic_PH_other_acc}
  \boldsymbol{a}_i(t) = \boldsymbol{a}_i^{\left( n \right)} + 
  \frac{\left( t - t_n \right)}{\left( t_{n+1} - t_n \right)} \left(
  \boldsymbol{a}_i^{\left( n + 1 \right)} - \boldsymbol{a}_i^{\left( n \right)}
  \right) \quad t_n \leq t \leq t_{n+1}
\end{align}
If the acceleration varies linearly with the time, the velocity 
should vary quadratically with the time, and the displacement 
should vary cubic with the time. Hence, equations 
\eqref{Eqn:Monolithic_PH_acc}--\eqref{Eqn:Monolithic_PH_disp} 
are \emph{not} inherently consistent.

In addition, this assumption need not be consistent with the 
underlying time stepping scheme, which is typically derived 
by assuming an ansatz functional form for the variation of the 
acceleration, velocity or displacement with respect to the time. 
For example, Newmark average acceleration scheme $(\gamma = 1/2, 
\beta = 1/4)$ is constructed by assuming that the acceleration 
is constant within a time-step \cite{Wood}. 
The assumption made in deriving the PH method that the acceleration 
varies linearly with time within a system time step (i.e., equation 
\eqref{Eqn:Monolithic_PH_acc} or \eqref{Eqn:Monolithic_PH_other_acc}) 
will not be consistent if, say, one employs the Newmark average 
acceleration scheme under the multi-time-step coupling method. 
More importantly, as shown in the previous section, such 
a mathematically \emph{in}consistent assumption is not 
warranted to develop a multi-time-step coupling method. 
Also, the multi-time-step coupling method as presented in 
Reference \cite{Prakash_Hjelmstad_IJNME_2004_v61_p2183} 
is restricted to two subdomains. There is no restriction 
on the number of subdomains in the proposed multi-time-step 
coupling method. 

%-----------------------------------------------------;
%  Remark: Remark that PH method is for 2 subdomains  ;
%-----------------------------------------------------;
\begin{remark}
  As mentioned earlier, the PH method (as presented in 
  Reference \cite{Prakash_Hjelmstad_IJNME_2004_v61_p2183}) 
  can handle only two subdomains. Preference is given 
  to the subdomain that has the coarsest time-step.
  For example, in the final form of the PH method 
  (see \cite[equation 43]{Prakash_Hjelmstad_IJNME_2004_v61_p2183}), 
  the forcing function to advance subdomain $B$ 
  uses $\mathbf{S}_i$, which is based on the 
  quantities of subdomain $A$. But the forcing 
  function to advance subdomain $A$ does not 
  employ any quantities of subdomain $B$.
  Recently, a tree-based approach has been proposed 
  in Reference \cite{Prakash_PhDThesis_UIUC} that 
  combines two subdomains at a time to solve multiple 
  subdomains, which will be computationally intensive.
  In the case of two subdomains (i.e., $S = 2$), 
  the proposed coupling method will be same as 
  the PH method if the applied external forces 
  on the subdomain with the coarse time-step 
  is affine with respect to time. The proposed 
  coupling method, however, can handle multiple 
  subdomains, and does not give preference to 
  any subdomain. It should be emphasized that 
  if one wants to implement in a recursive 
  manner using a tree-based approach, the 
  proposed method is amenable. 
\end{remark}

%
%************************************************;
%                                                ;
%  NAME                                          ;
%    S5_Monolithic_Stability.tex                 ;
%                                                ;
%  WRITTEN BY                                    ;
%    Kalyana Babu Nakshatrala                    ;
%                                                ;
%************************************************;
\section{A THEORETICAL ANALYSIS OF THE PROPOSED COUPLING METHOD}
\label{Sec:S5_Monolithic_Stability}
%==========================================================;
%  Subsection: Stability analysis using the energy method  ;
%==========================================================;
\subsection{Stability analysis using the energy method}
We shall employ the energy method to show the stability of 
the proposed multi-time-step coupling method. The energy 
method is a popular strategy employed in Mathematical 
Analysis to derive estimates and to perform stability 
analysis. The method is widely employed in the theory 
of partial differential equations \cite{Evans_PDE}, 
and numerical analysis \cite{Richtmyer_Morton,Hughes}.
The basic idea behind the energy method is to choose an 
appropriate norm (which is referred to as the energy norm) 
and show that the solution is bounded under this norm. It 
should be noted that the energy norm may not correspond 
to the physical energy. 

We shall now introduce the notation that is needed to 
apply the energy method. The jump and average operators 
over the system time-step are, respectively, denoted by 
$\llbracket\cdot\rrbracket$ and $\llangle\cdot\rrangle$. 
That is, 
%-----------------------------------------------;
%  Equation: System jump and average operators  ;
%-----------------------------------------------;
\begin{subequations}
  \label{Eqn:Monolithic_system_operators}
  \begin{align}
    \label{Eqn:Monolithic_system_jump}
    \left\llbracket \boldsymbol{x}^{(n)} \right\rrbracket 
    &:= \boldsymbol{x}^{(n+1)} - \boldsymbol{x}^{(n)} \\
    \label{Eqn:Monolithic_system_average}
    \left\llangle \boldsymbol{x}^{(n)} \right\rrangle &:= 
    \frac{1}{2} \left(\boldsymbol{x}^{(n)} + \boldsymbol{x}^{(n+1)}
    \right)  
  \end{align}
\end{subequations}
The jump and average operators over the subdomain time-step 
of the $i$-th subdomain are, respectively, denoted by $\left
[\cdot\right]_i$ and $\langle\cdot\rangle_i$. That is, 
%--------------------------------------------------;
%  Equation: Subdomain jump and average operators  ;
%--------------------------------------------------;
\begin{subequations}
  \label{Eqn:Monolithic_subdomain_operators}
  \begin{align}
    \label{Eqn:Monolithic_subdomain_jump}
    \left[\boldsymbol{x}^{\left( n + \frac{j}{\eta_i)} \right)} \right]_i &:= 
    \boldsymbol{x}^{\left( n+ \frac{j+1}{\eta_i} \right)} - 
    \boldsymbol{x}^{\left( n + \frac{j}{\eta_i} \right)} \\
    \label{Eqn:Monolithic_subdomain_average}
    \left\langle \boldsymbol{x}^{\left( n + \frac{j}{\eta_i} \right)} \right\rangle_i 
    &:= \frac{1}{2} \left(\boldsymbol{x}^{\left( n + \frac{j}{\eta_i} \right)} 
    + \boldsymbol{x}^{\left( n + \frac{j+1}{\eta_i} \right)}\right)  
  \end{align}
\end{subequations}
It is easy to show that, for any symmetric matrix $\boldsymbol{S}$, 
the jump and average operators obey the following relationship:
%-------------------------------------------;
%  Equation: Jump and average relationship  ;
%-------------------------------------------;
\begin{align}
  \left\llbracket \boldsymbol{x}^{(n)}\right\rrbracket^{\mathrm{T}} 
  \boldsymbol{S} \left\llangle \boldsymbol{x}^{(n)}\right\rrangle = 
  \frac{1}{2} \left\llbracket {\boldsymbol{x}^{(n)}}^{\mathrm{T}} 
  \boldsymbol{S} \boldsymbol{x}^{(n)}\right\rrbracket
\end{align}
A similar relation holds for $\left[\cdot\right]_i$ and 
$\langle \cdot \rangle_i$. It is important to note that 
the jump and average operators are linear. That is, for 
any $\alpha, \beta \in \mathbb{R}$ we have 
%-----------------------------------------------------;
%  Equation: Subdomain jump and average relationship  ;
%-----------------------------------------------------;
\begin{subequations}
  \begin{align}
    &\llbracket \alpha \boldsymbol{x} + \beta \boldsymbol{y} 
    \rrbracket = \alpha \llbracket \boldsymbol{x} \rrbracket 
    + \beta \llbracket \boldsymbol{y} \rrbracket \\
    &\llangle \alpha \boldsymbol{x} + \beta \boldsymbol{y} 
    \rrangle = \alpha \llangle \boldsymbol{x} \rrangle + 
    \beta \llangle \boldsymbol{y} \rrangle
  \end{align}
\end{subequations}
We shall call a sequence of vectors $\left\{\boldsymbol{x}^{(n)}
\right\}_{n = 0}^{\infty}$ to be bounded $\forall n$ if there 
exists a real number $0 < M  <+\infty$ such that
%---------------------------------;
%  Equation: Sequence is bounded  ;
%---------------------------------;
\begin{align}
  \left\|\boldsymbol{x}^{(n)}\right\| < M \quad \forall n
\end{align}

For convenience, we shall use $\boldsymbol{A}_i$ to denote 
%--------------------------------------;
%  Equation: Definition of A_i matrix  ;
%--------------------------------------;
\begin{align}
  \boldsymbol{A}_i := \boldsymbol{M}_i + \left(\Delta t_i\right)^2 
  \left(\beta_i - \frac{\gamma_i}{2}\right) \boldsymbol{K}_i
\end{align}
The critical time-step $\Delta t_i^{\mathrm{crit}} \geq 0$ in 
the $i$-th subdomain is the maximum time-step for which the 
matrix $\boldsymbol{A}_i$ is positive definite. 
It should be emphasized that $\Delta t_i^{\mathrm{crit}}$ is 
the critical subdomain time-step assuming that there is no 
coupling between subdomains, which can be easily calculated. 
Let $\omega_i^{\mathrm{max}}$ be the maximum eigenvalue 
of the generalized eigenvalue problem for the $i$-th 
subdomain. That is, 
%--------------------------------------------;
%  Equation: Generalized eigenvalue problem  ;
%--------------------------------------------;
\begin{align}
  \omega_i^2 \boldsymbol{M}_i \boldsymbol{x}_i = 
  \boldsymbol{K}_i \boldsymbol{x}_i
\end{align}
where $\boldsymbol{x}_i$ is the corresponding eigenvector. 
Then the critical time-step for the $i$-th subdomain can 
be written as follows:
%------------------------------------------;
%  Equation: Subdomain critical time-step  ;
%------------------------------------------;
\begin{align}
  \Delta t_i^{\mathrm{crit}} = \left\{
  \begin{array}{ll}
    +\infty & \mbox{for} \; 2 \beta_i \geq \gamma_i \geq 1/2 \\
    \frac{1}{\omega_i^{\mathrm{max}} \sqrt{\gamma_i/2 - \beta_i}} 
    & \mbox{for} \; \gamma_i \geq 1/2 \; \mbox{and} \; \beta_i < \gamma_i/2
  \end{array}
  \right.
\end{align}
We shall choose the subdomain time-step to be smaller than 
the corresponding critical time-step for the subdomain. 
That is, 
%-------------------------------------------;
%  Equation: Choice of subdomain time-step  ;
%-------------------------------------------;
\begin{align}
  \Delta t_i < \Delta t_i^{\mathrm{crit}} 
\end{align}
A detailed discussion on the critical time-steps for 
Newmark family of time integrators can be found in 
references \cite{Hughes,Wood}.
For Newmark family of time stepping schemes, 
it is easy to check the following identities:
%-----------------------------------------------------;
%  Equation: Newmark family of time stepping schemes  ;
%-----------------------------------------------------;
\begin{subequations}
\begin{align}
\label{Eqn:Monolithic_jump_v_subdomain}
  &\left[\boldsymbol{v}_i^{\left(n + \frac{j}{\eta_i}\right)}\right]_{i} = 
  \Delta t_{i} \left(\left\langle \boldsymbol{a}_i^{\left(n + \frac{j}{\eta_i} 
    \right)}\right\rangle_{i} + \left(\gamma_{i} - \frac{1}{2}\right)
  \left[\boldsymbol{a}_i^{\left(n + \frac{j}{\eta_i}\right)}\right]_{i} \right) \\
  \label{Eqn:Monolithic_jump_d_subdomain}
  &\left[\boldsymbol{d}_i^{\left(n + \frac{j}{\eta_i}\right)}\right]_{i} = 
  \Delta t_{i} \left\langle\boldsymbol{v}_{i}^{\left(n + \frac{j}{\eta_{i}}
    \right)}\right\rangle_i + \Delta t_{i}^{2}  \left(\beta_{i} - 
  \frac{\gamma_i}{2}\right)\left[\boldsymbol{a}_i^{\left(n + \frac{j}{\eta_i}
      \right)}\right]_{i}
\end{align}
\end{subequations}

%-------------------------------------;
%  Theorem: Newmark scheme stability  ;
%-------------------------------------;
\begin{theorem}
  If $\Delta t_i < \Delta t_i^{\mathrm{crit}}$ in all subdomains, 
  then the velocity and acceleration vectors for all subdomains 
  are bounded $\forall n$ under the proposed multi-time-step 
  coupling method.
\end{theorem}
\begin{proof}
  Using the governing equation for the $i$-th subdomain, and 
  the linear interpolation of the Lagrange multiplier, we obtain 
  the following equation:
  %----------------------------;
  %  Equation: i-th subdomain  ;
  %----------------------------;
  \begin{align}
    \boldsymbol{M}_{i} \left[\boldsymbol{a}_i^{\left(n + 
        \frac{j}{\eta_i} \right)}\right]_{i} + \boldsymbol{K}_{i} 
    \left[\boldsymbol{d}_i^{\left(n + \frac{j}{\eta_i}\right)}\right]_{i} 
    = \frac{1}{\eta_{i}}\boldsymbol{C}_{i}^{\mathrm{T}}\left\llbracket 
    \boldsymbol{\lambda}^{(n)}\right\rrbracket
  \end{align}
  Using equation \eqref{Eqn:Monolithic_jump_d_subdomain}, 
  the above equation can be rewritten as follows:
  \begin{align}
    \boldsymbol{A}_{i} \left[\boldsymbol{a}_i^{\left(n + 
        \frac{j}{\eta_i} \right)}\right]_{i} 
    + \Delta t_i \boldsymbol{K}_{i} \left\langle \boldsymbol{v}_i^{
      \left(n + \frac{j}{\eta_i}\right)}\right\rangle_{i} 
    = \frac{1}{\eta_{i}}\boldsymbol{C}_{i}^{\mathrm{T}}\left
    \llbracket \boldsymbol{\lambda}^{(n)}\right\rrbracket
  \end{align}
  Premultiplying both sides by $\left[\boldsymbol{v}_i^{\left(n + \frac{j}{\eta_i}\right)}\right]_{i}$ 
  and using equation \eqref{Eqn:Monolithic_jump_v_subdomain}, we obtain the following equation:
\begin{align}
  \Delta t_i \left\langle \boldsymbol{a}_i^{\left( n + \frac{j}{\eta_i} \right)}
  \right\rangle_{i}^{\mathrm{T}} \boldsymbol{A}_i \left[\boldsymbol{a}_i^{
      \left(n + \frac{j}{\eta_i} \right)}\right]_{i} 
  &+ \Delta t_i \left(\gamma_i - \frac{1}{2}\right) \left[ \boldsymbol{a}_i^{
      \left( n + \frac{j}{\eta_i} \right)}
    \right]_{i}^{\mathrm{T}} \boldsymbol{A}_i \left[\boldsymbol{a}_i^{
      \left(n + \frac{j}{\eta_i} \right)}\right]_{i} \\ \nonumber
  &+ 
  \Delta t_{i} \left[\boldsymbol{v}_i^{\left(n + \frac{j}{\eta_i}\right)}
    \right]_{i}^{\mathrm{T}}\boldsymbol{K}_{i} \left \langle 
  \boldsymbol{v}_{i}^{\left(n + \frac{j}{\eta_i}\right)}\right\rangle_{i} 
  = \frac{1}{\eta_{i}}\left\llbracket \boldsymbol{\lambda}^{(n)}
  \right \rrbracket^{\mathrm{T}}\boldsymbol{C}_{i} \left[
    \boldsymbol{v}_i^{\left(n + \frac{j}{\eta_i}\right)}\right]_{i}
\end{align}
Since $\gamma \geq \frac{1}{2}$ and $\boldsymbol{A}_i$ is 
positive definite (as $\Delta t_i < \Delta t_i^{\mathrm{crit}}$), 
we can conclude that 
\begin{align}
  \Delta t_i
  \left\langle \boldsymbol{a}_i^{\left( n + \frac{j}{\eta_i} \right)}
    \right\rangle_{i}^{\mathrm{T}} \boldsymbol{A}_i 
  \left[\boldsymbol{a}_i^{\left(n + \frac{j}{\eta_i} \right)}\right]_{i} + 
  \Delta t_{i} \left[\boldsymbol{v}_i^{\left(n + \frac{j}{\eta_i}\right)}
    \right]_{i}^{\mathrm{T}}\boldsymbol{K}_{i} \left\langle
  \boldsymbol{v}_{i}^{\left(n + \frac{j}{\eta_i}\right)}\right\rangle_{i} 
  \leq \frac{1}{\eta_{i}}\left\llbracket \boldsymbol{\lambda}^{(n)}
  \right\rrbracket^{\mathrm{T}}\boldsymbol{C}_{i}\left[
    \boldsymbol{v}_i^{\left(n + \frac{j}{\eta_i}\right)}\right]_{i}
\end{align}
Noting that $\Delta t = \eta_i \Delta t_i$, and the matrices 
$\boldsymbol{A}_i$ and $\boldsymbol{K}_i$ are symmetric, we 
obtain the following:
\begin{align}
  \frac{\Delta t}{2} \left[{\boldsymbol{a}_i^{\left(n + \frac{j}{\eta_i}\right)}}^{\mathrm{T}} 
    \boldsymbol{A}_{i} \boldsymbol{a}_{i}^{\left(n + \frac{j}{\eta_i}\right)}\right]_{i} 
  + \frac{\Delta t}{2} \left[{\boldsymbol{v}_i^{\left(n + \frac{j}{\eta_i}\right)}}^{\mathrm{T}} 
    \boldsymbol{K}_{i} \boldsymbol{v}_{i}^{\left(n + \frac{j}{\eta_i}\right)}\right]_{i} 
  \leq \left\llbracket \boldsymbol{\lambda}^{(n)}\right\rrbracket^{\mathrm{T}}
  \boldsymbol{C}_{i}\left[\boldsymbol{v}_i^{\left(n + \frac{j}{\eta_i}\right)}\right]_{i}
\end{align}
By summing over $j \; (j = 1, \cdots, \eta_i)$ we 
obtain the following: 
\begin{align}
  \frac{\Delta t}{2} \left \llbracket 
       {\boldsymbol{a}_{i}^{\left( n \right)}}^{\mathrm{T}} 
       \boldsymbol{A}_{i} \boldsymbol{a}_{i}^{\left( n \right)} 
       + {\boldsymbol{v}_{i}^{\left( n \right)}}^{\mathrm{T}} 
       \boldsymbol{K}_{i} \boldsymbol{v}_{i}^{\left( n \right)} 
       \right \rrbracket \leq 
              {\left \llbracket \lambda^{\left( n \right)} \right \rrbracket}^{\mathrm{T}} 
              \sum_{i = 1}^{S} \boldsymbol{C}_{i} \left \llbracket \boldsymbol{v}_{i}^{\left( n \right)} 
              \right \rrbracket 
\end{align}
Summing over $i \; (i = 1, \cdots, S)$ and using the 
continuity of velocities at system time-steps, we 
obtain the following inequality:
\begin{align}
  \sum_{i = 1}^{S} \left \llbracket 
      {\boldsymbol{a}_{i}^{\left( n \right)}}^{\mathrm{T}} 
      \boldsymbol{A}_{i} \boldsymbol{a}_{i}^{\left( n \right)} 
      + {\boldsymbol{v}_{i}^{\left( n \right)}}^{\mathrm{T}} 
      \boldsymbol{K}_{i} \boldsymbol{v}_{i}^{\left( n \right)} 
      \right \rrbracket \leq 0
\end{align}
This further implies that 
\begin{align}
  \sum_{i = 1}^{S}
  \left({\boldsymbol{a}_{i}^{(n+1)}}^{\mathrm{T}} \boldsymbol{A}_{i} 
  \boldsymbol{a}_{i}^{(n+1)} + {\boldsymbol{v}_{i}^{(n+1)}}^{\mathrm{T}} 
  \boldsymbol{K}_{i}\boldsymbol{v}_{i}^{(n+1)} \right) 
  &\leq \sum_{i = 1}^{S}\left( {\boldsymbol{a}_{i}^{(n)}}^{\mathrm{T}} 
  \boldsymbol{A}_{i} \boldsymbol{a}_{i}^{(n)} + 
  {\boldsymbol{v}_{i}^{(n)}}^{\mathrm{T}} \boldsymbol{K}_{i}\boldsymbol{v}_{i}^{(n)} 
  \right) \nonumber \\
&\leq ... \leq \sum_{i = 1}^{S} 
\left( {\boldsymbol{a}_{i}^{(0)}}^{\mathrm{T}} \boldsymbol{A}_{i} 
\boldsymbol{a}_{i}^{(0)} + {\boldsymbol{v}_{i}^{(0)}}^{\mathrm{T}} 
\boldsymbol{K}_{i}\boldsymbol{v}_{i}^{(0)} \right)
\end{align}
Since the matrices $\boldsymbol{A}_i \; (i = 1, \cdots, S)$ are 
positive definite, the matrices $\boldsymbol{K}_i \; (i = 1, \cdots, S)$ 
are positive semidefinite, and the vectors $\boldsymbol{v}_i^{(0)}$ 
and $\boldsymbol{a}_i^{(0)}$ are bounded, one can conclude that 
the vectors $\boldsymbol{a}_i^{(n)}$ and $\boldsymbol{v}_i^{(n)}$ 
are bounded $\forall n$ and for all subdomains.
\end{proof}

%-----------------------------------------;
%  Remark: Remark about boundedness of v  ;
%-----------------------------------------;
\begin{remark}
  Strictly speaking, in the above proof, one can only conclude 
  that $\boldsymbol{v}_i^{(n)}$ are bounded except for vectors 
  that have a component in the null space of $\boldsymbol{K}_i$. 
  This is the case even for the undecomposed case (i.e., no 
  coupling) under the energy method.
\end{remark}

%==========================================;
%  Subsection: Influence of perturbations  ;
%==========================================;
\subsection{Influence of perturbations under the proposed coupling method}
We shall perform the analysis assuming no subcycling. We 
will follow a procedure similar to the one presented in 
\cite{Hairer_Wanner} for differential-algebraic equations. 
We shall begin with the original system of equations over 
a (system) time-step:
%-----------------------------;
%  Equation: Original system  ;
%-----------------------------;
\begin{subequations}
  \begin{align}
    \label{Eqn:Original_Subdomain_GE}
    &\boldsymbol{M}_i \boldsymbol{a}_i^{\left(n+1\right)} + 
    \boldsymbol{K}_i \boldsymbol{d}_i ^ {\left(n+1\right)} =
    \boldsymbol{f}_i ^ {\left(n+1\right)} + \boldsymbol{C}_i^{\mathrm{T}} 
    \boldsymbol{\lambda}^{\left(n+1\right)} \\
    \label{Eqn:Original_Newmark_v_subdomain}
    & \boldsymbol{v}_{i}^{\left(n+1\right)} = 
    \boldsymbol{v}_{i}^{\left(n\right)} + \Delta t 
    \left(\left(1 - \gamma_i\right) \boldsymbol{a}_{i}^{\left(n\right)} 
    + \gamma_i\boldsymbol{a}_{i}^{\left(n+1\right)}\right) \\
    \label{Eqn:Original_Newmark_d_subdomain}
    & \boldsymbol{d}_{i}^{\left(n+1\right)} = 
    \boldsymbol{d}_{i}^{\left(n \right)} + \Delta t  
    \boldsymbol{v}_i^{\left(n\right)} + 
    \frac{\Delta t^2}{2}\left(\left(1 - 2 \beta_i\right) 
    \boldsymbol{a}_{i}^{\left(n\right)} + 2 \beta_i 
    \boldsymbol{a}_{i}^{\left(n+1\right)}\right) \\
   \label{Eqn:Original_Constraint}
    &\sum_{i=1}^{S} \boldsymbol{C}_i \boldsymbol{v}_i^{(n+1)} = 
    \boldsymbol{0}
  \end{align}
\end{subequations}
Now consider the following perturbed system: 
%------------------------------;
%  Equation: Perturbed system  ;
%------------------------------;
\begin{subequations}
  \begin{align}
    \label{Eqn:Perturbed_Subdomain_GE}
    &\boldsymbol{M}_i \widehat{\boldsymbol{a}}_i^{\left(n+1\right)} + 
    \boldsymbol{K}_i \widehat{\boldsymbol{d}_i}^{\left(n+1\right)} =
    \boldsymbol{f}_i ^ {\left(n+1\right)} + \boldsymbol{C}_i^{\mathrm{T}} 
    \widehat{\boldsymbol{\lambda}}^{\left(n+1\right)} \\
    \label{Eqn:Perturbed_Newmark_v_subdomain}
    & \widehat{\boldsymbol{v}}_{i}^{\left(n+1\right)} = 
    \widehat{\boldsymbol{v}}_{i}^{\left(n\right)} + \Delta t 
    \left(\left(1 - \gamma_i\right) \widehat{\boldsymbol{a}}_{i}^{\left(n\right)} 
    + \gamma_i \widehat{\boldsymbol{a}}_{i}^{\left(n+1\right)}\right) + 
    \Delta t \boldsymbol{\varepsilon}_{v_i} \\
    \label{Eqn:Perturbed_Newmark_d_subdomain}
    & \widehat{\boldsymbol{d}}_{i}^{\left(n+1\right)} = 
    \widehat{\boldsymbol{d}}_{i}^{\left(n \right)} + \Delta t  
    \widehat{\boldsymbol{v}}_i^{\left(n\right)} + 
    \frac{\Delta t^2}{2} \left(\left(1 - 2 \beta_i\right) 
    \widehat{\boldsymbol{a}}_{i}^{\left(n\right)} 
    + 2 \beta_i \widehat{\boldsymbol{a}}_{i}^{\left(n+1\right)}\right) 
    + \Delta t^2 \boldsymbol{\varepsilon}_{d_i} \\
   \label{Eqn:Perturbed_Constraint}
    &\sum_{i=1}^{S} \boldsymbol{C}_i \widehat{\boldsymbol{v}}_i^{(n+1)} 
    = \boldsymbol{\varepsilon}_{\lambda}
  \end{align}
\end{subequations}
where $\boldsymbol{\varepsilon}_{v_i}$, 
$\boldsymbol{\varepsilon}_{d_i}$ and $\boldsymbol{\varepsilon}_{\lambda}$ 
are, respectively, the perturbations to the original system of equations 
\eqref{Eqn:Original_Subdomain_GE}--\eqref{Eqn:Original_Constraint}. 
The solution to this perturbed system of equations 
will be $\widehat{\boldsymbol{a}}_i^{(n+1)}$, 
$\widehat{\boldsymbol{v}}_i^{(n+1)}$, 
$\widehat{\boldsymbol{d}}_i^{(n+1)}$ and 
$\widehat{\boldsymbol{\lambda}}^{(n+1)}$. 
For convenience, we shall define the 
following quantities: 
%-----------------------------------;
%  Equation: Perturbed differences  ;
%-----------------------------------;
\begin{subequations}
  \begin{align}
    &\delta \boldsymbol{a}_i^{(n+1)} := 
    \widehat{\boldsymbol{a}}_i^{(n+1)} 
    - \boldsymbol{a}_i^{(n+1)} \\
    &\delta \boldsymbol{v}_i^{(n+1)} := 
    \widehat{\boldsymbol{v}}_i^{(n+1)} 
    - \boldsymbol{v}_i^{(n+1)} \\
    &\delta \boldsymbol{d}_i^{(n+1)} := 
    \widehat{\boldsymbol{d}}_i^{(n+1)} 
    - \boldsymbol{d}_i^{(n+1)} \\
    &\delta \boldsymbol{\lambda}^{(n+1)} := 
    \widehat{\boldsymbol{\lambda}}^{(n+1)} 
    - \boldsymbol{\lambda}^{(n+1)} 
  \end{align}
\end{subequations}
By subtracting equation \eqref{Eqn:Original_Subdomain_GE} 
from equation \eqref{Eqn:Perturbed_Subdomain_GE} we obtain 
the following:
%-------------------------------------------;
%  Equation: Difference subdomain equation  ;
%-------------------------------------------;
\begin{align}
  \label{Eqn:Difference_Subdomain_GE}
  \boldsymbol{M}_i \delta \boldsymbol{a}_i^{\left(n+1\right)} + 
  \boldsymbol{K}_i \delta \boldsymbol{d}_i^{\left(n+1\right)} = 
  \boldsymbol{C}_i^{\mathrm{T}} \delta 
  \boldsymbol{\lambda}^{\left(n+1\right)} 
\end{align}
Using equations \eqref{Eqn:Original_Newmark_d_subdomain} 
and \eqref{Eqn:Perturbed_Newmark_d_subdomain}, the above 
equation can be written as follows:
%-----------------------------------;
%  Equation: Intermediate equation  ;
%-----------------------------------;
\begin{align}
  \label{Eqn:Perturbed_Intermediate}
  \delta \boldsymbol{a}_i^{(n+1)} + \boldsymbol{B}_i^{-1} 
  \boldsymbol{K}_i \left(\delta \boldsymbol{d}_i^{(n)}
  + \Delta t \delta \boldsymbol{v}_i^{(n)} + \Delta t^2 
  (1/2 - \beta_i) \delta \boldsymbol{a}_i^{(n)}\right)
  = \boldsymbol{B}_i^{-1} \boldsymbol{C}_i^{\mathrm{T}} 
  \delta \boldsymbol{\lambda}^{(n+1)} - \Delta t^2 \boldsymbol{B}_i^{-1} 
  \boldsymbol{K}_i \boldsymbol{\varepsilon}_{d_i}
\end{align}
where the matrix $\boldsymbol{B}_i$ has been defined as follows:
%--------------------------------------;
%  Equation: Definition of A_i matrix  ;
%--------------------------------------;
\begin{align}
  \boldsymbol{B}_i := \boldsymbol{M}_i + 
  \beta_i \Delta t^2 \boldsymbol{K}_i
\end{align}
The operation $\boldsymbol{B}_i^{-1}$ in equation 
\eqref{Eqn:Perturbed_Intermediate} is justified 
as the matrix is positive definite and hence 
invertible. 
By multiplying both sides of equation 
\eqref{Eqn:Perturbed_Intermediate} by 
$\gamma_i \Delta t$ and using equations 
\eqref{Eqn:Original_Newmark_v_subdomain} 
and \eqref{Eqn:Perturbed_Newmark_v_subdomain}, 
one can arrive at the following equation: 
%--------------------------------------;
%  Equation: Perturbed Intermediate 2  ;
%--------------------------------------;
\begin{align}
  \delta \boldsymbol{v}_i^{(n+1)} - \delta \boldsymbol{v}_i^{(n)} 
  &- (1 - \gamma_i) \Delta t \delta \boldsymbol{a}_i^{(n)} 
  - \Delta t \boldsymbol{\varepsilon}_{v_i} + \gamma_i 
  \Delta t \boldsymbol{B}_i^{-1} \boldsymbol{K}_i \left(\delta 
  \boldsymbol{d}_i^{(n)}+ \Delta t \delta \boldsymbol{v}_i^{(n)} 
  + \Delta t^2 (1/2 - \beta_i) \delta \boldsymbol{a}_i^{(n)}
  \right) \nonumber \\
  &= \gamma_i \Delta t \boldsymbol{B}_i^{-1} 
  \boldsymbol{C}_i^{\mathrm{T}} \delta 
  \boldsymbol{\lambda}^{(n+1)} - \gamma_i 
  \Delta t^3 \boldsymbol{B}_i^{-1} 
  \boldsymbol{K}_i \boldsymbol{\varepsilon}_{d_i}
\end{align}
We shall assume that $\sum_{i=1}^{S} \boldsymbol{C}_i 
\delta \boldsymbol{v}_i^{(n)} =\boldsymbol{0}$. That 
is, the constraint is exactly satisfied at the $n$-th 
time level. 
Premultiplying both sides by $\boldsymbol{C}_i$, 
summing over $i$ (i.e., the number of subdomains), 
and using equations \eqref{Eqn:Original_Constraint} 
and \eqref{Eqn:Perturbed_Constraint}; one can arrive 
at the following equation:
%--------------------------------------;
%  Equation: Perturbed Intermediate 3  ;
%--------------------------------------;
\begin{align}
  \boldsymbol{\varepsilon}_{\lambda} - \Delta t \sum_{i=1}^{S} 
  (1 - \gamma_i) \boldsymbol{C}_i \delta \boldsymbol{a}_i^{(n)} 
  &- \Delta t \sum_{i=1}^{S} \boldsymbol{C}_i \boldsymbol{\varepsilon}_{v_i} 
  + \Delta t \sum_{i=1}^{S} \gamma_i \boldsymbol{C}_i \boldsymbol{B}_i^{-1} 
  \boldsymbol{K}_i \left(\delta \boldsymbol{d}_i^{(n)}
  + \Delta t \delta \boldsymbol{v}_i^{(n)} + \Delta t^2 
  (1/2 - \beta_i) \delta \boldsymbol{a}_i^{(n)}\right) \nonumber \\
  &= \Delta t \left(\sum_{i=1}^{S} \gamma_i \boldsymbol{C}_i 
  \boldsymbol{B}_i^{-1} \boldsymbol{C}_i^{\mathrm{T}} \right) 
  \delta \boldsymbol{\lambda}^{(n+1)} 
  - \Delta t^3 \sum_{i=1}^{S} \gamma_i \boldsymbol{C}_i 
  \boldsymbol{B}_i^{-1} \boldsymbol{K}_i 
  \boldsymbol{\varepsilon}_{d_i}
\end{align}
By taking norm on both sides and invoking triangle 
inequality, one can arrive at the following estimate 
for $\delta \boldsymbol{\lambda}^{(n+1)}$: 
%-----------------------------------------;
%  Equation: Estimate for \delta \lambda  ;
%-----------------------------------------;
\begin{align}
  \label{Eqn:Monolithic_lambda_estimate}
  \|\delta \boldsymbol{\lambda}^{(n+1)}\| \leq C_{\lambda} 
  \left(\frac{1}{\Delta t} \|\boldsymbol{\varepsilon}_{\lambda}\| 
  + \sum_{i=1}^{S} \left(
   \|\boldsymbol{\varepsilon}_{v_i}\|
  + \Delta t^2 \|\boldsymbol{\varepsilon}_{d_i}\|
  + \|\delta \boldsymbol{a}_{i}^{(n)}\|
  + \Delta t \|\delta \boldsymbol{v}_{i}^{(n)}\|
  + \|\delta \boldsymbol{d}_{i}^{(n)}\|
  \right) \right)
\end{align}
where $C_{\lambda}$ is a constant. 
Following a similar procedure for 
displacements, velocities, and accelerations 
we obtain the following: 
\begin{align}
% Estimate for \delta d
&\| \delta \boldsymbol{d}_i^{\left( n + 1 \right)} \| \leq
C_{d} \left( \| \delta \boldsymbol{d}_i^{(n)}\| + 
\Delta t \| \delta \boldsymbol{v}_i^{(n)}\| + 
\Delta t^{2} \| \boldsymbol{\varepsilon}_{d_i}\| + 
\Delta t \| \boldsymbol{\varepsilon}_{\lambda}\| + 
\sum_{i = 1}^{S} \left( \Delta t^2 \| \delta \boldsymbol{a}_i^{(n)}\| +
\Delta t \| \boldsymbol{\varepsilon}_{v_i}\|\right)\right) \\
% Estimate for \delta v
&\| \delta \boldsymbol{v}_i^{\left( n + 1 \right)} \| \leq
C_{v} \left( \| \delta \boldsymbol{v}_i^{(n)}\| + 
\| \boldsymbol{\varepsilon}_{\lambda}\| + 
\sum_{i = 1}^{S} \left( \Delta t \| \delta \boldsymbol{a}_i^{(n)}\| +
\Delta t \| \delta \boldsymbol{d}_i^{(n)}\| +
\Delta t^3 \| \boldsymbol{\varepsilon}_{d_i}\| + 
\Delta t \| \boldsymbol{\varepsilon}_{v_i}\|\right)\right) \\
% Estimate for \delta a
&\| \delta \boldsymbol{a}_i^{\left( n + 1 \right)} \| \leq
C_{a} \left( \frac{1}{\Delta t} \| \boldsymbol{\varepsilon}_{\lambda}\| +
\Delta t \| \delta \boldsymbol{v}_i^{(n)} \| + 
\sum_{i = 1}^{S} \left( \| \delta \boldsymbol{a}_i^{(n)}\| +
 \| \delta \boldsymbol{d}_i^{(n)}\| + 
 \Delta t^2 \| \boldsymbol{\varepsilon}_{d_i} \| + 
 \| \boldsymbol{\varepsilon}_{v_i} \| \right)\right)
\end{align}
where $C_{d}$, $C_{v}$ and $C_{a}$ are constants.
From the above estimate \eqref{Eqn:Monolithic_lambda_estimate}, 
one can see that a perturbation in the constraint, 
$\boldsymbol{\varepsilon}_{\lambda}$, leads to an 
amplification by $1/\Delta t$ in the Lagrange 
multiplier. On the other hand, the perturbations 
in the variables $\boldsymbol{d}_i$ and $\boldsymbol{v}_i$ 
lead to (at most) linear growth in the Lagrange multiplier. 
Clearly, the estimate for the proposed coupling method 
under \emph{no} subcycling follows the typical behavior 
of differential-algebraic equations. An extension of 
this study to include subcycling will require a more 
involved and careful analysis, and is beyond the scope 
of this paper. 

%==========================================================;
%  Subsection: On drifts in displacement and acceleration  ;
%==========================================================;
\subsection{On drifts in interface displacement 
and acceleration vectors}
\label{Subsec:Monolithic_bounds_on_drifts}
In a time continuous setting, enforcing the continuity 
of either displacements, velocities or accelerations 
are all mathematically equivalent. However, in a 
numerical setting this equivalence will not hold, 
and the numerical performance will depend on the 
type of the constraint that is being enforced.
As mentioned in the previous sections, we employ the 
continuity of velocities at the subdomain interface 
at every system time-step (which we referred to as 
the $\boldsymbol{v}$-continuity). This may lead to 
drift in the displacements and the accelerations 
along the subdomain interface. We now derive bounds 
on these drifts, which could serve as a valuable 
check for the correctness of a numerical implementation.

For the present study, we shall assume that there 
is no subcycling (i.e., $\eta_i = 1$), and no mixed 
methods are employed (i.e., $\beta_i = \beta$, 
$\gamma_i = \gamma$). The errors due to finite 
precision arithmetic and their numerical propagation 
are ignored. 
For convenience, let us denote the drift in the 
displacements and the drift in the accelerations 
along the subdomain interface as follows:
%------------------------------------;
%  Equation: Definitions for drifts  ;
%------------------------------------;
\begin{subequations}
  \begin{align}
  \boldsymbol{a}_{\mathrm{drift}}^{(n)} := \sum_{i=1}^{S} 
    \boldsymbol{C}_i \boldsymbol{a}_i^{(n)} \\
    \boldsymbol{d}_{\mathrm{drift}}^{(n)} := \sum_{i=1}^{S} 
    \boldsymbol{C}_i \boldsymbol{d}_i^{(n)} 
  \end{align} 
\end{subequations}
Basically, the drift in displacements (or 
accelerations) is the measure of error in 
meeting the continuity of displacements 
(or accelerations) across the subdomain 
interface. 
The drifts satisfy the following relations:
%-----------------------------;
%  Equation: Drift relations  ;
%-----------------------------;
\begin{subequations}
  \begin{align}
    \label{Eqn:acc_drift}
    \boldsymbol{a}_{\mathrm{drift}}^{( n + 1)} &= 
    \left( 1 - \frac{1}{\gamma} \right) 
    \boldsymbol{a}_{\mathrm{drift}}^{(n)} \\
    \label{Eqn:disp_drift}
    \boldsymbol{d}_{\mathrm{drift}}^{(n + 1)} &= 
    \boldsymbol{d}_{\mathrm{drift}}^{(n)} +  
    \left( \frac{1}{2} - \frac{\beta}{\gamma}\right) \Delta t^{2}
    \boldsymbol{a}_{\mathrm{drift}}^{(n)} 
  \end{align}
\end{subequations}
Thus, one can draw the following conclusions 
about the drifts:
%-----------------------------------;
%  Enumerate: Conclusions on drift  ;
%-----------------------------------;
\begin{enumerate}[(i)]
\item For numerical stability of a time-stepping 
  scheme under Newmark family, $\gamma \geq 1/2$. 
  Therefore, 
  \begin{align}
    \|\boldsymbol{a}_{\mathrm{drift}}^{(n+1)}\| \leq 
    \|\boldsymbol{a}_{\mathrm{drift}}^{(n)}\|
  \end{align}
  One has the equality only when $\gamma = 1/2$ 
  (e.g., Newmark average acceleration scheme, 
  central difference scheme, Newmark linear 
  acceleration scheme).
\item For any time stepping scheme with 
  $\gamma = 2 \beta$ (e.g., Newmark average 
  acceleration scheme) we have 
  \begin{align}
    \boldsymbol{d}_{\mathrm{drift}}^{(n+1)} = 
    \boldsymbol{d}_{\mathrm{drift}}^{(n)} \quad 
    n = 1, 2, \cdots
  \end{align}
\end{enumerate}
The above claims will be numerically substantiated 
in a subsequent section using the test problem 
outlined in subsection \ref{Subsec:Monolithic_square_plate}.

%
%************************************************;
%                                                ;
%  NAME                                          ;
%    S6_Monolithic_SDOF.tex                      ;
%                                                ;
%  WRITTEN BY                                    ;
%    Saeed Karimi                                ;
%    Kalyana Babu Nakshatrala                    ;
%                                                ;
%************************************************;
\section{SPLIT DEGREE-OF-FREEDOM LUMPED PARAMETER SYSTEM}
\label{Sec:S6_Monolithic_SDOF}
Consider a split agree of freedom whose 
motion can be described by the following 
system of ordinary differential/algebraic equation:
%----------------------------------------------;
%  Equation: DAEs for split degree-of-freedom  ;
%----------------------------------------------;
\begin{subequations}
  \begin{align}
    \label{Eqn:SDOF-2ndOrder-2Dom-1}
    &m_{A} \ddot{u}_A(t) + k_{A} u_A(t) = f_{A}(t) + \lambda(t) \\
    \label{Eqn:SDOF-2ndOrder-2Dom-2}	
    &m_{B} \ddot{u}_B(t) + k_{B} u_B(t) = f_{B}(t) - \lambda(t) \\
    &\dot{u}_{A}(t) - \dot{u}_B(t) = 0
  \end{align}
\end{subequations}
The following parameters are used: $m_{A} = 0.1$, 
$m_{B} = 0.005$, and the stiffness of springs are 
$k_{A} = 2.5$ and $k_{B} = 50$. The subdomain 
time-steps are taken as $\Delta t_{A} = 0.02$ 
and $\Delta t_{B} = 0.005$. The system time-step 
is taken as $\Delta t = 0.02$. The values of the 
external forces are taken to be zero, that is 
$f_{A} = 0$ and $f_{B} = 0$. The initial conditions 
are $d_{0} = 0.1$ and $v_{0} = 1.0$. The problem 
is solved over a time interval of $[0, 0.5]$. In 
all the cases, Newmark average acceleration scheme 
is used in all the subdomains. 
The resulting numerical results for kinematic variables 
are shown in Figure \ref{Fig:SDOF_kinematic}. Since the 
external forces applied are constant ($f_A = f_B = 0$) 
the PH method and the proposed coupling methods yield 
the same results. The GC method suffers from excessive 
damping and fails to match the exact results. Similar 
observation can be made about the interface force as 
well as total physical energy of the system, as shown 
in Figure \ref{Fig:SDOF_lambda_energy}. 

%
%************************************************;
%                                                ;
%  NAME                                          ;
%    S7_Monolithic_Energy.tex                    ;
%                                                ;
%  WRITTEN BY                                    ;
%    Saeed Karimi                                ;
%    Kalyana Babu Nakshatrala                    ;
%                                                ;
%************************************************;
\section{ON ENERGY CONSERVING VS. ENERGY PRESERVING COUPLING}
\label{Sec:S7_Monolithic_Energy}
In this section we address the energy preserving and energy 
conserving properties of the proposed multi-time-step coupling 
method. Two different notions of energy preserving will be 
considered. In particular, the following questions will be 
answered: 
\begin{enumerate}[(a)]
\item Does the coupling method add or extract 
  energy from the system of subdomains in comparison 
  with the case of no coupling?
\item Do the interface forces perform net work?
\item Under what conditions does the coupling method 
  conserve the total energy of the system of subdomains? 
\end{enumerate}
To this end, the kinetic energy and the potential 
energy of the $i$-th subdomain are, respectively, 
defined as follows:
%------------------------------------------------------;
%  Equation: Subdomain kinetic and potential energies  ;
%------------------------------------------------------;
\begin{align}
  \mathcal{T}_i\left(\boldsymbol{v}_i\right) := \frac{1}{2} 
  \boldsymbol{v}_i^{\mathrm{T}} \boldsymbol{M}_i 
  \boldsymbol{v}_i \quad
  \mathcal{V}_i\left(\boldsymbol{d}_i\right) := \frac{1}{2} 
  \boldsymbol{d}_i^{\mathrm{T}} \boldsymbol{K}_i 
  \boldsymbol{d}_i
\end{align}
The total energy of the $i$-th subdomain is given by 
%------------------------------------;
%  Equation: Subdomain total energy  ;
%------------------------------------;
\begin{align}
  \mathcal{E}_i\left(\boldsymbol{d}_i,\boldsymbol{v}_i\right) 
  := \mathcal{T}_i\left(\boldsymbol{v}_i\right) + \mathcal{V}_i 
  \left(\boldsymbol{d}_i\right)
\end{align}
The total energy of all the subdomains at the $n$-th 
(system) time level can be written as follows: 
%-----------------------------------------;
%  Equation: Energy of the entire system  ;
%-----------------------------------------;
\begin{align}
  \mathcal{E}^{(n)} \equiv \mathcal{E}\left(\boldsymbol{d}_1^{(n)}, 
  \cdots, \boldsymbol{d}_S^{(n)}, \boldsymbol{v}_1^{(n)}, \cdots, 
  \boldsymbol{v}_S^{(n)}\right) := \sum_{i=1}^{S} \mathcal{E}_i
  \left(\boldsymbol{d}_i^{(n)}, \boldsymbol{v}_i^{(n)}\right)
\end{align}
In the remainder of this section, we shall 
assume that the external forces are zero 
(i.e., $\boldsymbol{f}_i(t) = \boldsymbol{0} 
\; \forall i$). 
For the proposed multi-time-step method, 
one can derive the following relationship:
%------------------------------------;
%  Equation: Relation for jump in E  ;
%------------------------------------;
\begin{align}
  \label{Eqn:Monolithic_jump_in_E}
  \mathcal{E}^{(n+1)} - \mathcal{E}^{(n)} = 
  \mathcal{E}^{(n \rightarrow n+1)}_{\mathrm{algorithm}} 
  + \mathcal{E}^{(n \rightarrow n+1)}_{\mathrm{interface}}
\end{align}
where $\mathcal{E}_{\mathrm{algorithm}}^{(n \rightarrow n+1)}$ 
and $\mathcal{E}_{\mathrm{interface}}^{(n \rightarrow n + 1)}$ 
are, respectively, defined as follows:
%--------------------------------------------------------;
%  Equation: Definition for E_algorithm and E_interface  ;
%--------------------------------------------------------;
\begin{align}
  \label{Eqn:Monolithic_E_algorithm}
  \mathcal{E}^{(n \rightarrow n+1)}_{\mathrm{algorithm}} &:= 
  - 2 \sum_{i=1}^S \sum_{j=0}^{\eta_i - 1}\left(\gamma_i - \frac{1}{2}\right) 
  \mathcal{V}_i\left(\left[\boldsymbol{d}_i^{\left( n + \frac{j}{\eta_i} \right)}\right]_{i}\right) 
  -\Delta t^2 \sum_{i=1}^{S} 
  \frac{1}{\eta_i^2} \left(\beta_i - \frac{\gamma_i}{2}\right) 
  \left\llbracket \mathcal{T}_i\left(\boldsymbol{a}_i^{(n)}\right)
  \right \rrbracket \nonumber \\
  &-\Delta t^2 \sum_{i=1}^{S} \frac{1}{\eta_i^2} 
  \left(\beta_i - \frac{\gamma_i}{2}\right) 
  \left(2 \gamma_i - 1\right) \left(\sum_{j=0}^{\eta_i -1} 
  \mathcal{T}_i\left(
  \left[\boldsymbol{a}_i^{\left( n + \frac{j}{\eta_i} \right)}\right]_{i}\right) \right)\\
  \label{Eqn:Monolithic_E_interface}
  \mathcal{E}_{\mathrm{interface}}^{(n \rightarrow n+1)} &:= 
  \sum_{i=1}^{S} \sum_{j=0}^{\eta_i - 1}
  \left(\left(1 - \gamma_i\right) \boldsymbol{\lambda}^{\left( n+ \frac{j}{\eta_i} \right)} 
  + \gamma_i \boldsymbol{\lambda}^{\left( n + \frac{j+1}{\eta_i}\right)}\right)^{\mathrm{T}}
  \boldsymbol{C}_i \left[\boldsymbol{d}_i^{\left( n + \frac{j}{\eta_i} \right)}\right]_{i}
\end{align}
If there is no subcycling in all the subdomains 
(i.e., $\eta_i = 1 \; \forall i$), the above 
relationship can be simplified as follows: 
%------------------------------------;
%  Equation: Relation for jump in E  ;
%------------------------------------;
\begin{align}
\label{Eqn:Monolithic_jump_in_E}
  \mathcal{E}^{(n+1)} - \mathcal{E}^{(n)} &= 
  \underbrace{
  - 2 \sum_{i=1}^S \left(\gamma_i - \frac{1}{2}\right) 
  \mathcal{V}_i\left(\left\llbracket\boldsymbol{d}_i^{(n)}
  \right\rrbracket\right)
  -\Delta t^2 \sum_{i=1}^{S} \gamma_i \left(2 \beta_i - 
  \gamma_i\right) \mathcal{T}_i\left(\left\llbracket\boldsymbol{a}_i^{(n)}
    \right\rrbracket\right)}_{\mathcal{E}_{\mathrm{algorithmic}}^{(n \rightarrow n+1)}} 
  \nonumber \\
  &+\underbrace{\sum_{i=1}^{S} 
    \left((1 - \gamma_i) \boldsymbol{\lambda}^{(n)} + 
    \gamma_i \boldsymbol{\lambda}^{(n+1)}\right)^{\mathrm{T}} 
    \boldsymbol{C}_i \left\llbracket\boldsymbol{d}_i^{(n)}
    \right\rrbracket}_{\mathcal{E}_{\mathrm{interface}}^{(n \rightarrow n +1)}}
\end{align}

%====================================================;
%  Subsection: Energy preserving in the first sense  ;
%====================================================;
\subsection{Energy preserving in the first sense}
We shall call that the coupling method \emph{preserves 
  energy in the first sense} if the coupling neither 
adds nor extracts energy over a system time-step in 
comparison to that of no coupling. By no coupling, 
we mean that the problem \eqref{Eqn:vContinuity_GE} 
is solved without decomposing into subdomains (i.e., 
$S = 1$), no subcycling (i.e., $\eta_i = 1$), and no 
mixed methods (i.e., $\gamma_i = \gamma$ and $\beta_i 
= \beta \; \forall i$). We denote the total energy at 
integral time levels under no coupling as follows:
%----------------------------;
%  Equation: En no coupling  ;
%----------------------------;
\begin{align}
  \mathcal{E}^{(n)}_{\mbox{no coupling}} := 
  \mathcal{T}^{(n)}_{\mbox{no coupling}} 
  + \mathcal{V}^{(n)}_{\mbox{no coupling}}
\end{align}
where 
%--------------------------------------------;
%  Equation: Kinetic and potential energies  ;  
%--------------------------------------------;
\begin{align}
  &\mathcal{T}^{(n)}_{\mbox{no coupling}} := \frac{1}{2} 
  {\boldsymbol{v}^{(n)}}^{\mathrm{T}} \boldsymbol{M} 
  \boldsymbol{v}^{(n)} \\
  &\mathcal{V}^{(n)}_{\mbox{no coupling}} := \frac{1}{2} 
             {\boldsymbol{d}^{(n)}}^{\mathrm{T}} 
             \boldsymbol{K} \boldsymbol{d}^{(n)}
\end{align}
Mathematically, preserving energy in the 
first sense implies that 
%--------------------------------------------------;
%  Equation: Preserving energy in the first sense  ;
%--------------------------------------------------;
\begin{align}
  \mathcal{E}^{(n)} = \mathcal{E}^{(n)}_{\mbox{no coupling}} 
  \quad \forall n
\end{align}
The numerical solution presented in Figure 
\ref{Fig:Monolithic_energy_preservation} 
confirms that the proposed multi-time-step 
coupling method, in general, does not 
preserve energy in the first sense. 

%------------------------------------------------;
%  Remark: Remark about algorithmic dissipation  ;
%------------------------------------------------;
\begin{remark}
  It should be noted that many stable time stepping 
  schemes under the Newmark family are dissipative 
  \cite{Hughes}. That is, 
  %--------------------------------;
  %  Equation: Energy dissipation  ;
  %--------------------------------;
  \begin{align}
    \mathcal{E}^{(n+1)}_{\mbox{no coupling}} 
    < \mathcal{E}^{(n)}_{\mbox{no coupling}} 
    \quad \forall n
  \end{align}
  Only the Newmark average acceleration scheme 
  $(\gamma = 1/2, \; \beta = 1/4)$ under the 
  Newmark family conserves energy for linear 
  problems (e.g. linear elastodynamics). That is, 
  %---------------------------------;
  %  Equation: Energy conservation  ;
  %---------------------------------;
  \begin{align}
    \mathcal{E}^{(n+1)}_{\mbox{no coupling}} = 
    \mathcal{E}^{(n)}_{\mbox{no coupling}} 
    \quad \forall n
  \end{align}
\end{remark}

%=====================================================;
%  Subsection: Energy preserving in the second sense  ;
%=====================================================;
\subsection{Energy preserving in the second sense}
We shall call that the coupling method \emph{preserves 
  energy in the second sense} if the interface forces 
(i.e., the multipliers $\boldsymbol{\lambda}$) do not 
perform net work over a system time-step. That is, 
%--------------------------------------------------;
%  Equation: Preserves energy in the second sense  ;
%--------------------------------------------------;
\begin{align}
  \mathcal{E}^{(n\rightarrow n+1)}_{\mathrm{interface}} = 0 
  \quad \forall n
\end{align}
In general, the proposed multi-time-step 
coupling method does not preserve energy 
even in the second sense. However, using 
equation \eqref{Eqn:Monolithic_E_interface}, 
one can show that a sufficient condition for 
the coupling method to preserve energy in the 
second sense is to have $\gamma_i = \gamma \; 
\forall i$, $\gamma_i = 2 \beta_i$, and no 
subcycling (i.e., $\eta_i = 1 \; \forall i$). 
This sufficient condition also guides one to 
construct a simple example that substantiates 
the claim that the proposed coupling method need not preserve 
the energy in the second sense. By choosing Newmark 
average acceleration scheme $(\gamma = 1/2, \; 
\beta = 1/4)$ in all subdomains we will have 
$\mathcal{E}^{(n\rightarrow n+1)}_{\mathrm{algorithm}} = 0 \; 
\forall n$. This implies that the difference between 
$\mathcal{E}^{(n+1)}$ and $\mathcal{E}^{(n)}$ is solely 
due to $\mathcal{E}^{(n\rightarrow n+1)}_{\mathrm{interface}}$. 
If there is subcycling then one could have 
%----------------------------------------;
%  Equation: Subcycling and E_interface  ;
%----------------------------------------;
\begin{align}
  \mathcal{E}^{(n \rightarrow n+1)}_{\mathrm{interface}} 
  \neq 0
\end{align}
Based on the above reasoning, Figure 
\ref{Fig:Monolithic_multi_time_step_energy_plots} 
presents the numerical results to substantiates 
the above claim. 

%=================================;
%  Subsection: Energy conserving  ;
%=================================;
\subsection{Energy conserving}
We shall say that the coupling method 
\emph{conserves energy} exactly if 
%------------------------------;
%  Equation: Conserves energy  ;
%------------------------------;
\begin{align}
  \mathcal{E}^{(n+1)} = \mathcal{E}^{(n)} 
  \quad \forall n
\end{align}
Based on equation \eqref{Eqn:Monolithic_jump_in_E}, a 
necessary and sufficient condition for the coupling 
method for conserve energy is 
\begin{align}
  \mathcal{E}^{(n \rightarrow n+1)}_{\mathrm{algorithm}} + 
  \mathcal{E}^{(n \rightarrow n+1)}_{\mathrm{interface}} 
  = 0 \quad \forall n
\end{align}
where $\mathcal{E}^{(n \rightarrow n+1)}_{\mathrm{algorithm}}$ 
and $\mathcal{E}^{(n \rightarrow n+1)}_{\mathrm{interface}}$ are, 
respectively, defined in equations \eqref{Eqn:Monolithic_E_algorithm} 
and \eqref{Eqn:Monolithic_E_interface}. 
A sufficient condition can be written as follows: 
%--------------------------------------------------------------;
%  Equation: Equivalent statements for conservation of energy  ;
%--------------------------------------------------------------;
\begin{align}
  \mathcal{E}^{(n \rightarrow n+1)}_{\mathrm{algorithm}} 
  = 0 \quad \mathrm{and} \quad 
  \mathcal{E}^{(n \rightarrow n+1)}_{\mathrm{interface}} 
  = 0  \quad \forall n
\end{align}
The following theorem provides a way to achieve 
the above sufficient condition. 

%--------------------------------;
%  Theorem: Energy conservation  ;
%--------------------------------;
\begin{theorem}
  If all the subdomains employ the Newmark average acceleration 
  scheme (i.e., $\gamma_i = 1/2$ and $\beta_i = 1/4$), and there 
  is no subcycling (i.e., $\eta_i = 1 \; \forall i$), then the 
  coupling method exactly conserves energy when $\boldsymbol{f}_i
  (t) = \boldsymbol{0} \; \forall i$.
\end{theorem}
%--------------------------------------------;
%  Proof of the energy conservation theorem  ;
%--------------------------------------------;
\begin{proof}
  This proof is a simple extension of the proof for single domain 
  (i.e., without coupling). For Newmark average acceleration time 
  stepping scheme the following identities hold:
  \begin{subequations}
    \begin{align}
      \label{Eqn:Monolithic_jump_d}
      \left\llbracket\boldsymbol{d}_i^{(n)}\right\rrbracket = 
      \Delta t \left\llangle\boldsymbol{v}_i^{(n)}\right\rrangle \\
      \label{Eqn:Monolithic_jump_v}
      \left\llbracket\boldsymbol{v}_i^{(n)}\right\rrbracket = 
      \Delta t \left\llangle\boldsymbol{a}_i^{(n)}\right\rrangle
    \end{align} 
  \end{subequations}
  The governing equation for $i$-th subdomain implies that 
  \begin{align}
    \boldsymbol{M}_i \left\llangle\boldsymbol{a}_i^{(n)} \right\rrangle 
    + \boldsymbol{K}_i \left\llangle\boldsymbol{d}_i^{(n)}\right\rrangle 
    = \boldsymbol{C}_i^{\mathrm{T}} \left\llangle\boldsymbol{\lambda}^{(n)}
    \right\rrangle
  \end{align}
  Premultiplying by $\left\llangle\boldsymbol{v}_i^{(n)}\right\rrangle$, 
  using the above relations \eqref{Eqn:Monolithic_jump_d}--\eqref{Eqn:Monolithic_jump_v}, 
  summing over all the subdomains and using the continuity of velocities, we get the following:
  \begin{align}
    \sum_{i=1}^{S} \left\llangle\boldsymbol{v}_i^{(n)}\right\rrangle^{\mathrm{T}}
    \boldsymbol{M}_i \left\llbracket\boldsymbol{v}_i^{(n)}\right\rrbracket 
    + \sum_{i=1}^{S}
    {\left\llbracket\boldsymbol{d}_i^{(n)}\right\rrbracket}^{\mathrm{T}} 
    \boldsymbol{K}_i \left\llangle\boldsymbol{d}_i^{(n)}\right\rrangle  
    &= \sum_{i=1}^S
               {\left\llangle\boldsymbol{v}_i^{(n)}\right\rrangle}^{\mathrm{T}}  
               \boldsymbol{C}_i^{\mathrm{T}} \left\llangle\boldsymbol{\lambda}^{(n)} \right\rrangle 
               \nonumber \\
               &= {\left\llangle\boldsymbol{\lambda}^{(n)}\right\rrangle}^{\mathrm{T}} \sum_{i=1}^{S} 
               \boldsymbol{C}_i \left\llangle\boldsymbol{v}_i^{(n)} \right\rrangle = 0
  \end{align}
  Using the symmetry of the matrices $\boldsymbol{M}_i$ 
  and $\boldsymbol{K}_i$, and noting the linearity of 
  the jump operator, we have the following 
  \begin{align}
    \left\llbracket \frac{1}{2} 
    \sum_{i=1}^{S} {\boldsymbol{v}_i^{(n)}}^{\mathrm{T}} 
    \boldsymbol{M}_i \boldsymbol{v}_i^{(n)} + \frac{1}{2} 
    \sum_{i=1}^{S}{\boldsymbol{d}_i^{(n)}}^{\mathrm{T}} 
    \boldsymbol{K}_i \boldsymbol{d}_i^{(n)} \right\rrbracket = 0
  \end{align}
  which shows that the total energy is exactly 
  conserved over a system time-step.
\end{proof}
It is noteworthy that if $\gamma_i = \gamma > 
1/2$ and $\beta_i = \gamma_i / 2$, and there 
is no subcycling then we have 
\begin{align}
  \mathcal{E}^{(n+1)} - \mathcal{E}^{(n)} 
  = -2 \left(\gamma - \frac{1}{2}\right) 
   \sum_{i=1}^{S} \mathcal{V}_{i} 
   \left(\left\llbracket\boldsymbol{d}^{(n)}_i 
   \right\rrbracket\right) < 0
\end{align}
which implies that the coupling method will 
be strictly energy decaying. As mentioned 
in Section \ref{Sec:S2_Monolithic_Newmark}, 
$\gamma < 1/2$ is not in the allowable 
range of values under the Newmark family 
of time integrators because of numerical 
stability.

%=========================================================;
%  Subsection: Is the PH method really energy preserving  ;
%=========================================================;
\subsection{Is the PH method really energy preserving?}
We are now set nicely to examine the claim made in 
Reference \cite{Prakash_Hjelmstad_IJNME_2004_v61_p2183} 
that the PH method preserves energy. 
In the absence of external forces, the proposed 
coupling method is the same as the PH method. 
Therefore, based on the earlier discussion in 
this section, the PH method is neither energy 
conserving nor energy preserving in both first 
and second senses. 
The source of error that led to the false claim 
is due to the use of an inappropriate definition 
for the work done by the interface. 
Using the notation introduced in this paper, the 
expression considered in \cite[equations (58) 
and (61)]{Prakash_Hjelmstad_IJNME_2004_v61_p2183} for work done 
by the interface can be written as follows:
 \begin{align}
    \frac{1}{\Delta t_A} \left\llbracket \boldsymbol{v}_A^{(n)}
    \right\rrbracket^{\mathrm{T}} 
    \boldsymbol{C}_{A}^{\mathrm{T}} \left\llbracket
    \boldsymbol{\lambda}^{(n)}\right\rrbracket 
    + \frac{1}{\Delta t_B} \sum_{j=1}^{\eta_B} 
    \left[\boldsymbol{v}_B^{\left( n+ \frac{j-1}{\eta_B} \right)}\right]_{B}^{\mathrm{T}} 
    \boldsymbol{C}_{B}^{\mathrm{T}} \left[\boldsymbol{\lambda}^{\left( n +\frac{j-1}{\eta_B} \right)} \right]_{B}
\end{align}
But the above expression is \emph{not} 
appropriate for the work done by the 
interface forces. 
A comment is also warranted on the numerical 
results presented in \cite[Figures 8 \& 11]{Prakash_Hjelmstad_IJNME_2004_v61_p2183}, 
which have been used to support their claim. 
For the chosen test problems, these figures 
report that $E_{\mathrm{total}}$ is constant 
under the PH method where 
\begin{align}
E_{\mathrm{total}} := \frac{1}{2} {\boldsymbol{a}_{A}^{(n)}}^{\mathrm{T}} 
\boldsymbol{A}_A \boldsymbol{a}_{A}^{(n)} + 
\frac{1}{2} {\boldsymbol{v}_{A}^{(n)}}^{\mathrm{T}}
\boldsymbol{K}_A \boldsymbol{v}_{A}^{(n)} + 
\frac{1}{2} {\boldsymbol{a}_{B}^{(n)}}^{\mathrm{T}} 
\boldsymbol{A}_B \boldsymbol{a}_{B}^{(n)} + 
\frac{1}{2} {\boldsymbol{v}_{B}^{(n)}}^{\mathrm{T}}
\boldsymbol{K}_B \boldsymbol{v}_{B}^{(n)}
\end{align}
Recall that 
\begin{align}
\mathbf{A}_i := \boldsymbol{M}_i + \Delta 
t_{i}^2 \left(\beta_i - \frac{\gamma_i}{2}
\right) \boldsymbol{K}_i \quad i = A, B
\end{align}
The constant value for $E_{\mathrm{total}}$ 
has then been used to support that the PH 
method is energy preserving. The following 
remarks on the nature of $E_{\mathrm{total}}$ 
will put the things in perspective: 
\begin{enumerate}[(i)]
\item $E_{\mathrm{total}}$ is \emph{not} equal 
to the physical total energy of the system (i.e., 
the sum of kinetic and potential energies). Hence, 
the preservation of $E_{\mathrm{total}}$ does not 
imply the preservation of the physical total 
energy of the system.
\item Even this quantity will not be constant 
under the PH method if the Newmark parameter 
$\gamma \neq 1/2$ even in one subdomain. The 
result shown in reference \cite[Figures 8 \& 11]
{Prakash_Hjelmstad_IJNME_2004_v61_p2183} used 
$\gamma = 1/2$ in all the subdomains.
\item It is also noteworthy that $E_{\mathrm{total}}$ 
is constant even for non-zero constant external force, 
which will not be the case with the physical total 
energy.
\item If preservation of such a quantity is essential 
for some reason, it should be noted that the proposed 
coupling method will also preserve $E_{\mathrm{total}}$ 
under the same assumptions on the Newmark parameter 
and the external force.
\end{enumerate}

%=========================================================;
%  Subsection: Effect of system time-step and subcycling  ;
%=========================================================;
\subsection{On the effect of system time-step 
  and subcycling on accuracy}
In absence of external forces, the exact solution satisfies 
$\mathcal{E}^{(n+1)} - \mathcal{E}^{(n)} = 0$. Therefore, the 
quantities $\mathcal{E}_{\mathrm{algorithm}}^{(n\rightarrow n+1)}$ and 
$\mathcal{E}_{\mathrm{interface}}^{(n\rightarrow n+1)}$ can serve as 
error / accuracy indicators of a multi-time-stepping scheme. 
Note that these quantities arise, respectively, due to 
time-stepping scheme, and due to decomposing domain into 
subdomains. Of course, both these quantities are affected 
by subcycling. 

%  E_algorithm  
From equation \eqref{Eqn:Monolithic_E_algorithm}, 
it is easy to check that
 $\mathcal{E}_{\mathrm{algorithm}}^{(n\rightarrow n+1)}$ 
is proportional to $\Delta t^2$ and inversely 
proportional to $\eta_i^2$. Therefore, algorithmic 
error in the subdomains can always be decreased  
by employing either of these two strategies: 
\begin{itemize}
\item decreasing the system time-step by keeping 
  the subcycling ratios fixed (i.e., keeping 
  $\eta_i$ fixed) 
\item decreasing the subdomain time-step (i.e., 
  increase the values of $\eta_i$) by keeping 
  the system time-step fixed
\end{itemize}
Equation \eqref{Eqn:Monolithic_E_interface} can 
be written as follows: 
\begin{align}
  \mathcal{E}_{\mathrm{interface}}^{(n \rightarrow n+1)} = 
  \Delta t \sum_{i=1}^{S} \left\{\frac{1}{\eta_i} 
  \sum_{j=0}^{\eta_i - 1}
  \left(\left(1 - \gamma_i\right) 
  \boldsymbol{\lambda}^{\left( n+ \frac{j}{\eta_i} \right)} 
  + \gamma_i \boldsymbol{\lambda}^{\left( n + \frac{j+1}{\eta_i}\right)}
  \right)^{\mathrm{T}}\boldsymbol{C}_i 
  \boldsymbol{v}_i^{\left( n + \frac{j}{\eta_i} \right)}\right\}
  + O\left(\frac{\Delta t^2}{\eta_i^2}\right)   
\end{align}
$\mathcal{E}_{\mathrm{interface}}^{(n\rightarrow n+1)}$ is 
linearly proportional to $\Delta t$, which indicates 
that the error due to domain decomposition can \emph{always} 
be decreased with lowering the system time-step. However, 
for a fixed system time-step, the quantity in the 
parenthesis can be of $O(1)$ in magnitude. Therefore, 
choosing smaller subdomain time-steps while keeping the 
system time-step fixed need not improve the accuracy. 
This quantity may even grow with increase in the subcycling ratios. 
Hence, an appropriate quantity that can indicate the 
improvement or worsening of accuracy by subcycling 
is $\mathcal{E}_{\mathrm{interface}}^{(n \rightarrow n+1)}$, 
which can be calculated on the fly during a numerical 
simulation. Larger values of 
$\mathcal{E}_{\mathrm{interface}}^{(n \rightarrow n+1)}$ 
in magnitude implies that subcycling is adversely 
affecting the accuracy. 

Summarizing, the accuracy of the numerical 
results under the proposed multi-time-step 
method can \emph{always} be improved by 
decreasing the system time-step. 
The overall accuracy need not always improve 
with subcycling for a fixed system time-step. 
These theoretical observations are numerically 
verified in Section \ref{Sec:S9_Monolithic_NR}.

%
%**************************************************;
%                                                  ;
%  NAME                                            ;
%    S8_Monolithic_BDF_IRK.tex                     ;
%                                                  ;
%  WRITTEN BY                                      ;
%    Saeed Karimi                                  ;
%    Kalyana Babu Nakshatrala                      ;
%                                                  ;
%**************************************************;
\section{ON THE PERFORMANCE OF BACKWARD DIFFERENCE FORMULAE 
  AND IMPLICIT RUNGE-KUTTA SCHEMES}
\label{Sec:S8_Monolithic_BDF_IRK}
In the numerical analysis literature, backward 
difference formulae (BDF) and implicit Runge-Kutta 
(IRK) schemes have been the schemes of choice for 
solving DAEs \cite{Hairer_Wanner,Ascher_Petzold}. 
The following quote by Petzold has been a popular 
catch-phrase for promoting BDF schemes: ``\emph{BDF 
is so beautiful that it is hard to imagine something 
else could be better}'' \cite[p.~481]{Hairer_Wanner}. 
This statement may be true for first-order DAEs that 
arise from modeling of physical systems involving 
dissipation. But these two families of schemes may 
\emph{not} be the best choices for second-order 
DAEs that posses important physical invariants 
(e.g., conservation of energy). 
In the context of second-order DAEs, Newmark family of 
time stepping schemes is also ``beautiful'' and can be 
``better.'' The Newmark family of time stepping schemes 
(which have been popular in Civil Engineering for solving 
ODEs arising in structural dynamics and earthquake 
engineering) did not get as much attention as they 
deserve to solve DAEs in both numerical analysis and 
engineering communities. The algebraic constraints 
in a DAE introduce high frequency modes, and fully 
implicit schemes such as Newmark family of time 
stepping schemes are particularly suited to avoid 
instabilities due to high frequency modes without 
introducing excessive damping. 

We now show that there will be excessive numerical 
damping if the proposed coupling method is based
on BDF or IRK schemes instead of the Newmark 
family of time stepping schemes. It may be argued 
that numerical damping is good for numerical stability, 
but excessive damping fails to preserve the important 
invariants (e.g., conservation of energy). Newmark 
family of time stepping schemes provide much better 
results under the same system time-step, especially, 
in the prediction of important physical invariants. 

The simplest scheme under both BDF and IRK families 
is the backward Euler scheme (which is also referred 
to as the implicit Euler scheme). We rewrite the 
governing equations as the following first-order 
DAEs:
%------------------------------;
%  Equation: First-order DAEs  ;
%------------------------------;
\begin{subequations}
  \label{S8:1st_order}
  \begin{align}
    &\boldsymbol{M}_i \dot{\boldsymbol{v}}_i(t) + 
    \boldsymbol{K}_i \boldsymbol{d}_i(t) = \boldsymbol{f}_i(t)
    + \boldsymbol{C}_i^{\mathrm{T}} \boldsymbol{\lambda} 
    \quad \forall i = 1, \cdots ,S \\
    &\dot{\boldsymbol{d}_i}(t) = \boldsymbol{v}_i(t) \\
    &\sum_{i=1}^{S} \boldsymbol{C}_i \boldsymbol{v}_i(t) = 
    \boldsymbol{0}
\end{align}   
\end{subequations}
Under the backward Euler scheme, the velocities 
and accelerations are approximated as follows:
%---------------------------------------------;
%  Equation: BE of velocity and acceleration  ;
%---------------------------------------------;
\begin{align}
  \label{S8:BE_approx}
  \boldsymbol{v}_{i}^{\left( n + 1 \right)} = 
  \frac{\boldsymbol{d}_i^{\left( n + 1 \right)} - 
    \boldsymbol{d}_i^{\left( n \right)}}{\Delta t}, 
  \quad  
  \boldsymbol{a}_{i}^{\left( n + 1 \right)} = 
  \frac{\boldsymbol{v}_i^{\left( n + 1 \right)} - 
    \boldsymbol{v}_i^{\left( n \right)}}{\Delta t} 
\end{align}
In the absence of subcycling, following a similar procedure 
presented in the previous sections, one can arrive at the 
following equation for the coupling method based on the 
backward Euler scheme:
%---------------------------------;
%  Equation: Jump in En under BE  ;
%---------------------------------;
\begin{align}
  \left\llbracket \mathcal{E}^{\left( n \right)}\right
  \rrbracket= -\sum_{i = 1}^{S} \left(\mathcal{T}_i 
  \left( \llbracket \boldsymbol{v}_i^{\left( n \right)}
    \rrbracket \right) + \mathcal{V}_i \left( \llbracket 
    \boldsymbol{d}_i^{\left( n \right)}\rrbracket \right) 
  \right) \quad \forall n
\end{align}
which is \emph{strictly negative} for any non-trivial motion 
of the subdomains. Figure \ref{Fig:Monolithic_energy_BE} 
nicely summarizes the above discussion using the split 
degree-of-freedom problem. The system shown in Figure
\ref{fig:SDOF2} was solved with the parameters: $m_A 
= 0.1$, $m_B = 0.005$, $k_A = 2.5$ and $k_B = 50$. The 
initial values are set to be as follows: $v_0 = 1.0$ 
and $d_0 = 0.1$. External forces are set to be zero. 
The proposed coupling method presented in this paper 
is employed to solve the coupled system using Newmark 
average acceleration and central difference methods, 
with no subcycling.   
In addition to the aforementioned excessive numerical 
dissipation, the following factors make BDF and IRK 
schemes not particularly suitable for developing a 
multi-time-step coupling:
\begin{enumerate}[(a)]
\item High-order BDF and IRK schemes are non-self-starting. 
\item BDF and IRK schemes are developed for first-order DAEs. 
  To solve a second-order DAE (which is the case in this paper), 
  auxiliary variables need to be introduced, which will increase 
  the number of unknowns and the computational cost. 
\item IRK schemes involve multiple stages, and are 
  generally considered difficult to implement. 
\end{enumerate}

%
%************************************************;
%                                                ;
%  NAME                                          ;
%    S9_Monolithic_NR.tex                        ;
%                                                ;
%  WRITTEN BY                                    ;
%    Saeed Karimi                                ;
%    Kalyana Babu Nakshatrala                    ;
%                                                ;
%************************************************;
\section{REPRESENTATIVE NUMERICAL RESULTS}
\label{Sec:S9_Monolithic_NR}
Using several canonical problems, we illustrate 
that the proposed multi-time-step coupling method 
possesses the following desirable properties:
%------------------------------------------------;
%  Enumerate: Properties of the proposed method  ;
%------------------------------------------------;
\begin{enumerate}[(I)]
\item All subdomains can subcycle simultaneously. 
  That is, $\Delta t_i < \Delta t \; \forall i = 1, 
  \cdots, S$.
  \item The method can handle multiple subdomains. 
  \item Drift in displacements along the subdomain 
  interface is not significant. 
  \item Under fixed subdomain time-steps, the 
  accuracy of numerical solutions can be 
  improved by decreasing the system time-step.
\item For a fixed system time-step, accuracy of 
  the solutions \emph{may} be improved using 
  subcycling.  We also show that monitoring 
  $\mathcal{E}_{\mathrm{interface}}^{(n\rightarrow n+1)}$ 
  at every system time-step can serve as a simple 
  criterion to decide whether or not subcycling 
  will improve the accuracy. This criterion can 
  be calculated on the fly during a numerical 
  simulation. 
\end{enumerate}

%=============================================================;
%  Subsection: Split degree-of-freedom with three subdomains  ;
%=============================================================;
\subsection{Split degree-of-freedom with three subdomains}
An attractive feature of the proposed coupling method is that 
it can handle multiple subdomains, which is illustrated in this 
test problem. The single degree-of-freedom is split into three 
subdomains $A$, $B$ and $C$, as shown in Figure \ref{fig:SDOF3}. 
The problem parameters are taken as follows: $m_A = 5$, $m_B 
= 0.1$, $m_C = 0.01$, $k_A = 5$, $k_B = 2.5$ and $k_C = 4$. 
Subdomain time-steps are taken as $\Delta t_A = 0.01$, $\Delta 
t_B = 0.005$ and $\Delta t_C = 0.0025$. The system time-step 
is taken as $\Delta t = 0.01$. Newmark average acceleration 
scheme is employed in all the subdomains. The subdomain external 
forces are taken as $f_A = f_C = 0$ and $f_B = 1$. The system 
is subject to the initial conditions $d_0 = 1.0$ and $v_0 = 0.0$. 
Figure \ref{Fig:Monolithic_3SDOF_kinematic} compares analytical 
solution with the numerical results for the kinematic quantities. 
Figure \ref{Fig:Monolithic_3SDOF_lambdas_energy} shows the 
Lagrange multipliers (i.e., interface forces) and the total 
energy of the system. The proposed coupling method performed 
well.

%===================================================================;
%  Subsection: One-dimensional problem with homogeneous properties  ;
%===================================================================;
\subsection{One-dimensional problem with homogeneous properties}
Consider the vibration of a homogeneous one-dimensional 
elastic axial bar with the left end of the bar fixed and 
a constant tip load is applied at the right end of the 
bar. The governing equations take the following form: 
%--------------------------------;
%  Equation: Governing equation  ;
%--------------------------------;
\begin{subequations}
  \begin{align}
    &\rho A \frac{\partial^2 u}{\partial t^2} - 
    \frac{\partial}{\partial x}\left(E A 
    \frac{\partial u}{\partial x}\right) 
    = P \delta(x=L) H(t=0) \quad \forall 
    x \in (0,L), \; \forall t \in (0, T] \\
      &u(x = 0, t) = 0 \quad \forall t \in (0, T] \\
        &E \frac{\partial u}{\partial x} (x = L, t) = 0 
        \quad \forall t \in (0, T] \\
          &u(x,t=0) = 0 \quad \forall x \in (0,L)\\
          &\frac{\partial u}{\partial t}(x,t=0) = 0 
          \quad \forall x \in (0,L)
  \end{align}
\end{subequations}
where $\delta(\cdot)$ is the Dirac-delta distribution, 
$H(\cdot)$ is the Heaviside function, and $P$ is a 
constant tip loading. The analytical solution for 
the displacement can be written as follows:
%---------------------------------;
%  Equation: Analytical solution  ;
%---------------------------------;
\begin{align}
  u(x,t) = \frac{Px}{EA} + \frac{8 P L}{\pi^2 EA} 
  \sum_{n=1,3, \cdots} (-1)^{\frac{n+1}{2}} 
  \frac{1}{n^2} \sin(\beta_n x) \cos(\omega_n t)
\end{align}
where 
%--------------------------------;
%  Equation: beta_n and omega_n  ;
%--------------------------------;
\begin{align}
  \beta_n = \frac{n \pi}{2 L}, \quad 
  \omega_n = \beta_n \sqrt{\frac{E}{\rho}} = 
  \frac{n \pi}{2L} \sqrt{\frac{E}{\rho}} 
\end{align}
This test problem is the same as the one considered in 
Reference \cite{Prakash_Hjelmstad_IJNME_2004_v61_p2183} 
but with different parameters. Herein, we shall use this 
test problem to illustrate that the proposed coupling 
method can handle multiple subdomains simultaneously, which is not 
the case with the PH method as presented in 
\cite{Prakash_Hjelmstad_IJNME_2004_v61_p2183}.  

The computational domain is divided into three 
subdomains of equal lengths, as shown in Figure 
\ref{Fig:1_dimensional_mainfig}. Each subdomain 
is uniformly meshed using five two-node line 
elements. The Young's modulus is taken as $E 
= 10^{4}$, the density $\rho = 0.1$, the area 
of cross section $A = 1$, the total length of 
the bar $L = 1$, and the tip loading is taken 
as $P = 10$. 
Newmark average acceleration scheme is employed in 
subdomains $A$ and $C$ ($\beta_A = \beta_C = 1/4$ 
and $\gamma_A = \gamma_C = 1/2$), and the central 
difference scheme is employed in subdomain $B$ 
($\beta_B = 0$ and $\gamma_B = 1/2$).  
The critical time-step is $1.217 \times 10^{-4}$. The 
system time-step is taken as $\Delta t = 10^{-3}$. The 
subdomain time-steps for $A$ and $C$ are taken as $\Delta 
t_A = \Delta t_C = 10^{-3}$. The problem is solved using 
three different subdomain time-steps for $B$, which are 
defined through $\eta_B = \Delta t / \Delta t_B = 10, \; 
100, \; 1000$. 
Figure \ref{Fig:1_dimensional_dtip_energy} shows 
the tip displacement and the total energy obtained 
using the proposed coupling method. Figures 
\ref{Fig:1_dimensional_drift} and 
\ref{Fig:1_dimensional_lambdas}, respectively, 
show drift in displacements and the interface 
Lagrange multipliers. \emph{These figure clearly 
  illustrate that, under a fixed system time-step, 
  the accuracy can be improved by employing subcycling 
  in the subdomains under the proposed coupling method. 
  This implies that the time-step required for the 
  explicit scheme need not limit the time-step in 
  the entire computational domain under the proposed 
  multi-time-step coupling method.}

The problem is solved again with subdomain B divided 
into 10 two-node linear elements. In this case the 
critical time-step is approximately $6.085 \times 
10^{-5}$. We took the subdomain time-steps to be 
fixed at $10^{-5}$ and altered the system time-step 
to  illustrate the effect of subcycling. The results 
are presented in figures \ref{Fig:1_dimensional_dtip_energy_2}, 
\ref{Fig:1_dimensional_drift_2} and \ref{Fig:1_dimensional_lambdas_2}. 
\emph{These figures illustrate that, under the proposed 
  multi-time-step coupling method with fixed subdomain 
  time-steps (i.e., fixed $\Delta t_i $), the accuracy 
  can be improved by employing smaller system time-steps.}

%========================================================;
%  Subsection: Square plate subjected to a corner force  ;
%========================================================;
\subsection{Square plate subjected to a corner force}
\label{Subsec:Monolithic_square_plate}
A bi-unit square of homogeneous elastic material 
is fixed at the left end and a constant force with 
components $f_x = f_y = 1$ is applied at the right 
bottom corner. 
The Lam\'e parameters are taken as $\lambda = 100$ and 
$\mu = 100$, and the mass density is taken as $\rho = 
100$. The computational domain is decomposed into four 
equally sized square subdomains. Four node quadrilateral 
elements are used to form a 5-element by 5-element mesh 
for each subdomain. Figure \ref{Fig:Monolithic_2D_pictorial} 
provides a pictorial description of the problem. 
A similar problem is also considered in Reference 
\cite{2010_Mahjoubi_Krenk_IJNME_v83_p1700}, which 
also addressed multi-time-step coupling method 
for structural dynamics. 

The central difference scheme $(\beta = 0, \gamma = 1/2)$ 
is employed for subdomains 1, 2 and 3, and Newmark average 
acceleration scheme $(\beta = 1/4, \gamma = 1/2)$ is employed 
for subdomain 4. 
Figure \ref{Fig:Monolithic_dve_h1} illustrates 
that the accuracy can be improved by decreasing 
the system time-step. 
Figure \ref{Fig:Monolithic_dve_h2} illustrates that 
the accuracy need not always improve by decreasing 
subdomain time-steps for a fixed system time-step. 
This is completely in accordance with the theoretical 
predictions. The numerical results shown in Figure 
\ref{Fig:Monolithic_dve_h2} also illustrate that the 
proposed coupling method allows subcycling in all 
the subdomains. This is evident from the fact that 
all the chosen values for $\eta_i \; (i = 1, \cdots, 
S)$ are greater than unity. 
As it can be seen in Figure \ref{Fig:2D_E_interface}, 
subcycling can result in increase in drift. 
Figure \ref{Fig:Monolithic_2D_d_and_v_x_dot8} shows that 
there is no \emph{appreciable} drift in displacements 
along the subdomain interface, and there is no drift 
in the velocities along the subdomain interface, as 
the proposed method imposes constraints on the continuity 
of velocities at every system time-step. 
Figure \ref{Fig:Bound_on_drift} shows that the 
theoretical bounds on the drifts in equations 
\eqref{Eqn:acc_drift}--\eqref{Eqn:disp_drift} 
match well with the numerical results.  
In all the numerical results, the proposed 
multi-time-step coupling method performed 
well, and behaved in accordance with the 
theoretical predictions derived in this 
paper. 

%=======================================================;
% 	subsection: 2D elastic wave propagation  	;
%=======================================================;
\subsection{Two-dimensional wave propagation problem}
\label{Subsec:2D_Wave}
Consider the transverse motion of a plate 
subject governed by the following equations:
\begin{subequations}
\label{Eqn:2D_Wave}
\begin{align}
  &\frac{1}{c_0^{2}} \frac{\partial^2 u}{\partial t^2} 
  -\left(\frac{\partial^2 u}{\partial x^2} + 
  \frac{\partial^2 u}{\partial y^2}
  \right) = f(\mathbf{x},t) \quad
  \forall \left( \mathbf{x}, t \right) \in \Omega \times \mathcal{I}\\
  &u(\mathbf{x},t) = u^{\mathrm{p}}(\mathbf{x},t) 
  \quad \left( \mathbf{x}, t \right) \in 
  \Gamma^{\mathrm{D}} \times \mathcal{I} \\
  &\mathrm{grad}[u] \cdot \widehat{\mathbf{n}} 
  (\mathbf{x}) = s^{\mathrm{p}}(\mathbf{x},t) \quad \left( \mathbf{x}, t \right) 
  \in \Gamma^{\mathrm{N}}\times \mathcal{I} \\
  &u(\mathbf{x},t = 0) = u_0(\mathbf{x}) \quad 
  \mathbf{x} \in \Omega \\
  &\dot{u}(\mathbf{x},t = 0) = v_0(\mathbf{x}) 
  \quad \mathbf{x} \in \Omega 
\end{align}
\end{subequations}
where $u(\mathbf{x},t)$ is the transverse 
displacement, $c_0$ is the wave velocity, $f(\mathbf{x},t)$ 
is the forcing function, $\widehat{\mathbf{n}}(\mathbf{x})$ 
is the unit outward normal to the boundary, $u^{\mathrm{p}}
(\mathbf{x},t)$ is the prescribed displacement on the 
boundary, $s^{\mathrm{p}}(\mathbf{x},t)$ is the prescribed 
transverse traction, and $u_0(\mathbf{x})$ and 
$v_0(\mathbf{x})$ are, respectively, the initial displacement 
and initial velocity. Computational domain is denoted by
$\Omega$. The part of the boundary on which Neumann 
boundary condition is denoted by $\Gamma^{\mathrm{N}}$, and 
$\Gamma^{\mathrm{D}}$ is the part of the boundary on which 
Dirichlet boundary condition is prescribed. As usual, 
we assume that $\Gamma^{\mathrm{D}} \cap \Gamma^{\mathrm{N}} 
= \emptyset$, and $\Gamma^{\mathrm{D}} \cup \Gamma^{\mathrm{N}} 
= \partial \Omega$. The time interval of interest is 
$\mathcal{I}$.

We consider the computational domain to be a rectangle 
with $L_x = 2$ and $L_y = 1$. The boundary is fixed on 
three sides, and is excited by a sinusoidal force of 
the following form on the other side:
\begin{align}
  f(\mathbf{x},t) = \left\{ \begin{array}{c} 
    f_0 \sin \left( \frac{2 \pi}{\tau_{\mathrm{load}}} t\right)
    \quad t \in \left[ 0 , \tau_{\mathrm{load}}\right] \\
    0 \quad \quad \quad \quad \quad t > \tau_{\mathrm{load}} 
  \end{array} \right. , 
  \quad \mathbf{x} \in \{ 0 \}\times[2L_y/5,3L_y/5]
\end{align}
A pictorial description of the problem is shown 
in Figure \ref{Fig:2D_Wave_Pictorial}. The domain 
is decomposed into two subdomain, as shown in 
Figure \ref{Fig:2D_Wave_DD}. In this numerical 
example, we have taken $u_0 = 0$, $v_0 = 0$, 
$c_0 = 1$, $f_0 = 5.0$ , and $\tau_{\mathrm{load}} 
= 0.1$. Figure \ref{Fig:2D_Wave_Case2} shows 
the result for explicit/implicit integration 
using the proposed coupling method. In this case,
$(\gamma_1,\beta_1) = (1/2,0)$, and $(\gamma_2,
\beta_2) = ( 1/2,1/4)$. The system time-step is 
$\Delta t = 10^{-4}$, subdomain time-steps are 
$\Delta t_1 = 10^{-5}$ and $\Delta t_2 = 10^{-4}$. 
As one can see from this figure, the proposed 
coupling method performed well. In particular, 
there are no spurious reflections at the subdomain 
interface, and there is no noticeable drift in 
the transverse displacement along the subdomain 
interface. 

This problem also clearly demonstrates that the 
proposed multi-time-step coupling method can be 
attractive for wave propagation problems. The 
coupling method is more cost effective than 
mere employing either an explicit scheme or 
an implicit scheme in the entire domain. In 
wave propagation problems 
involving fast dynamics, small time-steps are needed, 
and hence explicit schemes are typically employed. This 
will result in taking large number of time-steps to be 
able to carry out the numerical simulation to a desired 
final time. On the other hand, under the proposed coupling 
method, one can use explicit methods in the regions with 
fast dynamics (which typically occur near the loading), 
and use an implicit time-stepping scheme with a larger 
subdomain time-step in the other regions. 
For the chosen problem, if one has to employ 
an explicit scheme in the entire domain, the 
time-step should be smaller than the critical 
time-step of $1.36 \times 10^{-5}$. Under the 
proposed multi-time-step coupling method, 
the user can employ an explicit scheme with 
time-steps smaller than the critical time-step 
near the load, and an unconditionally stable,
implicit time-stepping scheme with larger time-steps 
in the rest of the computational domain.

%
%**************************************************;
%                                                  ;
%  NAME                                            ;
%    S10_Monolithic_Conclusions.tex                ;
%                                                  ;
%  WRITTEN BY                                      ;
%    Saeed Karimi                                  ;
%    Kalyana Babu Nakshatrala                      ;
%                                                  ;
%**************************************************;
\section{CONCLUDING REMARKS}
\label{Sec:S10_Monolithic_Conclusions}
We have developed a multi-time-step coupling method 
that can handle \emph{multiple subdomains} with 
different time-steps in different subdomains. The 
coupling method can couple implicit and explicit 
time-stepping schemes under the Newmark family 
even with disparate time-steps of more than two 
orders of magnitude in different subdomains.
%------------------------------------;
%  Summary on the energy properties  ;
%------------------------------------;
A systematic study on the energy preservation and energy 
properties of the proposed coupling method is presented, 
and the corresponding sufficient conditions are also 
derived. The proposed coupling method, in general, is 
not energy preserving. Despite claims in the literature, 
the quest for energy preserving multi-time-step coupling 
method is still on. 
One of the main conclusions of this paper is about 
the effect of system time-step and subcycling on 
the accuracy. It has been shown that accuracy can 
always be improved by decreasing system time-step. 
It is widely believed that lowering subdomain 
time-step keeping the system time-step will also 
improve the accuracy under a multi-time-stepping 
scheme. Using careful mathematical analysis and 
numerical results, we have shown that this popular 
belief is not always the case. To this end, a simple 
criterion is also proposed, which can predict whether 
subcycling will improve accuracy. The criterion is to 
monitor $\mathcal{E}_{\mathrm{interface}}^{(n \rightarrow n+1)}$ 
at every system time-step, which can be calculated on 
the fly during a numerical simulation. Subcycling is 
desirable if this quantity is small. 

%-----------------------------------------------------;
%  Summary on suitability of parallel implementation  ;
%-----------------------------------------------------;
The proposed multi-time-step coupling (which is a dual 
Schur domain decomposition technique) is well-suited for 
parallel computing. Specifically, one can utilize the 
advances made on the FETI method, which has shown to 
be scalable in a parallel setting for dual Schur domain 
decomposition methods \cite{Farhat_Roux_IJNME_1991_v32_p1205}.
%------------------------------------;
%  Summary on possible future works  ;
%------------------------------------;
There are several ways one could make advancements to the research 
presented in this paper. \emph{On the theoretical front}, a plausible 
future work is to perform a mathematical analysis on the numerical 
characteristics of the proposed multi-time-step coupling method on 
the lines of local error, propagation of error, and influence of 
perturbations. \emph{On the computational implementation front}, 
one could implement the proposed coupling method in a parallel 
setting and do a systematic study on its parallel performance. 
\emph{On the algorithmic front}, the next logical step is to 
extend the proposed multi-time-step coupling method to first-order 
transient systems, and eventually to fluid-structure interaction 
problems.  

%
%************************************************;
%                                                ;
%  NAME                                          ;
%    Monolithic_Appendix.tex                     ;
%                                                ;
%  WRITTEN BY                                    ;
%    Saeed Karimi                                ;
%    Kalyana Babu Nakshatrala                    ;
%                                                ;
%************************************************;
\section{APPENDIX: DERIVATION OF GOVERNING EQUATIONS FOR 
  CONSTRAINED MULTIPLE SUBDOMAINS}
\label{Sec:Monolithic_Appendix}
One can derive the constrained governing equations for second-order 
transient systems in several ways. For example, in Reference 
\cite{Prakash_Hjelmstad_IJNME_2004_v61_p2183} the equations 
have been derived using Lagrangian formalism. Herein, we shall 
derive the constrained governing equations for second-order 
systems using the Gauss principle of least constraint and 
Gibbs-Appell equations. Both these approaches can easily 
handle constraints of various types, including non-holonomic 
constraints and friction. Correctly, these two principles are 
considered more fundamental than Lagrangian or Hamiltonian 
formalisms, and are particularly ideal to deal with constraints 
\cite{Udwadia_Kalaba}. We briefly outline these two powerful 
approaches with a hope that the future works on multi-time-step 
coupling methods will utilize either the Gauss principle of 
least constraint or the Gibbs-Appell equations.

%===================================================;
%  Subsection: Gauss principle of least constraint  ;
%===================================================;
\subsection{Gauss principle of least constraint}
The statement of Gauss principle of least constraint for 
a system of $N$ subdomains takes the following form:
%------------------------------------------;
%  Equation: Gauss principle minimization  ;
%------------------------------------------;
\begin{align}
  &\mathop{\mathrm{minimize}}_{\ddot{\boldsymbol{u}}_1, 
    \ddot{\boldsymbol{u}}_2, \cdots, \ddot{\boldsymbol{u}}_S} \quad  
  \frac{1}{2}\sum_{i = 1}^{S} \left(\ddot{\boldsymbol{u}}_{i} 
  - \boldsymbol{M}_{i}^{-1} \boldsymbol{h}_{i}\left(\dot{\boldsymbol{u}}_{i},
  \boldsymbol{u}_{i} , t\right)\right)^{\mathrm{T}} \boldsymbol{M}_{i} 
  \left(\ddot{\boldsymbol{u}}_{i} - \boldsymbol{M}_{i}^{-1} \boldsymbol{h}_{i}
  \left(\dot{\boldsymbol{u}}_{i},\boldsymbol{u_{i}} , t\right)\right) \\
  &\mbox{subject to} \quad \sum_{i=1}^{S} \boldsymbol{C}_i 
  \ddot{\boldsymbol{u}}(t) = \boldsymbol{0}
\end{align}	
Note that in a continuum setting, continuity of displacements, 
velocities and accelerations are all equivalent. The first-order 
optimality condition gives rise to the following equations:
\begin{subequations}
  \begin{align}
    \label{Eqn:Monolithic_subdomain}
    &\boldsymbol{M}_i \ddot{\boldsymbol{u}}_i = \boldsymbol{h}_i
    (\dot{\boldsymbol{u}}_i, \boldsymbol{u}_i, t) + 
    \boldsymbol{C}_i^{\mathrm{T}} \boldsymbol{\lambda} \\
    \label{Eqn:Monolithic_constraint}
    &\sum_{i=1}^{S} \boldsymbol{C}_i \ddot{\boldsymbol{u}} = \boldsymbol{0}
  \end{align}	
\end{subequations}
where $\boldsymbol{\lambda}$ is the vector of Lagrange 
multipliers for enforcing the constraint. Equations 
\eqref{Eqn:Monolithic_subdomain}--\eqref{Eqn:Monolithic_constraint} 
are the governing equations for constrained multiple subdomains. 	
In the case of linear elastic second-order systems, the vector 
$\boldsymbol{h}_i$ is given by the following expression:
\begin{align}
\boldsymbol{h}_i(\dot{\boldsymbol{u}}_i,\boldsymbol{u}_i,t) = 
-\boldsymbol{K}_i \boldsymbol{u}_i + \boldsymbol{f}_i^{\mathrm{ext}}(t)
\end{align}
It is noteworthy that the question whether the vector $\boldsymbol{h}_i$ 
can depend on acceleration is still a matter of debate among mechanicians. 
For further details on this issue see \cite{Pars_Dynamics,1998_Chen_JFI_v335_p871,1996_Leech_TM_v16_p221,2007_Zhechev_JAM_v74_p107}. 

%======================================;
%  Subsection: Gibbs-Appell equations  ;
%======================================;
\subsection{Gibbs-Appell equations}
The statement of Gibbs-Appell principle for a system of $N$ subdomains 
can be written as follows:
\begin{subequations}
  \begin{align}
    &\mathop{\mathrm{minimize}}_{\ddot{\boldsymbol{u}}_1,\ddot{\boldsymbol{u}}_2,\cdots,\ddot{\boldsymbol{u}}_S} 
    \quad \sum_{i=1}^{S} \left(\frac{1}{2} {\ddot{\boldsymbol{u}_i}}^{\mathrm{T}} 
    \boldsymbol{M}_i {\ddot{\boldsymbol{u}}}_i - 
               {\ddot{\boldsymbol{u}}_i}^{\mathrm{T}} 
               \boldsymbol{h}_i\right) \\
               &\mbox{subject to} \quad \sum_{i=1}^{S} \boldsymbol{C}_i 
               \ddot{\boldsymbol{u}}_i = \boldsymbol{0}
  \end{align}
\end{subequations}
The first-order optimality condition of the above constrained optimization 
problem gives rise to the same governing equations for constrained 
multiple subdomains (given by equations 
\eqref{Eqn:Monolithic_subdomain}--\eqref{Eqn:Monolithic_constraint}).

%============================;
%  Section: Acknowledgments  ;
%============================;
\section*{ACKNOWLEDGMENTS}
The authors acknowledge the support of the National 
Science Foundation under Grant no. CMMI 1068181. 
The opinions expressed in this paper are those 
of the authors and do not necessarily reflect 
that of the sponsor(s). 

%================;
%  Bibliography  ;
%================;
\bibliographystyle{unsrt}
\bibliography{Master_References/Master_References,Master_References/Books}

%================================;
%  Include all the figures here  ;
%================================;
%-------------------------------;
%  Figure: Multiple subdomains  ;
%-------------------------------;
\begin{figure}
  \centering
  \psfrag{w1}{$\Omega_1$}
  \psfrag{w2}{$\Omega_2$}
  \psfrag{wS}{$\Omega_S$}
  \psfrag{l}{$\boldsymbol{\lambda}$}
  \psfrag{CM}{Conforming mesh interface nodes}
  \psfrag{ir}{Interface interaction forces}
  \includegraphics[scale=0.65]{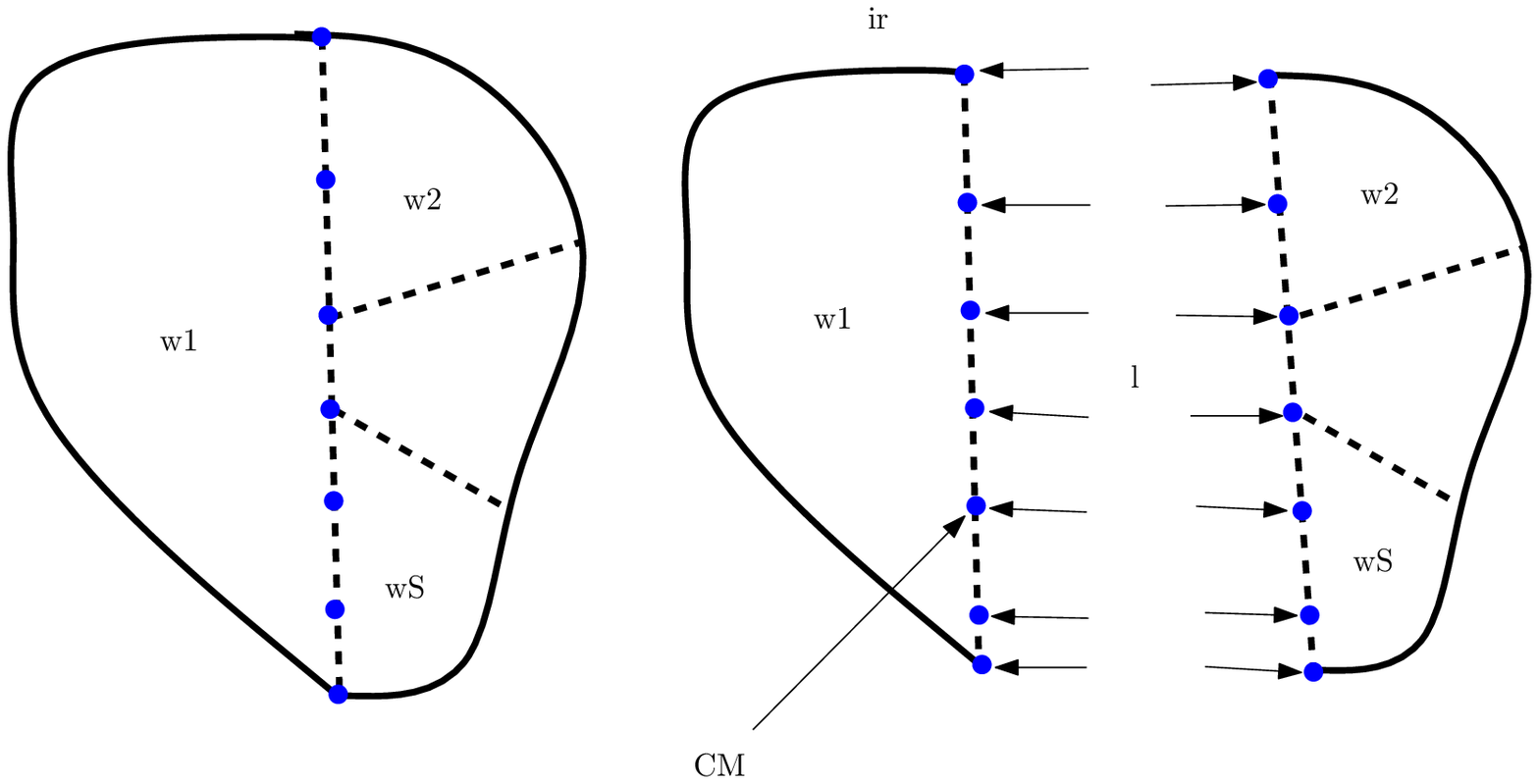}
  \caption{A pictorial description of multiple subdomains. The 
  domain $\Omega$ is decomposed into $S$ subdomains, which are 
  denoted by $\Omega_1, \cdots, \Omega_S$. The subdomain interface 
  is indicated using dashed curves. The mesh is assumed to be 
  conforming along the subdomain interface.  
  \label{Fig:Monolithic_subdomains}}
\end{figure}

%------------------------------------;
%  Figure: Multi-time-step notation  ;
%------------------------------------;
\begin{figure}
  \psfrag{A}{$A$}
  \psfrag{B}{$B$}
  \psfrag{tn}{$t_{n}$}
  \psfrag{tn1}{$t_{n+1}$}
  \psfrag{dt}{$\Delta t$}
  \psfrag{dtf}{$\Delta t_A$}
  \psfrag{dts}{$\Delta t_B$}
  \psfrag{information}{information exchange}
  \psfrag{enforcement}{enforcement of compatibility conditions}
  \includegraphics[scale=0.55]{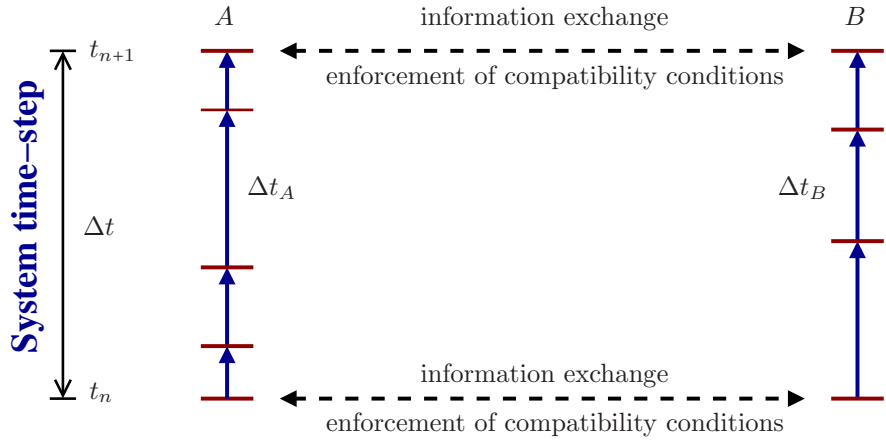}
  \caption{A pictorial description of time levels ($t_n$), 
    system time-step $(\Delta t)$, subdomain time-step 
    $(\Delta t_i)$, and subcycling. Note that $\eta_i = 
    \Delta t / \Delta t_i$. In this figure $i = A\; 
    \mbox{or}\; B$.
    \label{Fig:Monolithic_multi_time_step_notation}}
\end{figure}

%------------------------------;
%  Figure: SDOF 2 sub-domains  ;
%------------------------------;
\begin{figure}
  \centering
  \psfrag{ma}{$m_A$}
  \psfrag{mb}{$m_B$}
  \psfrag{ka}{$k_A$}
  \psfrag{kb}{$k_B$}
  \psfrag{fa}{$f_A$}
  \psfrag{fb}{$f_B$}
  \includegraphics[scale=0.7]{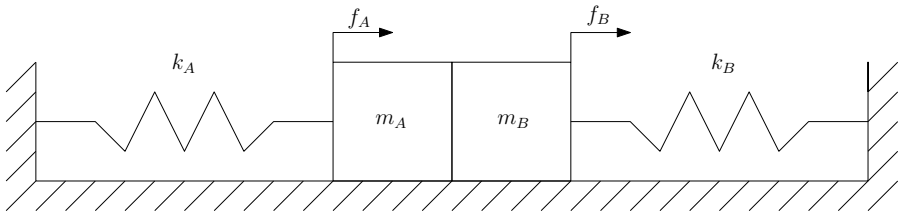}
  \caption{A pictorial description of the 
    split degree-of-freedom (SDOF) lumped 
    parameter system. The masses $A$ and 
    $B$ slide on a frictionless surface. 
    \label{fig:SDOF2}}
\end{figure}

%---------------------------------------------------;
%  Figure: Displacement, velocity and acceleration  ;
%---------------------------------------------------;
\begin{figure}
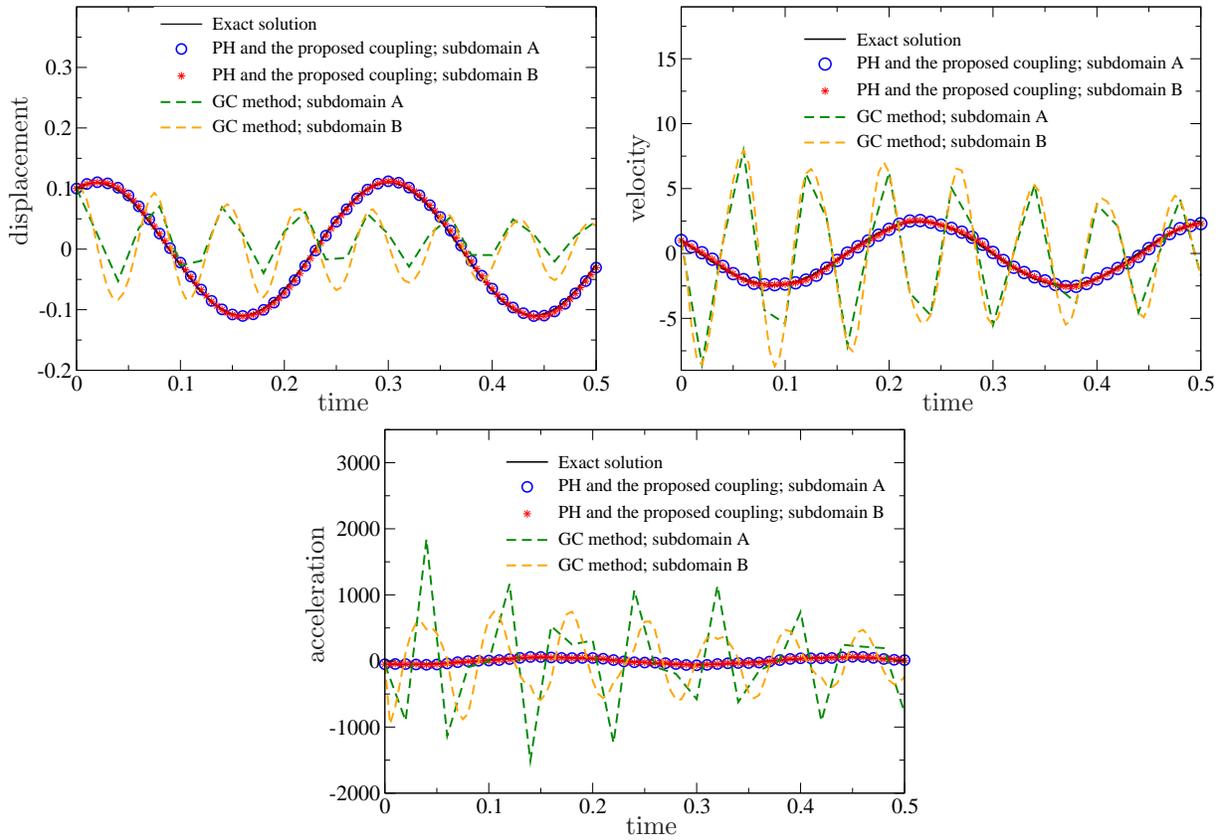

\centering
\psfrag{time}{time}
\psfrag{Analytic}{Exact}
\psfrag{displacement}{displacement}
\psfrag{displacement}{displacement}
\psfrag{velocity}{velocity}
\psfrag{acceleration}{acceleration}
\subfigure{
  \includegraphics[scale=0.32,clip]{Figures/SDOF2/disp.eps}}
\subfigure{
  \includegraphics[scale=0.32,clip]{Figures/SDOF2/vel.eps}}
\subfigure{
  \includegraphics[scale=0.32,clip]{Figures/SDOF2/acc.eps}}
\caption{SDOF lumped parameter system:~
  This figure compares the performance of the 
  proposed coupling method with that of the 
  GC and PH methods. The top, middle and 
  bottom subfigures, respectively, show 
  the displacement, velocity and acceleration. 
  It should be noted that, in the absence of 
  external forces, the proposed coupling method 
  and the PH method should produce the same 
  results, which is the case in this figure.
  The numerical results under the proposed 
  method match well with the analytical 
  solution. Note the rapid (unphysical) 
  decay under the GC method compared to 
  the other two methods.
  \label{Fig:SDOF_kinematic}}
\end{figure}

%	Lambda
\begin{figure}
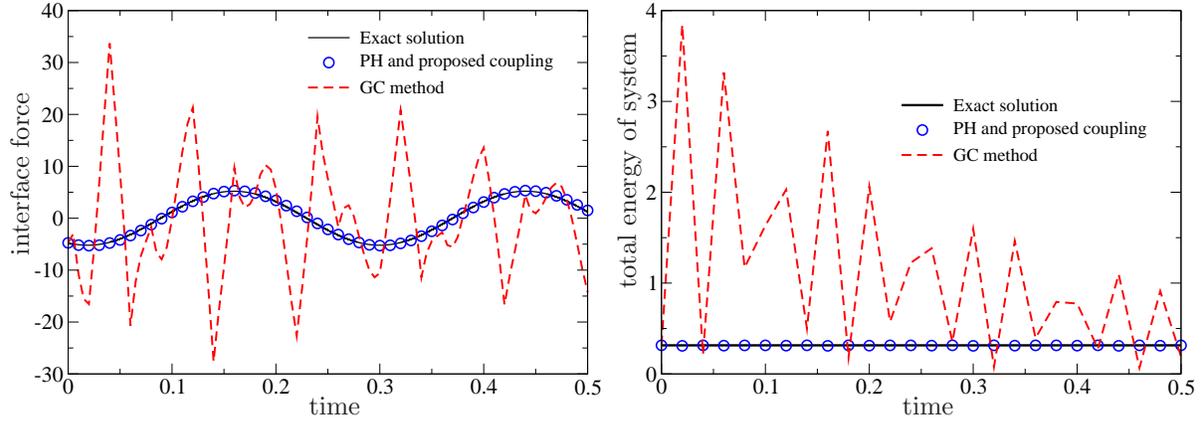

\centering
\psfrag{time}{time}
\psfrag{interface force}{interface force}
\psfrag{total energy of system}{total energy of system}
\subfigure{
  \includegraphics[scale=0.32,clip]{Figures/SDOF2/lam.eps}}
\subfigure{
  \includegraphics[scale=0.32,clip]{Figures/SDOF2/enrg.eps}}

\caption{SDOF lumped parameter system: The 
  top and bottom subfigures, respectively, 
  show the interface force and total energy 
  of the system. The numerical results under 
  the GC method do not match with the analytical 
  solution. 
  It should be noted that the total energy of 
  the system should be constant because the 
  system is elastic and the external force 
  is zero. As expected the GC method exhibits 
  decay in total energy. Although the proposed 
  coupling method and the PH method do not 
  preserve energy, they give close results 
  to the exact value for the chosen parameters. 
  However, this need not be the case if one 
  chooses a different time-stepping scheme. 
  \label{Fig:SDOF_lambda_energy}}
\end{figure}

%-------------------------------;
%  Figure: Energy preservation  ;
%-------------------------------;
\begin{figure}[h]
  \includegraphics[scale=0.45,clip]{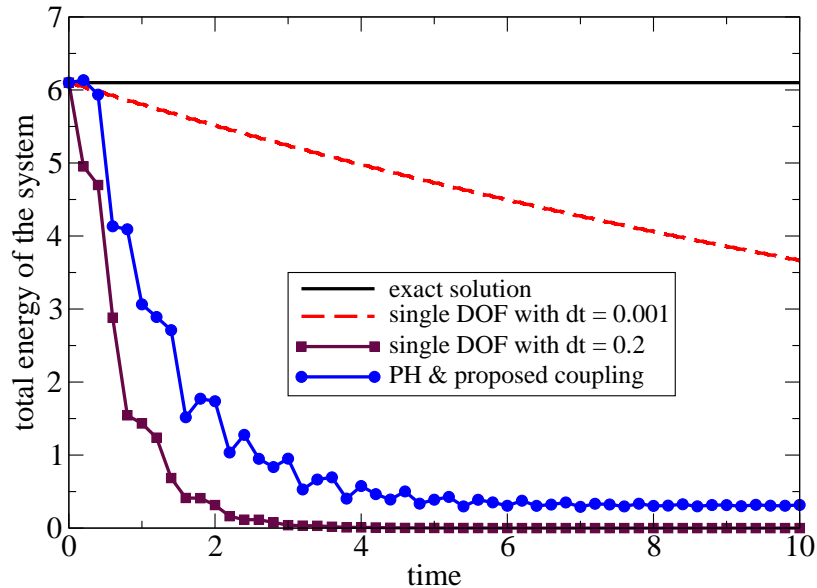}
  \caption{The problem parameters are $m_A = 1$, 
    $m_B = 1$, $k_A = 1000$, $k_B = 20$, $f_A = 0$, 
    $f_B = 0$, $\Delta t_A = 0.2$ and $\Delta t_B 
    = 0.001$. The Newmark parameters are taken 
    to be $(\beta = 0.3025, \gamma = 0.6)$. For 
    comparison, numerical solutions for single 
    degree of freedom (i.e., without splitting) 
    are also presented for two different time-steps 
    $\Delta t = 0.2$ and $\Delta t = 0.001$. It is 
    evident that the PH and proposed multi-time-step 
    couplings do not preserve energy in the first 
    sense. \label{Fig:Monolithic_energy_preservation}}
\end{figure}

%-------------------------------;
%  Figure: Energy conservation  ;
%-------------------------------;
\begin{figure}[h]
  \includegraphics[scale=0.4,clip]{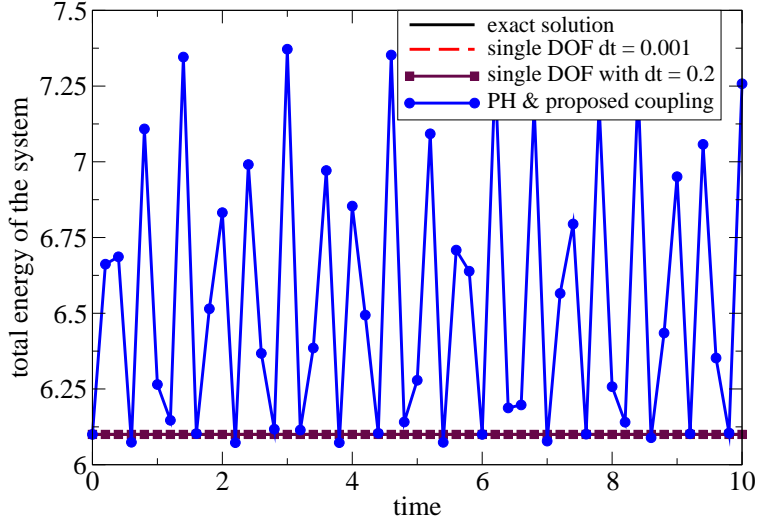}
  \caption{This figure illustrates that the proposed coupling 
    method does not conserve energy if there is subcycling. 
    The Newmark average acceleration method $(\beta = 0.25, 
    \gamma = 0.5)$ is employed in both subdomains.
  The problem parameters are $m_A = 1$, $m_B = 1$, $k_A = 1000$, 
  $k_B = 20$, $f_A = 0$, $f_B = 0$, $\Delta t_A = 0.2$ and $\Delta t_B 
  = 0.001$. For comparison, numerical solutions for single degree of 
  freedom (i.e., without splitting) are also presented for two different 
  time steps $\Delta t = 0.2$ and $\Delta t = 0.001$. This numerical 
  example can also serve to illustrate that the multi-time-step coupling 
  method preserves energy neither in the first sense nor in the second 
  sense. \label{Fig:Monolithic_multi_time_step_energy_plots}}
\end{figure}

%--------------------------;
%  Figure: Backward Euler  ;
%--------------------------;
\begin{figure}
  \centering
  \psfrag{energy}{energy}
  \psfrag{time}{time}
  \includegraphics[scale=0.37,clip]{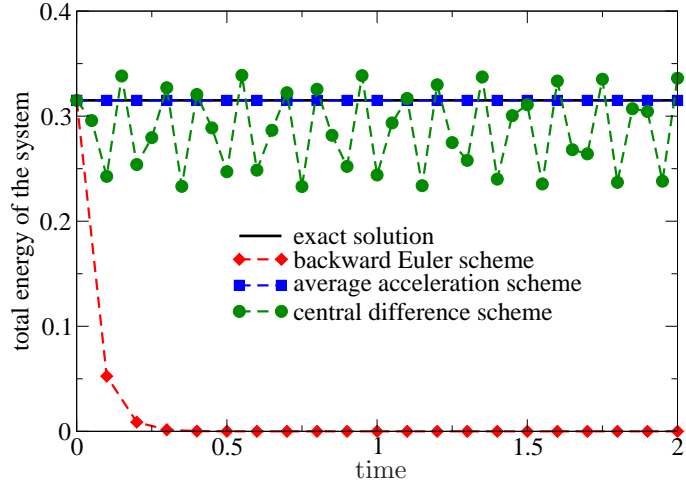}
  \caption{Coupling using the backward Euler scheme: 
    The second-order differential-algebraic equations 
    is converted to first-order differential-algebraic 
    equations, and the resulting system is solved using 
    the backward Euler scheme, which is the simplest 
    member of backward difference formulae (BDF) and 
    implicit Runge-Kutta (IRK) schemes. 
    BDF and IRK schemes are frequently employed to 
    solve differential-algebraic equations. As one 
    can see from the figure, the numerical solution 
    using the backward Euler is highly dissipative. 
    On the other hand, the proposed multi-time-step 
    coupling method based on Newmark family of time 
    integrators has better performance with respect 
    to the total energy of the system, which is an invariant 
    for the problem at hand.
    The proposed coupling method was employed to seek
    the numerical solution using Newmark average
    acceleration and central difference schemes with
    no subcycling. In the case of Newmark average 
    acceleration scheme, $\Delta t = 0.1$. A time-step
    of $\Delta t = 0.05$ was adopted for central difference
    scheme. For the backward Euler method $\Delta t = 0.1$.  
    \label{Fig:Monolithic_energy_BE}}
\end{figure}

%----------------------------------;
%  Figure: SDOF 3 sub-domains NRs  ;
%----------------------------------;
\begin{figure}
  \centering
  \psfrag{lAB}{$\lambda_{AB}(t)$}
  \psfrag{lBC}{$\lambda_{BC}(t)$}
  \psfrag{kA}{$k_A$}
  \psfrag{kB}{$k_B$}
  \psfrag{kC}{$k_C$}
  \psfrag{mA}{$m_A$}
  \psfrag{mB}{$m_B$}
  \psfrag{mC}{$m_C$}
  \psfrag{fA}{$f_A(t)$}
  \psfrag{fB}{$f_B(t)$}
  \psfrag{fC}{$f_C(t)$}
  \subfigure{
    \includegraphics[scale=0.75]{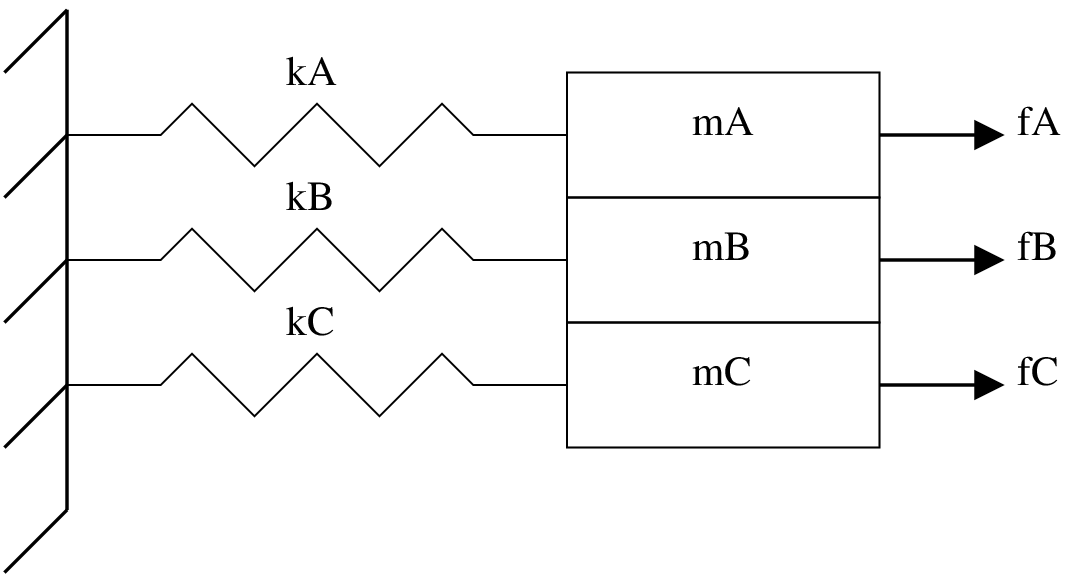}}
  \subfigure{
    \includegraphics[scale=0.75]{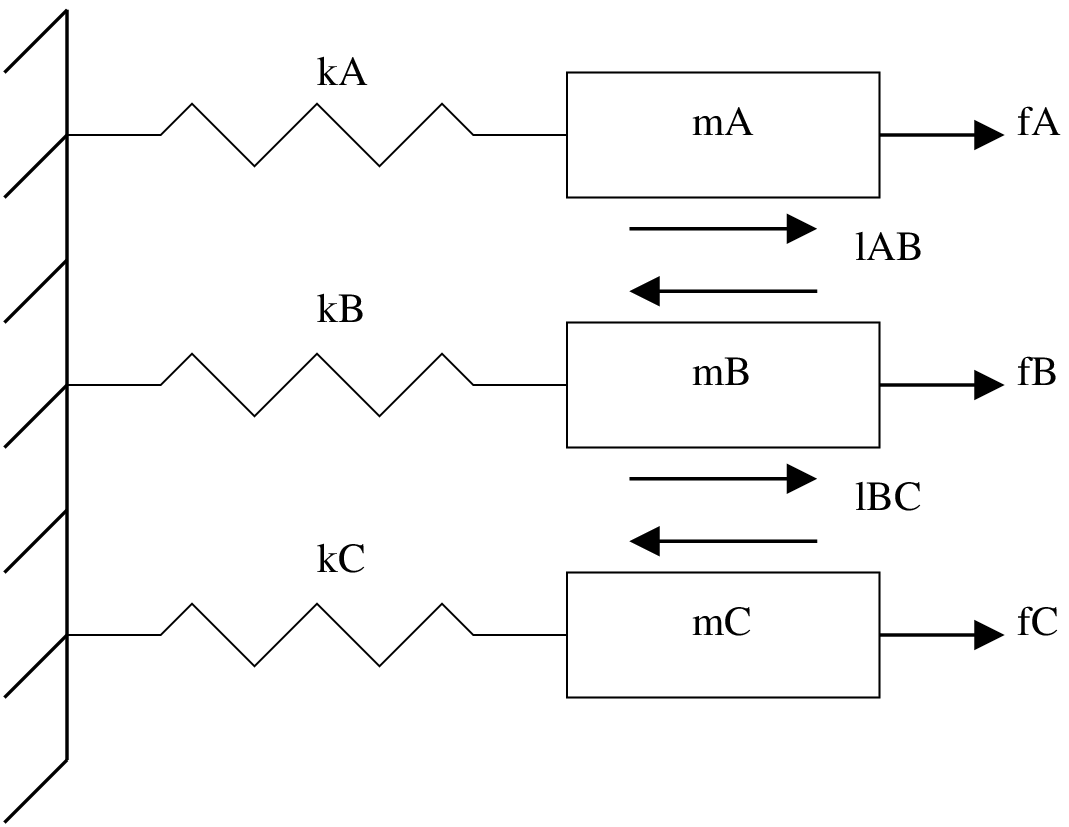}}
  \caption{Split degree-of-freedom with three subdomains: 
    The problem is solved using the proposed multi-time-step 
    coupling method. The problem parameters are $m_A = 5$, 
    $m_B = 0.1$, $m_C = 0.01$, $k_A = 5$, $k_B = 2.5$ and 
    $k_C = 4$. The subdomain time-steps are taken as $\Delta 
    t_A = 0.01$, $\Delta t_B = 0.005$ and $\Delta t_C = 
    0.0025$. Newmark parameters are $\beta_A = \beta_B = 
    \beta_C = 0.25$ and $\gamma_A = \gamma_B = \gamma_C = 0.5$.
      \label{fig:SDOF3}}
\end{figure}

%---------------------------------------------------------------------------;
%  Figure: SDOF three subdomains: Displacement, velocity, and acceleration  ;
%---------------------------------------------------------------------------;
\begin{figure}
  \centering
  \subfigure{
    \includegraphics[scale=0.32,clip]{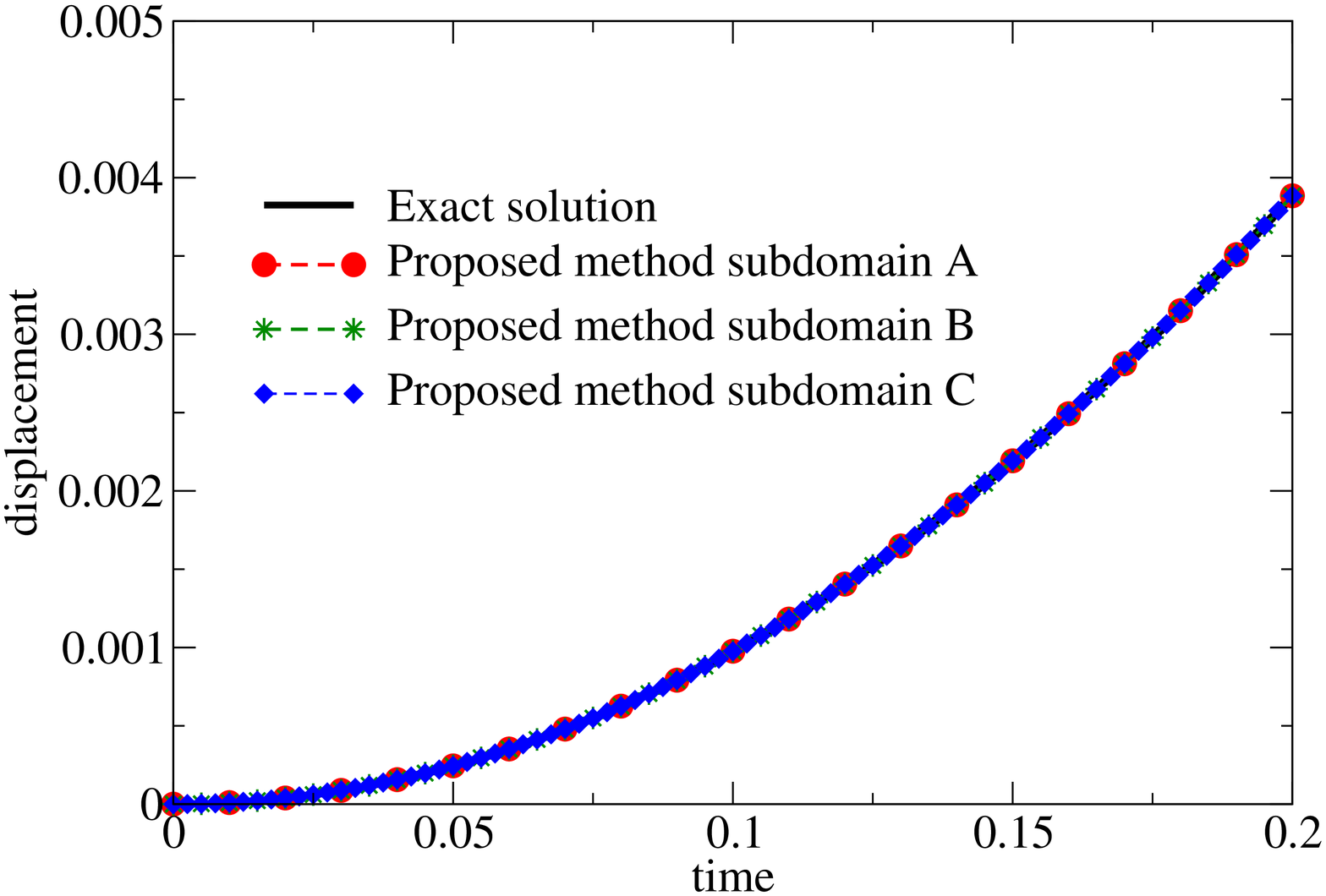}}
  \subfigure{
    \includegraphics[scale=0.32,clip]{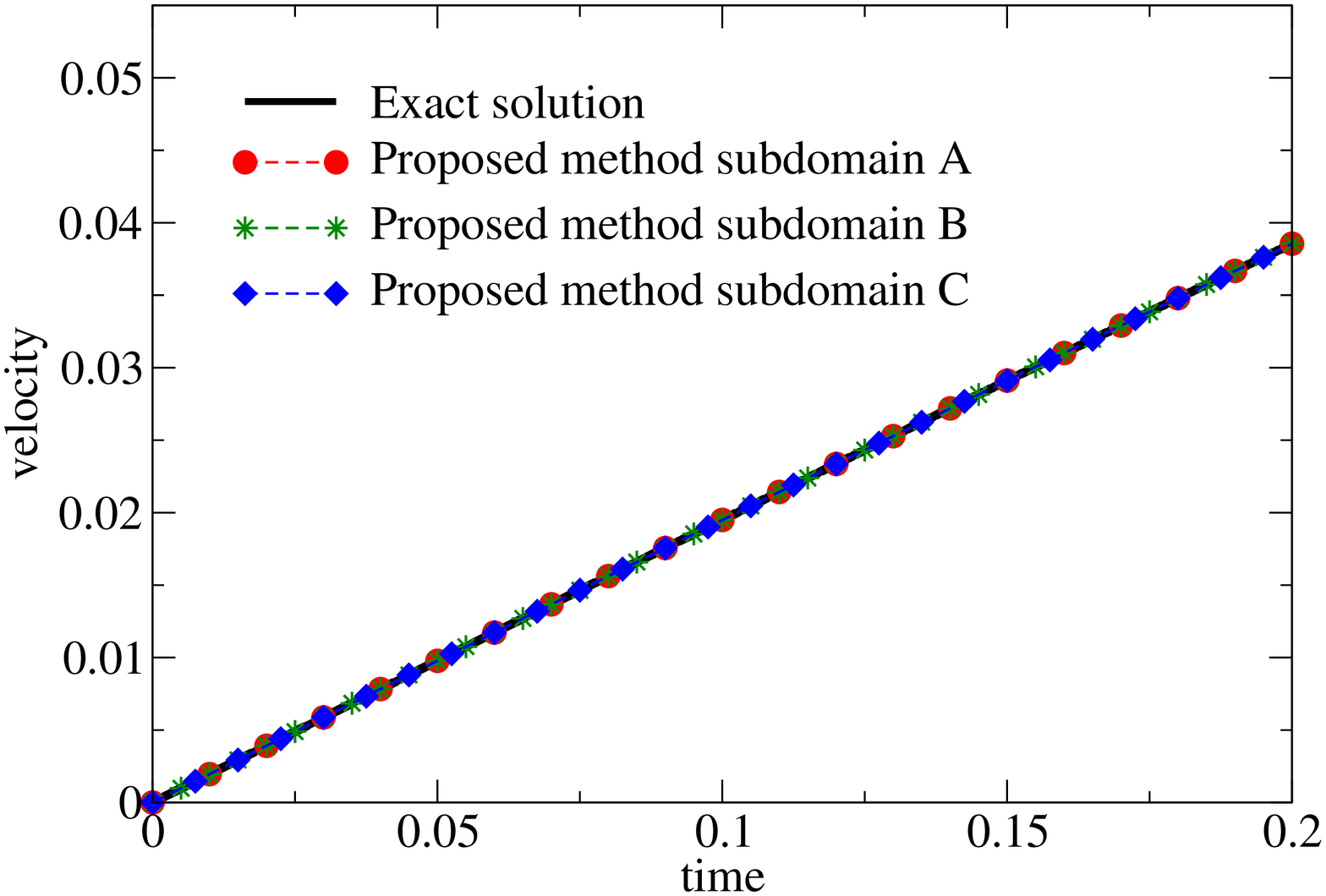}}
  \subfigure{
    \includegraphics[scale=0.32,clip]{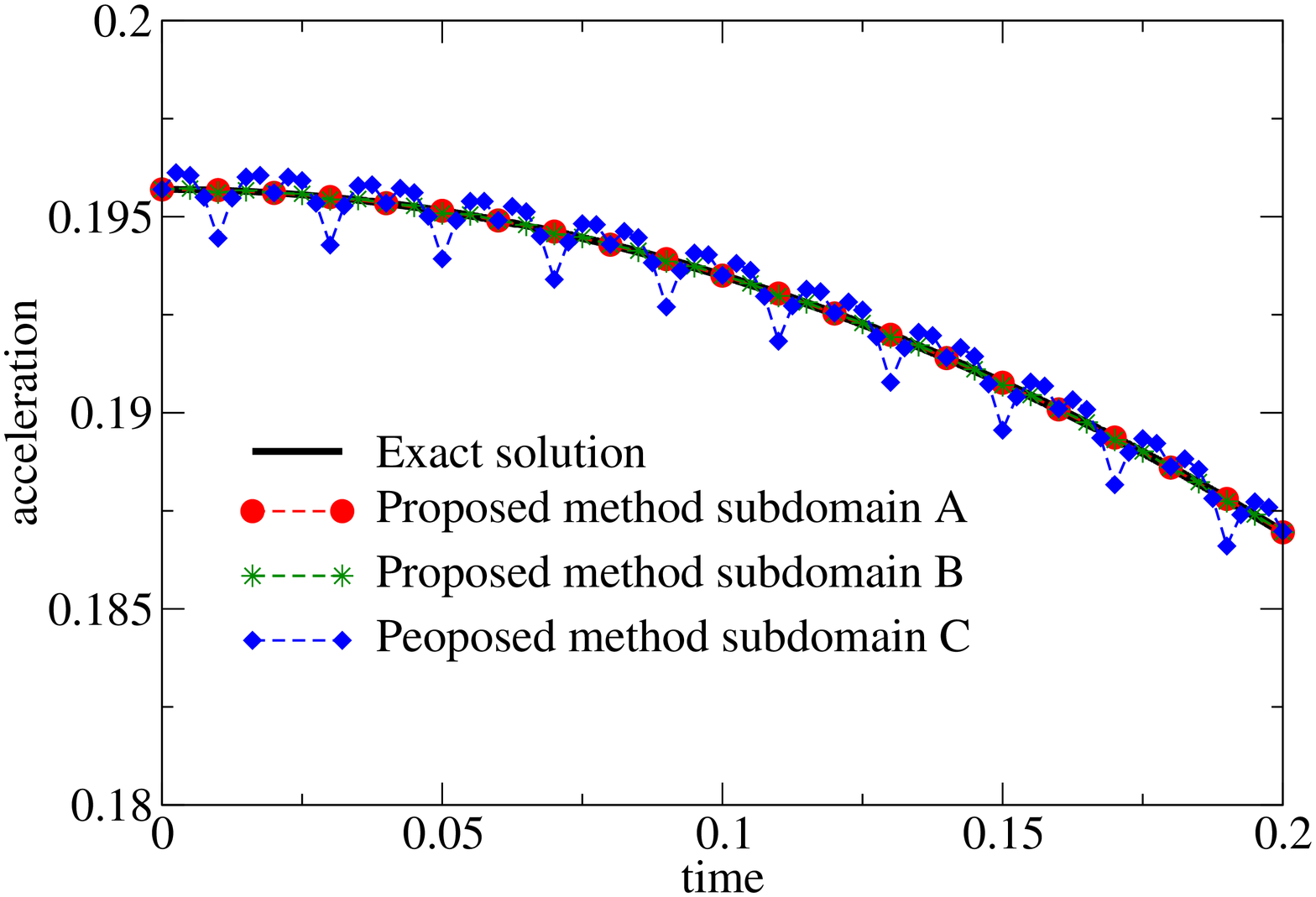}}
  \caption{Split degree-of-freedom with three subdomains: 
    Numerical and analytical results for displacement in 
    problem 2 is shown in this figure. As seen here, the 
    numerical results under the proposed coupling method 
    matches well with the exact values. 
    \label{Fig:Monolithic_3SDOF_kinematic}}
\end{figure}

%-----------------------------------------------------;
%  Figure: SDOF three subdomains: Lambdas and energy  ;
%-----------------------------------------------------;
\begin{figure}
  \centering
  \psfrag{time}{time}
  \psfrag{energy}{energy}
  \psfrag{interface force}{interface force}
  \subfigure{
  \includegraphics[scale=0.32,clip]{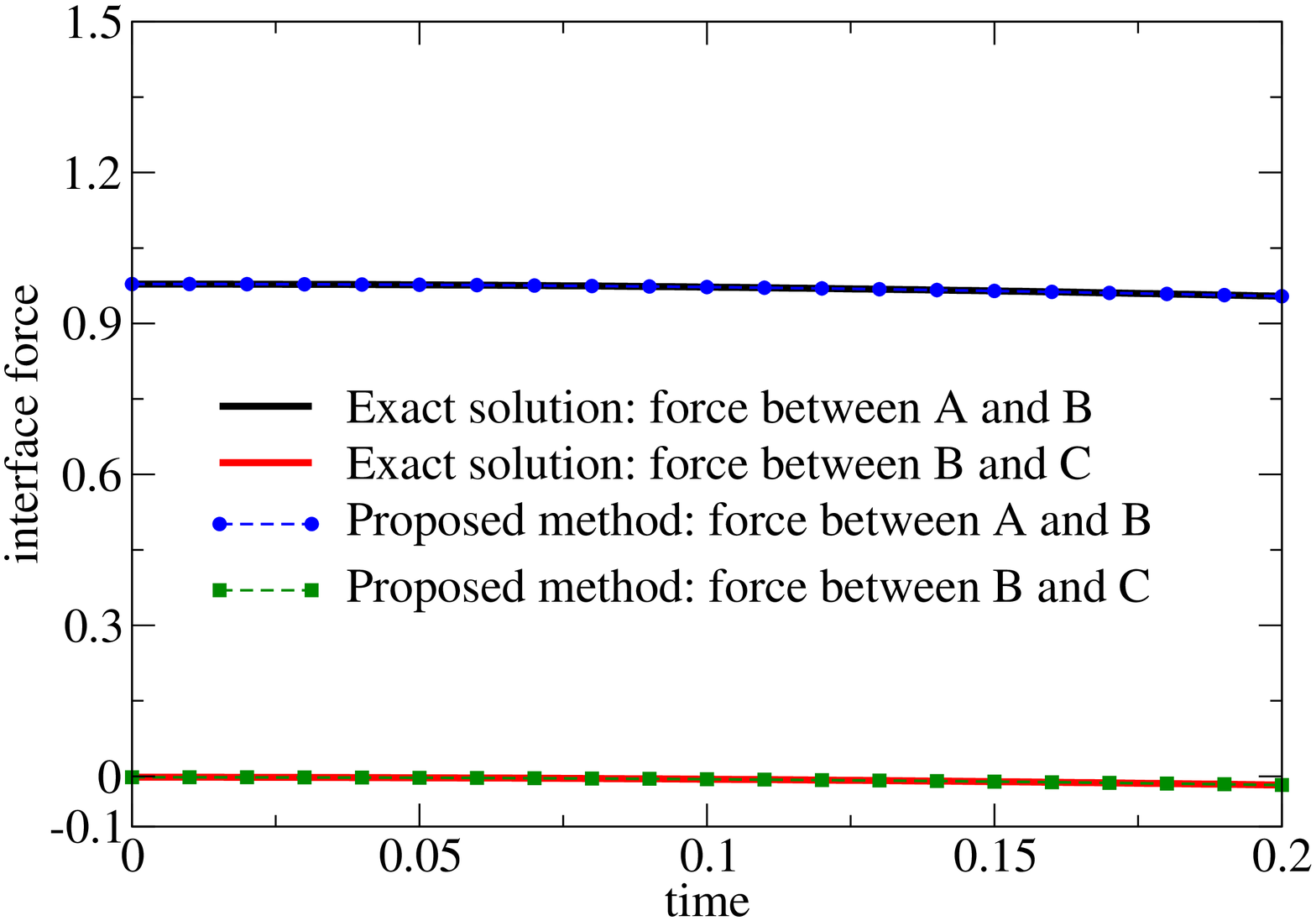}}
  \subfigure{
  \includegraphics[scale=0.32,clip]{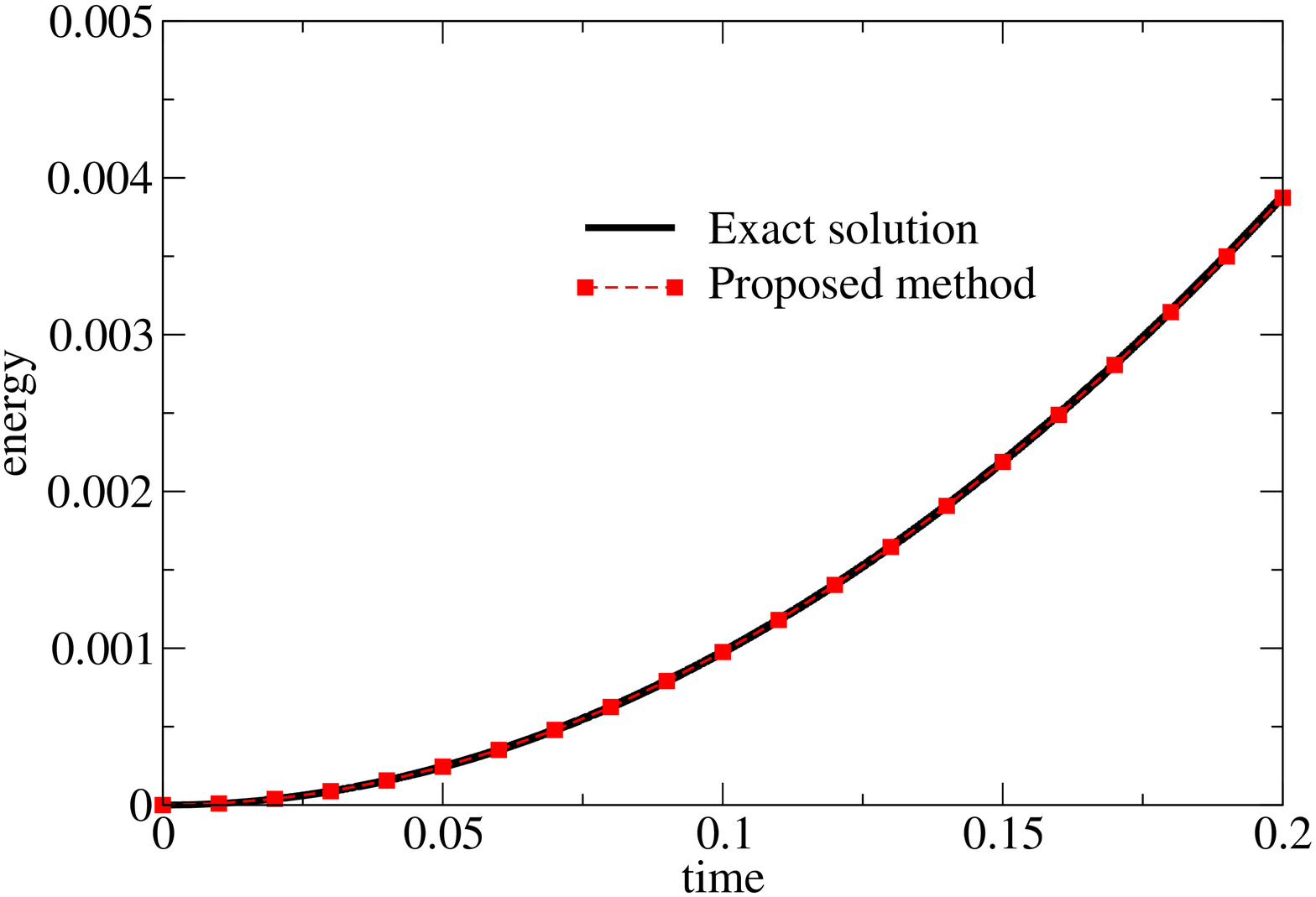}}
  \caption{Split degree-of-freedom with three subdomains: 
    The top figure shows the interface reaction forces 
    $\lambda_{AB}$ and $\lambda_{BC}$ with respect to 
    time. This bottom figure shows the total energy 
    of the system with respect to time. The numerical 
    results under the proposed coupling method match 
    well with the analytical solution. 
    \label{Fig:Monolithic_3SDOF_lambdas_energy}}
\end{figure}

%====================================;
%  Figures: One-dimensional problem  ;
%====================================;

%----------------------------------;
%  Figure: 1D bar, 3 sub-domains   ;
%----------------------------------;
\begin{figure}
	\centering
  \psfrag{SA}{subdomain $A$}
  \psfrag{SB}{subdomain $B$}
  \psfrag{SC}{subdomain $C$}
  \psfrag{P}{$P(t)$}
  \psfrag{LA}{$L_A = 1/3$}
  \psfrag{LB}{$L_B = 1/3$}
  \psfrag{LC}{$L_C = 1/3$}
   \includegraphics[scale=0.7]{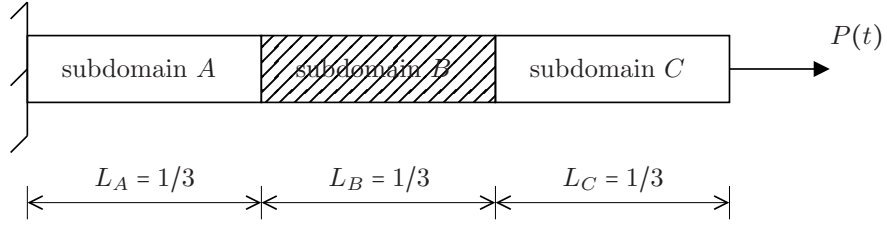}
  \caption{One-dimensional problem with homogeneous 
    properties:~Consider an axial elastic bar of unit 
    length. The left end of the bar is fixed, and a 
    constant load of $P(t) = 10$ is applied to the 
    right end of the bar. The initial displacement 
    and the initial velocity are both assumed to be 
    zero. The proposed coupling method is employed to 
    solve this problem by decomposing the computational 
    domain into three subdomains, which are denoted by 
    $A$, $B$ and $C$. Each subdomain is meshed 
    with two-node line elements. The left and right 
    subdomains employ Newmark average acceleration scheme 
    (i.e., $\beta_A = \beta_C = 1/4$ and $\gamma_A = 
    \gamma_C = 1/2$), and the middle subdomain employs 
    the central difference scheme (i.e., $\beta_B = 0$ 
    and $\gamma_B = 1/2$). 
    \label{Fig:1_dimensional_mainfig}}
\end{figure}

\clearpage 

%===================================;
%         SET 1 : 5-5-5             ;
%===================================;

%-------------------------------------------------------------;
%  Figure: One-dimensional homogeneous bar: Tip displacement  ;
%-------------------------------------------------------------;
\begin{figure}
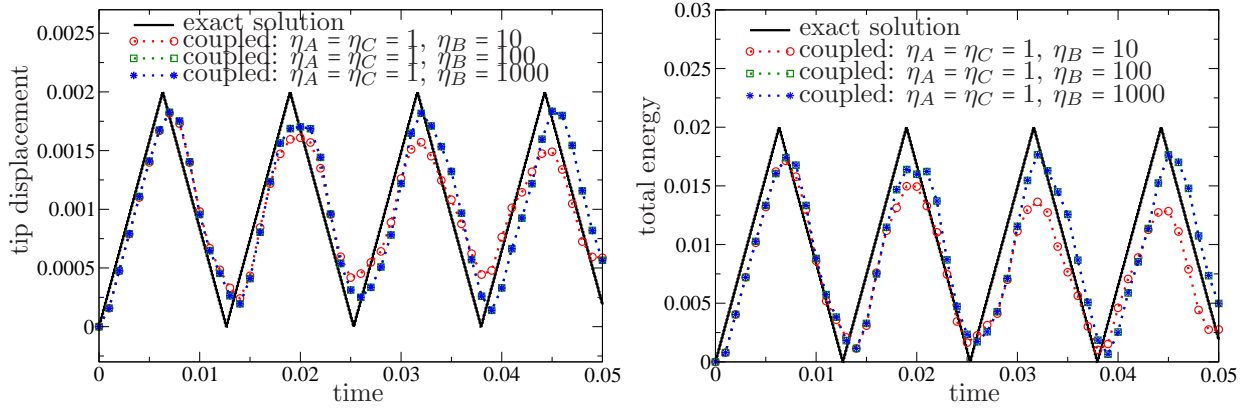

  \centering
  \psfrag{time}{time}
  \psfrag{exact}{exact solution}
  \psfrag{total energy}{total energy}
  \psfrag{tip displacement}{tip displacement}
  \psfrag{r1000}{coupled: $\eta_A = \eta_C = 1,\;\eta_B = 1000$}
  \psfrag{r10}{coupled: $\eta_A = \eta_C = 1, \; \eta_B = 10$}
  \psfrag{r100}{coupled: $\eta_A = \eta_C = 1, \; \eta_B = 100$}
  \subfigure{
  \includegraphics[scale=0.31,clip]{Figures/1D_3Subdomain/5-5-5/d_tip.eps}}
  \subfigure{
  \includegraphics[scale=0.31,clip]{Figures/1D_3Subdomain/5-5-5/enrg.eps}}
  \caption{One-dimensional problem with homogeneous 
  properties:~The top and bottom figures, respectively, 
  show the tip displacement and total energy as a 
  function of time. 
    The left and right subdomains employ Newmark average 
    acceleration scheme (which is an implicit scheme), 
    and the middle subdomain employs the central difference 
    scheme (which is an explicit scheme). The system 
    time-step is fixed and is taken as $\Delta t = 10^{-3}$. 
    The subdomain time-steps in the left and right subdomains 
    are chosen to be equal to the system 
    time-step. The time-step in the middle subdomain is 
    varied, and three different values are employed, which 
    are given by $\eta_B := \Delta t / \Delta t_B = 10, 100, 
    1000$. 
    \emph{The figure clearly shows that, under a fixed 
      system time-step, the accuracy can be improved by 
      employing subcycling in the subdomains under the 
      proposed coupling method. This implies that the 
      time-step required for the explicit scheme need 
      not limit the time-step in the entire computational 
      domain under the proposed multi-time-step coupling 
      method.} \label{Fig:1_dimensional_dtip_energy}}
\end{figure}

%--------------------------------------------------;
%  Figure: One-dimensional homogeneous bar: drift  ;
%--------------------------------------------------;
\begin{figure}
  \centering
  \psfrag{time}{time}
  \psfrag{drift AB}{$d_{\mbox{drift}}$ at AB interface}
  \psfrag{drift BC}{$d_{\mbox{drift}}$ at BC interface}
  \psfrag{r1000}{coupled: $\eta_A = \eta_C = 1,\; \eta_B = 1000$}
  \psfrag{r10}{coupled: $\eta_A = \eta_C = 1, \; \eta_B = 10$}
  \psfrag{r100}{coupled: $\eta_A = \eta_C = 1, \; \eta_B = 100$}
  \subfigure{
    \includegraphics[clip,scale=0.31]{Figures/1D_3Subdomain/5-5-5/driftAB.eps}}
  \subfigure{
    \includegraphics[clip,scale=0.31]{Figures/1D_3Subdomain/5-5-5/driftBC.eps}}
  \caption{One-dimensional problem with homogeneous properties:  
    The top figure shows the drift in the displacement at the 
    interface of subdomains $A$ and $B$. The bottom figure 
    shows the drift in displacement at the interface of 
    subdomains $B$ and $C$. At least for this problem, 
    one can conclude that there is no appreciable drift 
    in the displacements under the proposed coupling 
    method. \label{Fig:1_dimensional_drift}}
\end{figure}

%------------------------------------------------;
%  Figure: One-dimensional homogeneous bar: LAB  ;
%------------------------------------------------;
\begin{figure}
  \centering
  \psfrag{time}{time}
  \psfrag{lambda AB}{$\lambda_{AB}$}
  \psfrag{lambda BC}{$\lambda_{BC}$}
  \psfrag{exact}{exact solution}
  \psfrag{r1000}{proposed method: $\eta_A = \eta_C = 1,\; \eta_B = 1000$}
  \psfrag{r10}{proposed method: $\eta_A = \eta_C = 1, \; \eta_B = 10$}
  \psfrag{r100}{proposed method: $\eta_A = \eta_C = 1, \; \eta_B = 100$}
  \subfigure{
    \includegraphics[clip,scale=0.4]{Figures/1D_3Subdomain/5-5-5/L-AB.eps}}
  \subfigure{
    \includegraphics[clip,scale=0.4]{Figures/1D_3Subdomain/5-5-5/L-BC.eps}}
  \caption{One-dimensional problem with homogeneous 
    properties: The top figure shows the interface 
    force between subdomains $A$ and $B$. The bottom 
    figure shows the interface force between subdomains 
    $B$ and $C$. \label{Fig:1_dimensional_lambdas}}
\end{figure}

%===================================;
%         SET 2 : 5-10-5            ;
%===================================;

%--------------------------------------------------------;
%  Figure: One-dimensional: Tip displacement and energy  ;
%--------------------------------------------------------;
\begin{figure}
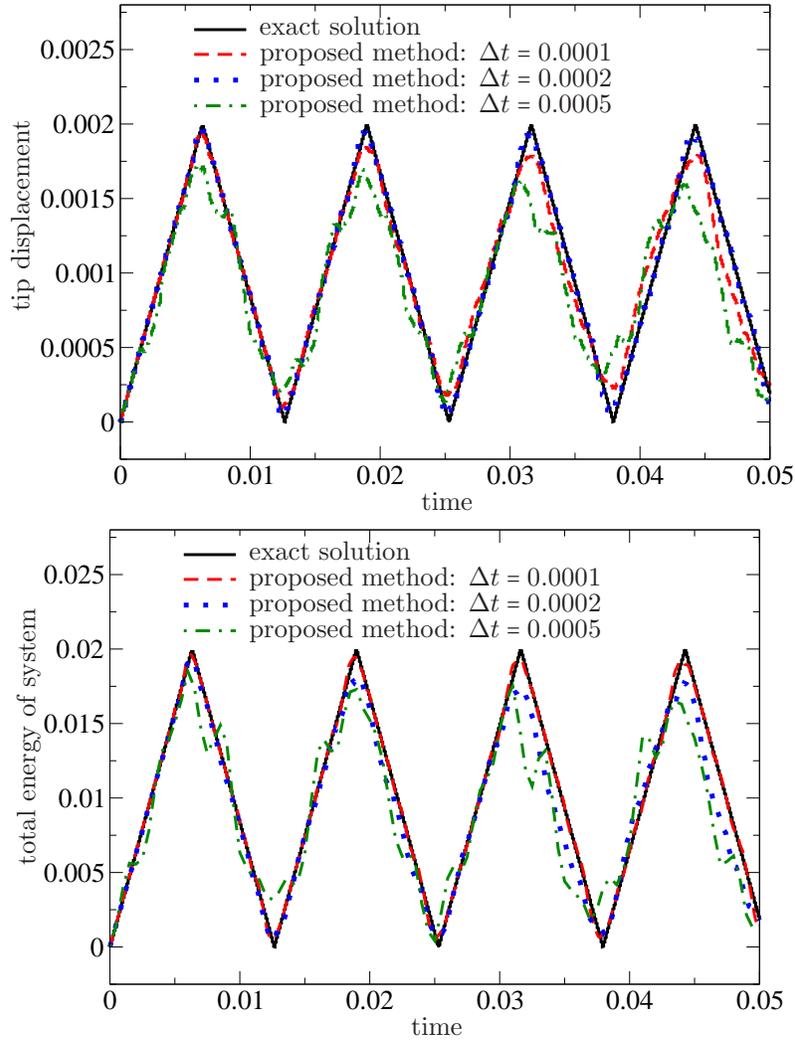

  \centering
  \psfrag{time}{time}
  \psfrag{exact}{exact solution}
  \psfrag{d_tip}{tip displacement}
  \psfrag{enrg}{total energy of system}
  \psfrag{dt1}{proposed method: $\Delta t = 0.0001$}
  \psfrag{dt2}{proposed method: $\Delta t = 0.0002$}
  \psfrag{dt3}{proposed method: $\Delta t = 0.0005$}
  \subfigure{
  \includegraphics[scale=0.4,clip]{Figures/1D_3Subdomain/5-10-5-set2/d.eps}}
  \subfigure{
  \includegraphics[scale=0.4,clip]{Figures/1D_3Subdomain/5-10-5-set2/enrg.eps}}
  \caption{One-dimensional problem with homogeneous 
  properties:~The top and bottom figures, respectively 
  show the tip displacement and the total energy as a 
  function of time. 
    The left and right subdomains employ Newmark average 
    acceleration scheme (which is an implicit scheme), 
    and the middle subdomain employs the central difference 
    scheme (which is an explicit scheme). The subdomain time-steps
    are fixed at $10^{-5}$. In each case the system time-step is
    changed as indicated in the figure. The numerical solutions 
    match well with the exact solution. 
    \emph{The figure clearly shows that, under the proposed 
      multi-time-step coupling method with fixed subdomain 
      time-steps (i.e., fixed $\Delta t_i $), 
      the accuracy can be improved by employing smaller 
      system time-steps.} 
      \label{Fig:1_dimensional_dtip_energy_2}}
\end{figure}

%--------------------------------------------------;
%  Figure: One-dimensional homogeneous bar: drift  ;
%--------------------------------------------------;
\begin{figure}
  \centering
  \psfrag{time}{time}
  \psfrag{drift AB}{$d_{\mbox{drift}}$ at AB interface}
  \psfrag{drift BC}{$d_{\mbox{drift}}$ at BC interface}
  \psfrag{dt1}{proposed method: $\Delta t = 0.0001$}
  \psfrag{dt2}{proposed method: $\Delta t = 0.0002$}
  \psfrag{dt3}{proposed method: $\Delta t = 0.0005$}
  \subfigure{
    \includegraphics[clip,scale=0.4]{Figures/1D_3Subdomain/5-10-5-set2/driftAB.eps}}
  \subfigure{
    \includegraphics[clip,scale=0.4]{Figures/1D_3Subdomain/5-10-5-set2/driftBC.eps}}
  \caption{One-dimensional problem with homogeneous properties:  
    The top figure shows the drift in the displacement at the 
    interface of subdomains $A$ and $B$. The bottom figure 
    shows the drift in displacement at the interface of 
    subdomains $B$ and $C$. \emph{At least for this test 
      problem, one can conclude that there is no appreciable 
      drift in the displacements under the proposed coupling 
      method.} \label{Fig:1_dimensional_drift_2}}
\end{figure}

%------------------------------------------------;
%  Figure: One-dimensional homogeneous bar: LAB  ;
%------------------------------------------------;
\begin{figure}
  \centering
  \psfrag{time}{time}
  \psfrag{lambda AB}{$\lambda_{AB}$}
  \psfrag{lambda BC}{$\lambda_{BC}$}
  \psfrag{exact}{exact solution}
  \psfrag{dt1}{proposed method: $\Delta t = 0.0001$}
  \psfrag{dt2}{proposed method: $\Delta t = 0.0002$}
  \psfrag{dt3}{proposed method: $\Delta t = 0.0005$}
  \subfigure{
    \includegraphics[clip,scale=0.4]{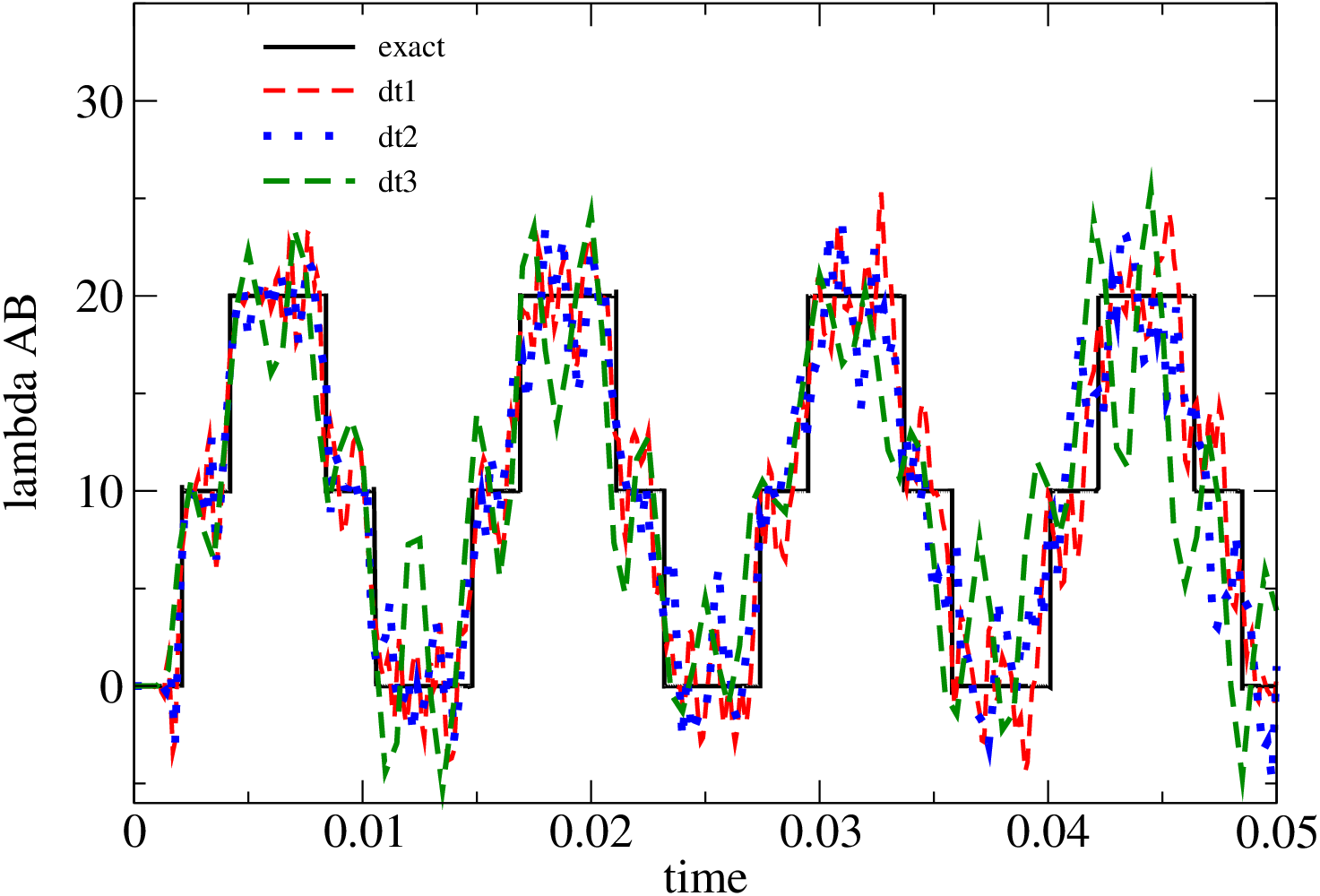}}
  \subfigure{
    \includegraphics[clip,scale=0.4]{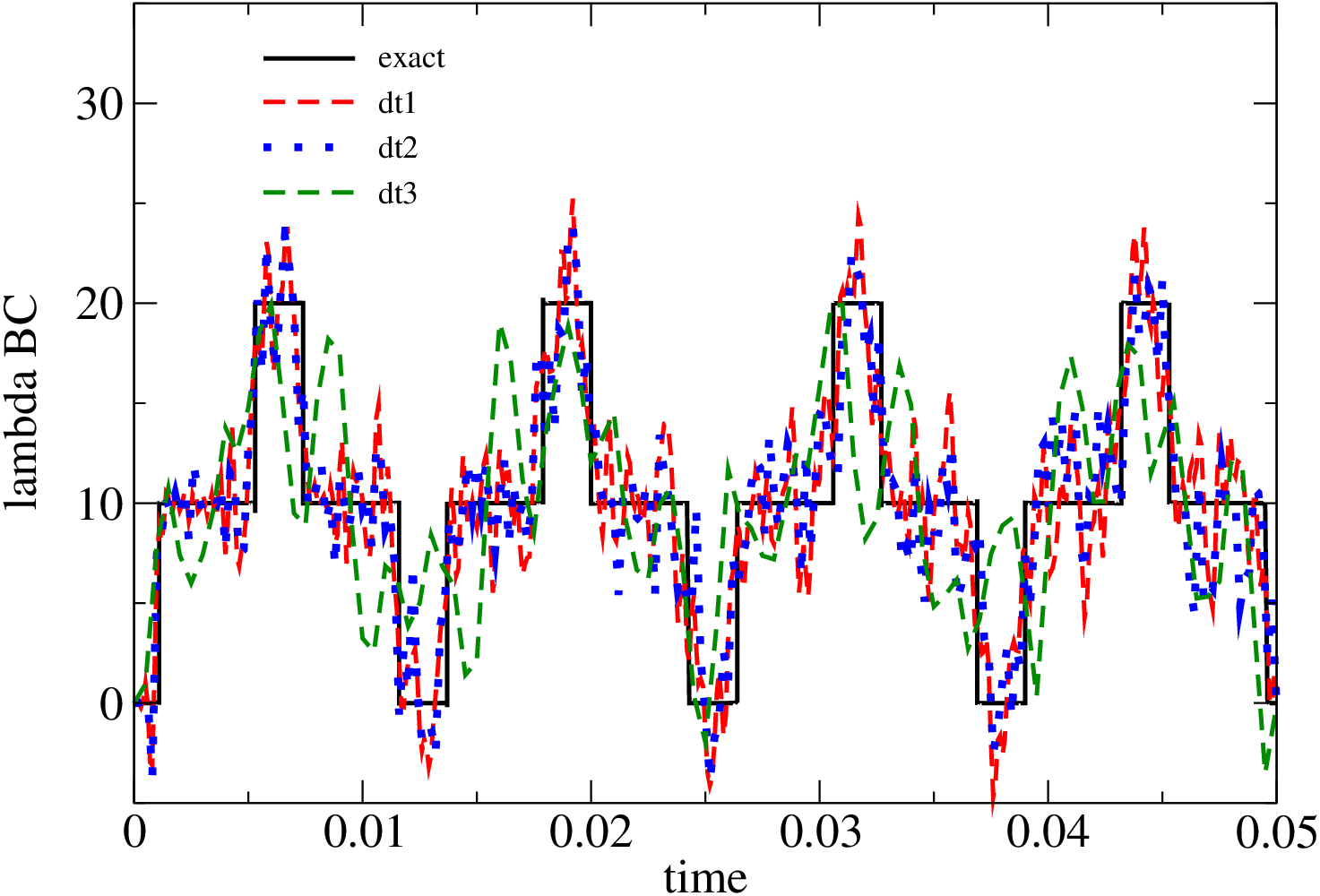}}
  \caption{One-dimensional problem with homogeneous 
    properties: The top figure shows the interface 
    force between subdomains $A$ and $B$. The bottom 
    figure shows the interface force between subdomains 
    $B$ and $C$. 
   \label{Fig:1_dimensional_lambdas_2}}
\end{figure}

%====================================;
%  Figures: Two-dimensional problem  ;
%====================================;

%-----------------------------;
%  Figure: Problem statement  ;
%-----------------------------;
\begin{figure}
	\centering
	\psfrag{Lx}{$L_x/2 = 0.5$}
	\psfrag{Ly}{\begin{sideways} $L_y/2 = 0.5$ \end{sideways}}
	\psfrag{subdomain 1}{subdomain 1}
	\psfrag{subdomain 2}{subdomain 2}
	\psfrag{subdomain 3}{subdomain 3}
	\psfrag{subdomain 4}{subdomain 4}
	\psfrag{x}{$x$}
	\psfrag{y}{$y$}
	\psfrag{fx}{$f_x$}
	\psfrag{fy}{$f_y$}
	\psfrag{pA}{{\color{blue}Point A}}
	\includegraphics[clip, scale=0.85]{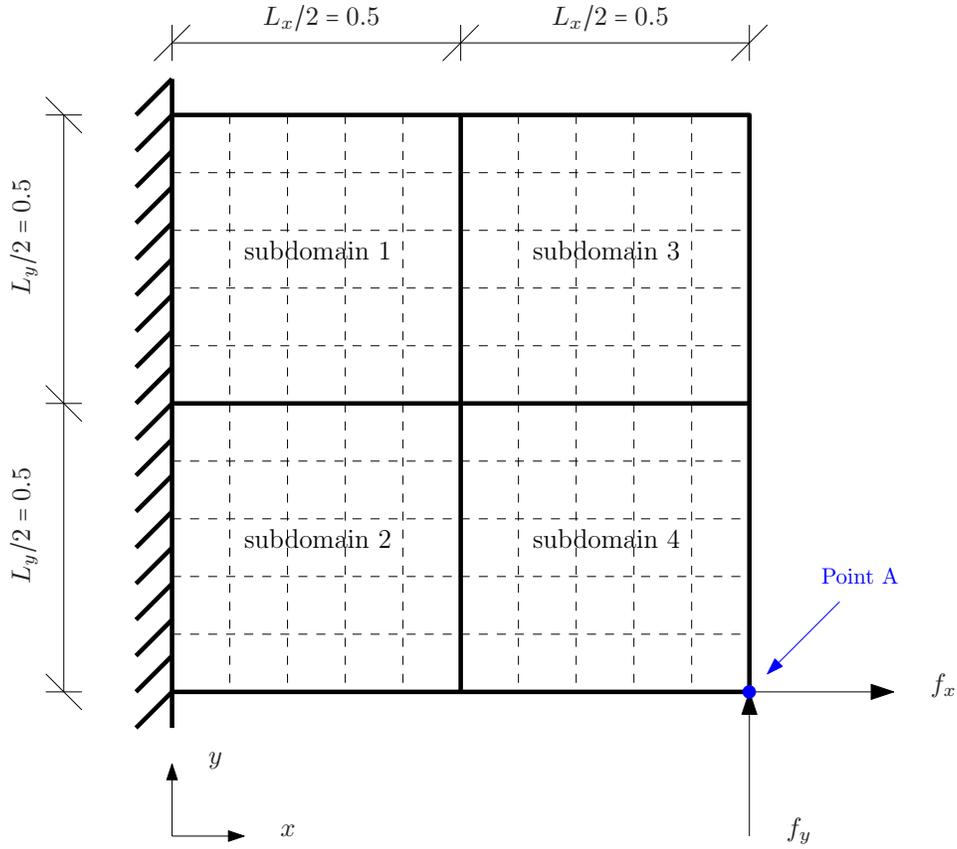}
	\caption{Square plate subjected to a corner force: 
	This figure provides a pictorial description of the test 
	problem. A bi-unit square of homogeneous elastic material 
	is fixed at the left side, a constant force with components 
	$f_x = f_y = 1$ is applied at Point A. Material parameters 
	are taken as $\lambda = 100$, $\mu = 100$, and $\rho = 100$. 
	The computational domain is divided into four subdomains, 
	and the resulting problem is solved using the proposed 
	multi-time-step coupling method. Each subdomain is meshed 
	using 25 equally-sized four-node quadrilateral elements. 
	\label{Fig:Monolithic_2D_pictorial}}
\end{figure}

%-------------------------------------------------;
%  Figure: d and v, Hypothesis 1, Lowering t_sys  ;
%-------------------------------------------------;
\begin{figure}
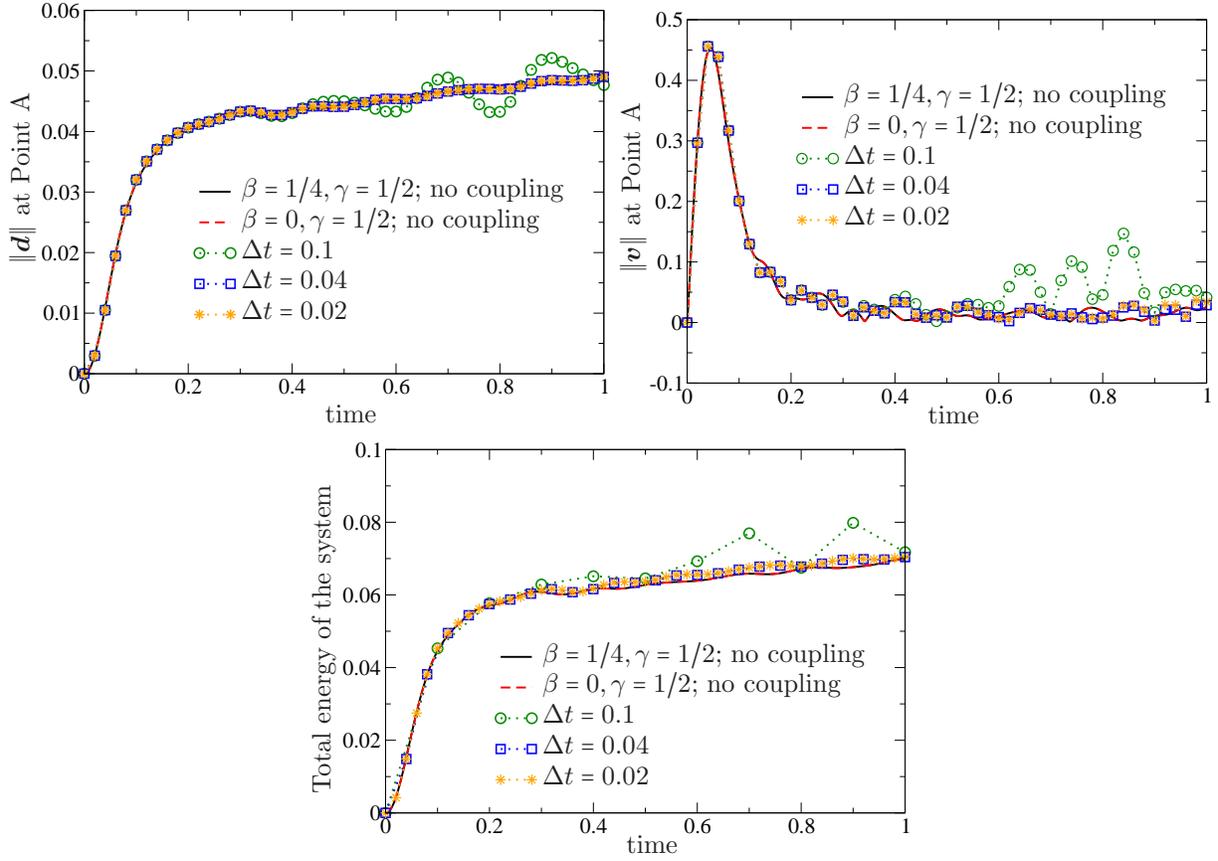

  \centering
  \psfrag{AA}{$\beta = 1/4, \gamma = 1/2$; no coupling}
  \psfrag{CD}{$\beta = 0, \gamma = 1/2$; no coupling}
  %  \psfrag{AA}{Avg. acc. scheme; no coupling; $\Delta t = 0.001$}
  %  \psfrag{CD}{Cent. diff. scheme; no coupling; $\Delta t = 0.001$}
  \psfrag{c1}{$\Delta t = 0.1$}
  \psfrag{c2}{$\Delta t = 0.04$}
  \psfrag{c3}{$\Delta t = 0.02$}
  \psfrag{time}{time}
  \psfrag{d}{$\|\boldsymbol{d}\|$ at Point A}
  \psfrag{v}{$\|\boldsymbol{v}\|$ at Point A}
  \psfrag{E}{Total energy of the system}
  \subfigure{
    \includegraphics[clip, scale = 0.32]{Figures/2D-structured/h1/d.eps}}
  \subfigure{
    \includegraphics[clip, scale = 0.32]{Figures/2D-structured/h1/v.eps}}
  \subfigure{
    \includegraphics[clip, scale = 0.32]{Figures/2D-structured/h1/enrg.eps}}
  \caption{Square plate subjected to a corner force: 
    This figure compares the numerical solutions (2-norm 
    of the displacement at Point A, 2-norm of the velocity 
    at Point A, and the total energy of the system) under 
    the proposed multi-time-step coupling method for various 
    system time-steps keeping the subdomain time-steps fixed. 
    The subdomain time-step in all subdomains is taken as 
    $0.02$. 
    The numerical solutions obtained without decomposing 
    the domain into subdomains and using a relatively 
    smaller time-step, $\Delta t = 0.001$, are also 
    presented for reference.
    \emph{
      The figure clearly demonstrates that, under fixed 
      subdomain time-steps, accuracy can be improved by 
      decreasing the system time-step.} 
    \label{Fig:Monolithic_dve_h1}}
\end{figure}

%---------------------------------------------;
%  Figure: d and v, Hypothesis 2: subcycling  ;
%---------------------------------------------;
\begin{figure}
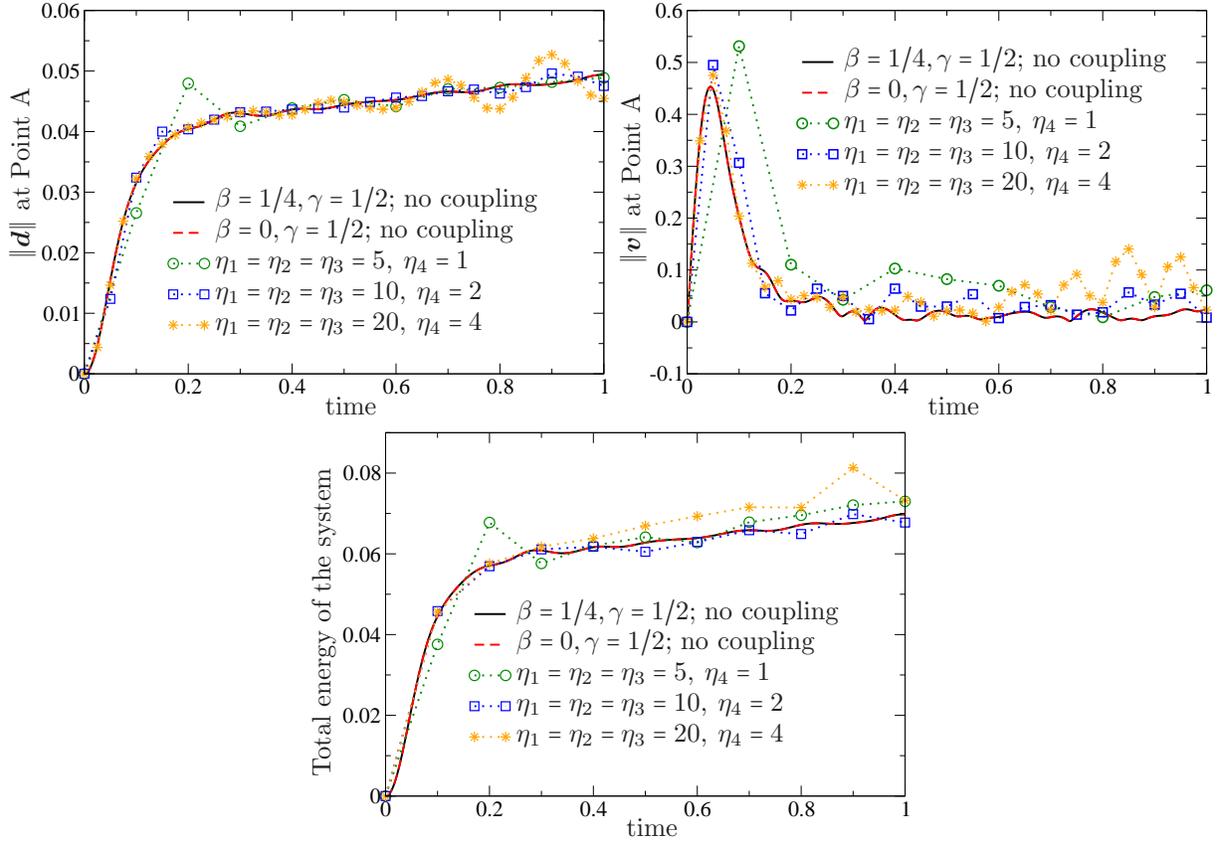

  \centering
  \psfrag{AA}{$\beta = 1/4, \gamma = 1/2$; no coupling}
  \psfrag{CD}{$\beta = 0, \gamma = 1/2$; no coupling}
  \psfrag{c1}{$\eta_1 = \eta_2 = \eta_3 = 5,\;\eta_4 = 1$}
  \psfrag{c2}{$\eta_1 = \eta_2 = \eta_3 = 10,\;\eta_4 = 2$}
  \psfrag{c3}{$\eta_1 = \eta_2 = \eta_3 = 20,\;\eta_4 = 4$}
  \psfrag{time}{time}
  \psfrag{d}{$\|\boldsymbol{d}\|$ at Point A}
  \psfrag{v}{$\|\boldsymbol{v}\|$ at Point A}
  \psfrag{E}{Total energy of the system}
  \subfigure{
    \includegraphics[clip, scale = 0.32]{Figures/2D-structured/h2/d.eps}}
  \subfigure{
    \includegraphics[clip, scale = 0.32]{Figures/2D-structured/h2/v.eps}}
  \subfigure{
    \includegraphics[clip, scale = 0.32]{Figures/2D-structured/h2/enrg.eps}}
  \caption{Square plate subjected to a corner force: 
    The system time-step is taken as $0.1$. The top and 
    middle figures, respectively, show the 2-norm of 
    the displacement and the 2-norm of the velocity at 
    Point A for various values of $\eta_i$ (i.e., for 
    various subdomain time-steps). The bottom figure 
    shows the total energy of the system for values 
    of $\eta_i$. 
    The numerical solutions obtained without decomposing 
    the domain into subdomains and using a relatively 
    smaller time-step, $\Delta t = 0.001$, are also 
    presented for reference.
    \emph{In this case, the accuracy did not improve 
      with subcycling, which is in accordance with 
      the theoretical predictions.} 
    \label{Fig:Monolithic_dve_h2}}	
\end{figure}

\begin{figure}
\centering 
\psfrag{E_interface}{$\mathcal{E}_{\mathrm{interface}}^{(n 
\rightarrow n+1)}$}
\psfrag{time}{time}
\psfrag{set 1}{$\eta_1 = \eta_2 = \eta_3 = 5,\;\eta_4 = 1$}
\psfrag{set 2}{$\eta_1 = \eta_2 = \eta_3 = 10,\;\eta_4 = 2$}
\psfrag{set 3}{$\eta_1 = \eta_2 = \eta_3 = 20,\;\eta_4 = 4$}
\includegraphics[clip, scale = 0.32]{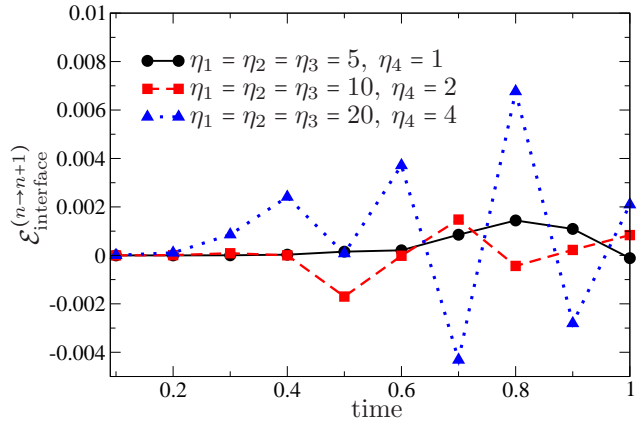}
\caption{Square plate subjected to a corner force:
	The system time-step is taken as $0.1$. The value of 
	$\mathcal{E}_{\mathrm{interface}}^{(n \rightarrow n+1)}$ 
	for the problem presented in \ref{Fig:Monolithic_dve_h2} 
	is plotted. As seen above, subcycling can increase the 
	discretization error at the interface.
	\label{Fig:2D_E_interface}}
\end{figure}

%--------------------------------------------------------------
% 	Figures : Section along x = 0.8
%--------------------------------------------------------------

%------------------------------
% 	norm d and norm v
%------------------------------
\begin{figure}
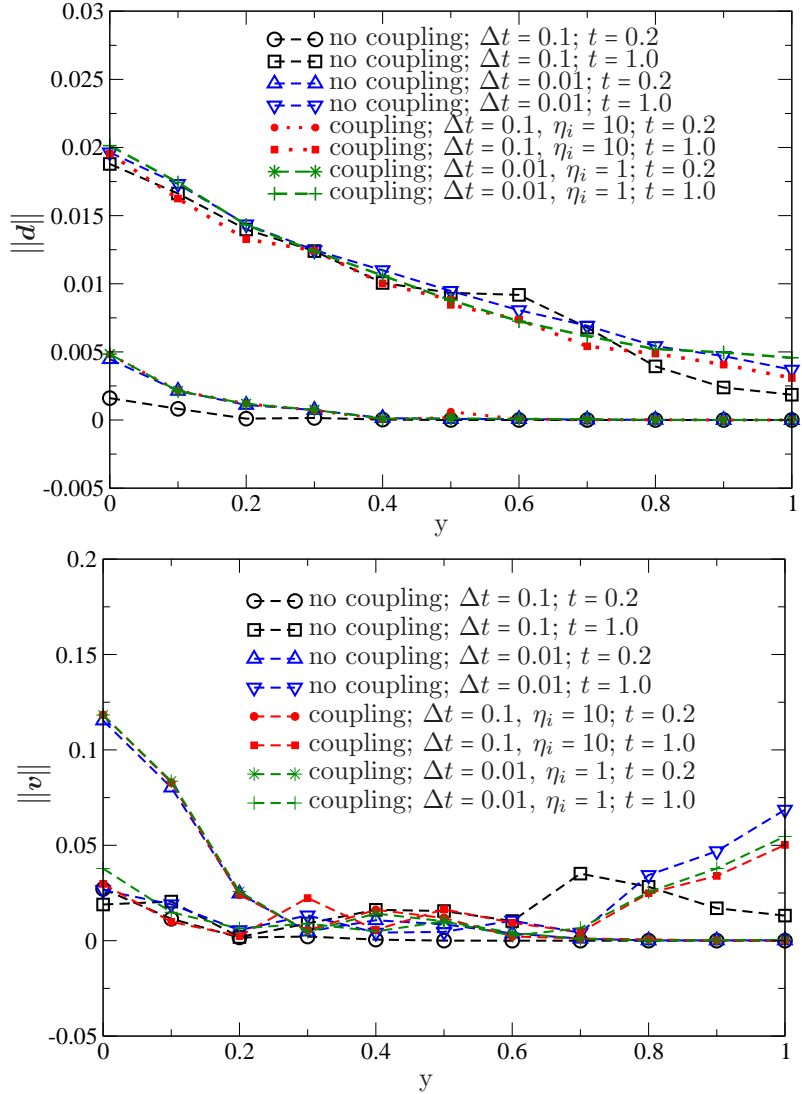

  \centering
  \psfrag{d}{$\big| \big| \boldsymbol{d} \big| \big|$}
  \psfrag{v}{$\big| \big| \boldsymbol{v} \big| \big|$}
  \psfrag{y}{$y$}
  \psfrag{UD_1_t1}{no coupling; $\Delta t = 0.1$; $t = 0.2$}
  \psfrag{UD_1_t2}{no coupling; $\Delta t = 0.1$; $t = 1.0$}
  \psfrag{UD_2_t1}{no coupling; $\Delta t = 0.01$; $t = 0.2$}
  \psfrag{UD_2_t2}{no coupling; $\Delta t = 0.01$; $t = 1.0$}
  \psfrag{DD_1_t1}{coupling; $\Delta t = 0.1,\; \eta_i = 10$; $t = 0.2$}
  \psfrag{DD_1_t2}{coupling; $\Delta t = 0.1,\; \eta_i = 10$; $t = 1.0$}
  \psfrag{DD_2_t1}{coupling; $\Delta t = 0.01,\; \eta_i = 1$; $t = 0.2$}
  \psfrag{DD_2_t2}{coupling; $\Delta t = 0.01,\; \eta_i = 1$; $t = 1.0$}
  \psfrag{time}{y}
  \subfigure{
    \includegraphics[clip,scale = 0.42]{Figures/2D-structured/section-final/d.eps}}
  \subfigure{
    \includegraphics[clip,scale = 0.42]{Figures/2D-structured/section-final/v.eps}}
  \caption{Square plate subjected to a corner force:~
    The figure compares the numerical solutions under 
    the proposed coupling method with that obtained 
    without decomposing into subdomains. 
    The top figure shows the 2-norm of the displacement 
    along y-direction at $x = 0.8$ at various time levels, 
    and the bottom figure shows the corresponding 2-norm 
    of the velocity. 
    One can also notice that there is no significant 
    drift in the displacements at the subdomain 
    interface. 
    \label{Fig:Monolithic_2D_d_and_v_x_dot8}}
\end{figure}

%---------------------------;
%  Figure: Bound on drifts  ;
%---------------------------;
\begin{figure}
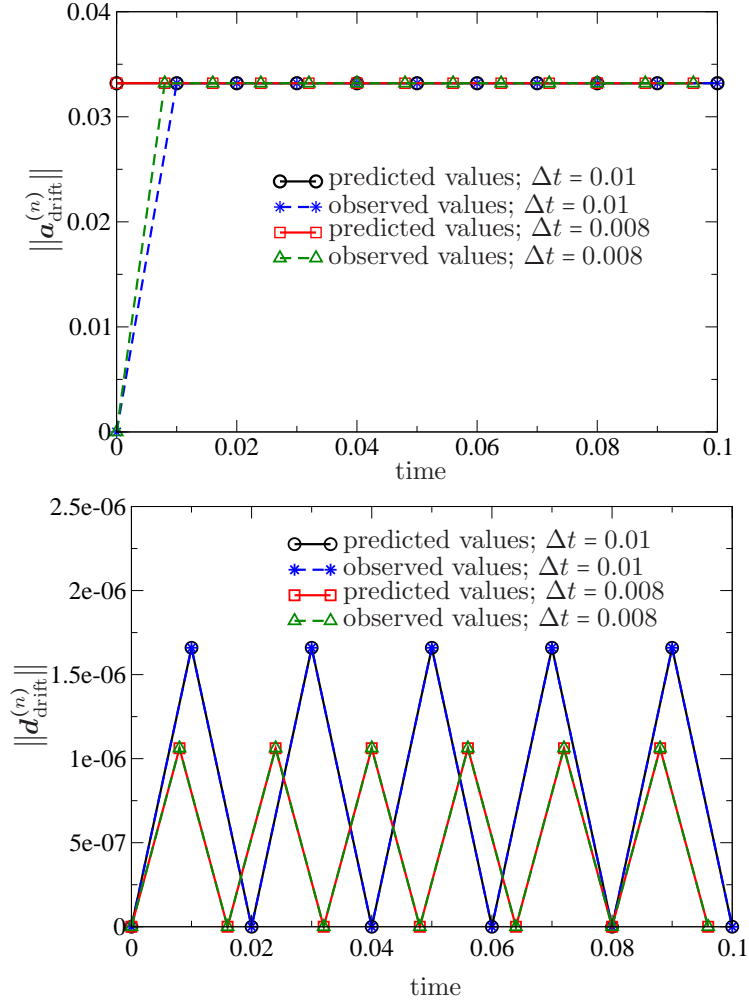

\centering
\psfrag{time}{time}
\psfrag{drift a}{$\big| \big| \boldsymbol{a}_{\mathrm{drift}}^{(n)} \big| \big|$}
\psfrag{drift d}{$\big| \big| \boldsymbol{d}_{\mathrm{drift}}^{(n)} \big| \big|$}
\psfrag{dt1 predicted}{predicted values; $\Delta t = 0.01$}
\psfrag{dt1 obs}{observed values; $\Delta t = 0.01$}
\psfrag{dt2 predicted}{predicted values; $\Delta t = 0.008$}
\psfrag{dt2 obs}{observed values; $\Delta t = 0.008$}
\subfigure{
	\includegraphics[clip, scale = 0.37]{Figures/2D-structured/drift/a.eps}
	}
\subfigure{
	\includegraphics[clip, scale = 0.37]{Figures/2D-structured/drift/d.eps}
	}
\caption{Bounds on drifts: The results in this figure 
substantiate the discussion presented in subsection 
\ref{Subsec:Monolithic_bounds_on_drifts}. The $L_2$-norm 
of the drift in acceleration and and the drift in 
displacement at the subdomain interface are shown. 
Newmark central difference scheme ($\beta = 0$, 
$\gamma = 1/2$) is employed in all subdomains and 
there is no subcycling. The theoretical predictions 
are based on equations \eqref{Eqn:acc_drift}--\eqref{Eqn:disp_drift}. It is noteworthy that the 
drift in the displacements along the subdomain 
interface is small under the proposed coupling 
method. \label{Fig:Bound_on_drift}}
\end{figure}

%==============================================================;
% 	Figure : 2D wave propagation problem 		       ;
%==============================================================;
%-----------------------------------;
%  Figure: pictorial description    ;
%-----------------------------------;
\begin{figure}
  \centering
  \includegraphics[scale = 0.85, clip]{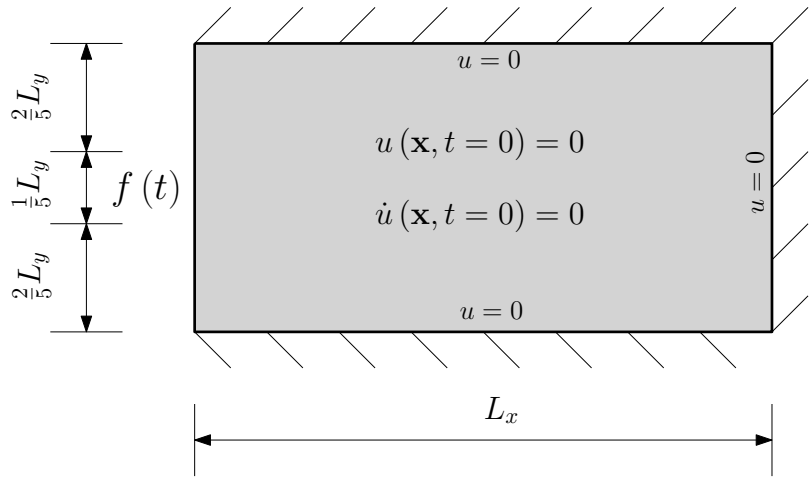}
  \caption{Two-dimensional wave propagation problem: A 
    pictorial description. The elastic body is assumed 
    to be isotropic and homogeneous. The force is applied 
    over a length of $1/5 L_y$ in the middle of the left 
    side of the boundary. The body is allowed to vibrate 
    freely after $t = \tau_{\mathrm{load}}$. No absorbing 
    boundary conditions are applied at the fixed ends.
    \label{Fig:2D_Wave_Pictorial}}
\end{figure}

%-------------------------------;
%  Figure: Domain decomposition ;
%-------------------------------;
\begin{figure}
\centering
\includegraphics[scale = 0.5]{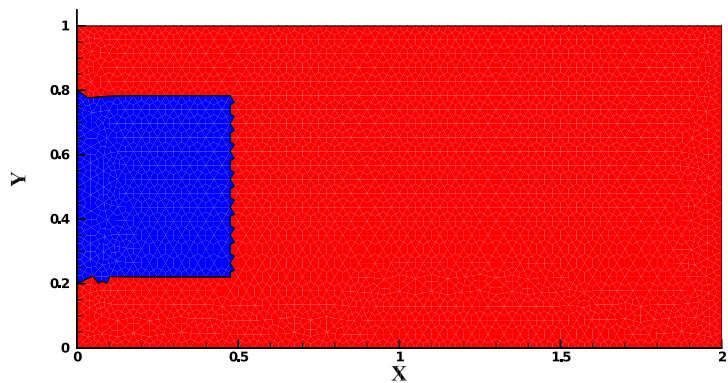}
\caption{Two-dimensional wave propagation problem: 
  The computational domain is divided into 
  two subdomains. Subdomain 1 is shown in 
  blue color, and subdomain 2 is shown in 
  red color. The mesh consists of 5604 
  four-node quadrilateral elements.
  \label{Fig:2D_Wave_DD}}
\end{figure}

%--------------------------------------;
% 	Figure: Case 2 results 	       ;
%--------------------------------------;
\begin{figure}
\centering
\subfigure[$t = 0.25$, $u_{\mathrm{min}} = -0.053$, $u_{\mathrm{max}} = 0.133$]{\includegraphics[scale = 0.3, clip]{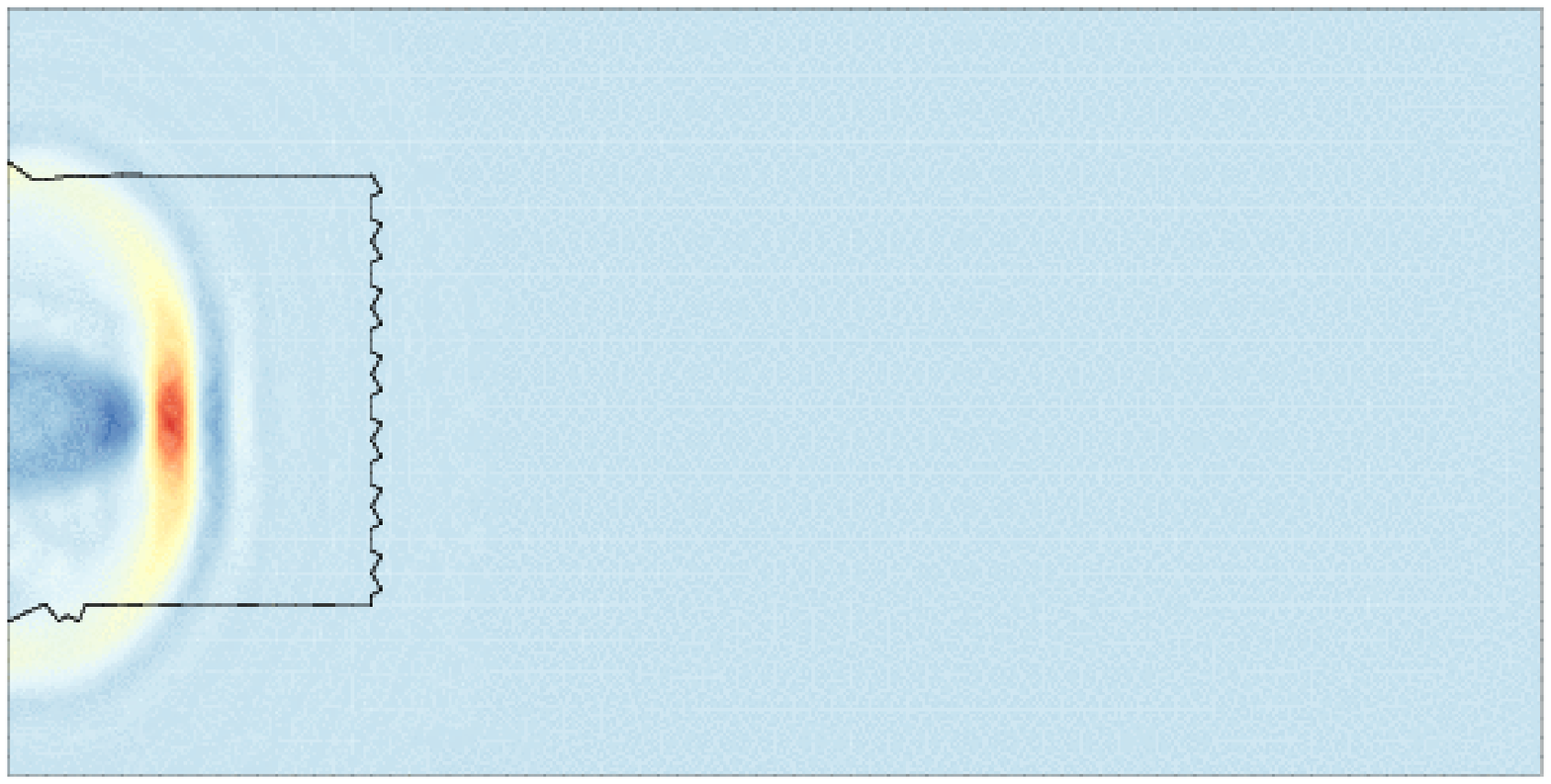}}
\subfigure[$t = 0.50$, $u_{\mathrm{min}} = -0.044$, $u_{\mathrm{max}} = 0.088$]{\includegraphics[scale = 0.3, clip]{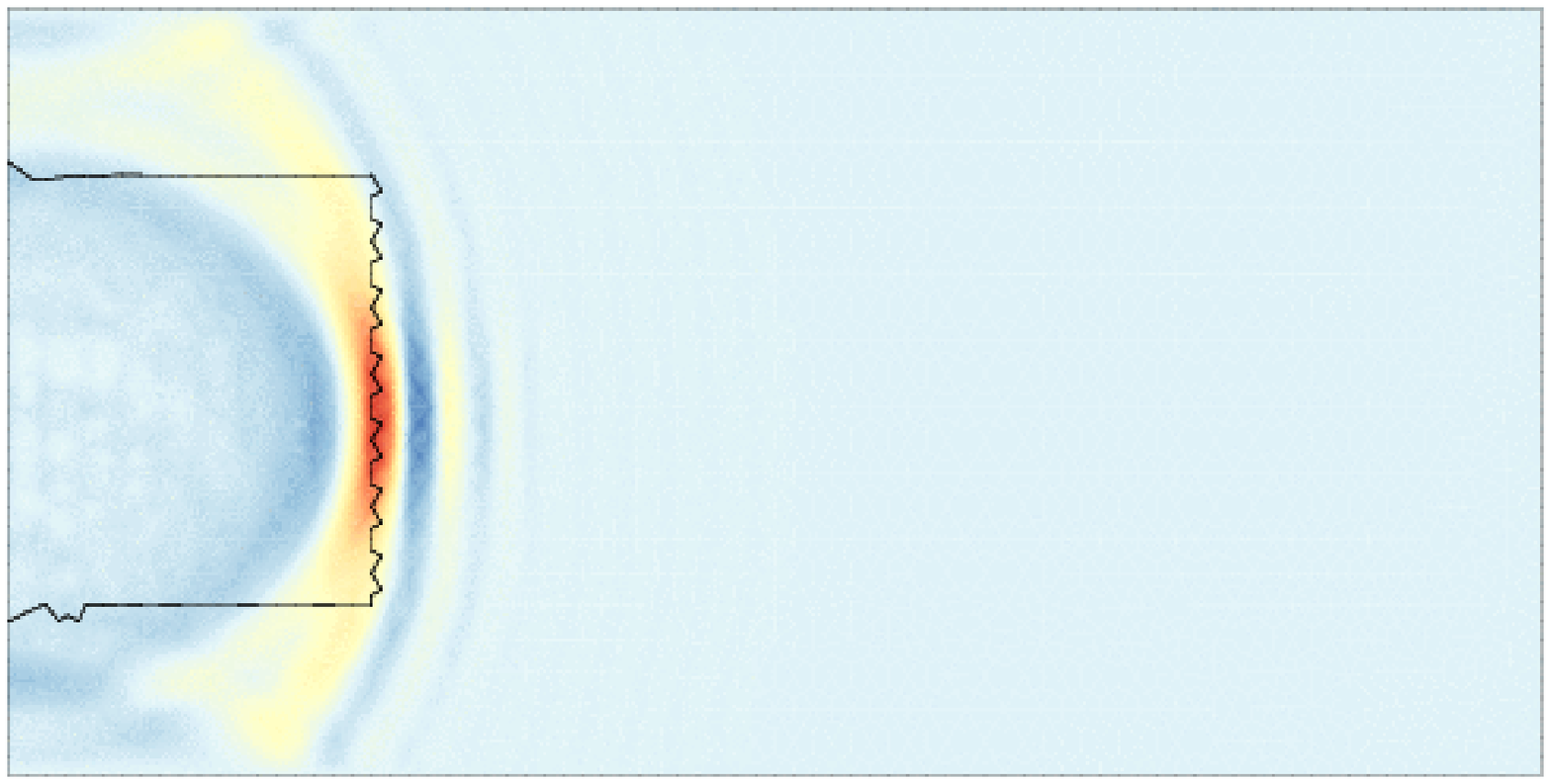}}
\subfigure[$t = 0.75$, $u_{\mathrm{min}} = -0.045$, $u_{\mathrm{max}} = 0.063$]{\includegraphics[scale = 0.3, clip]{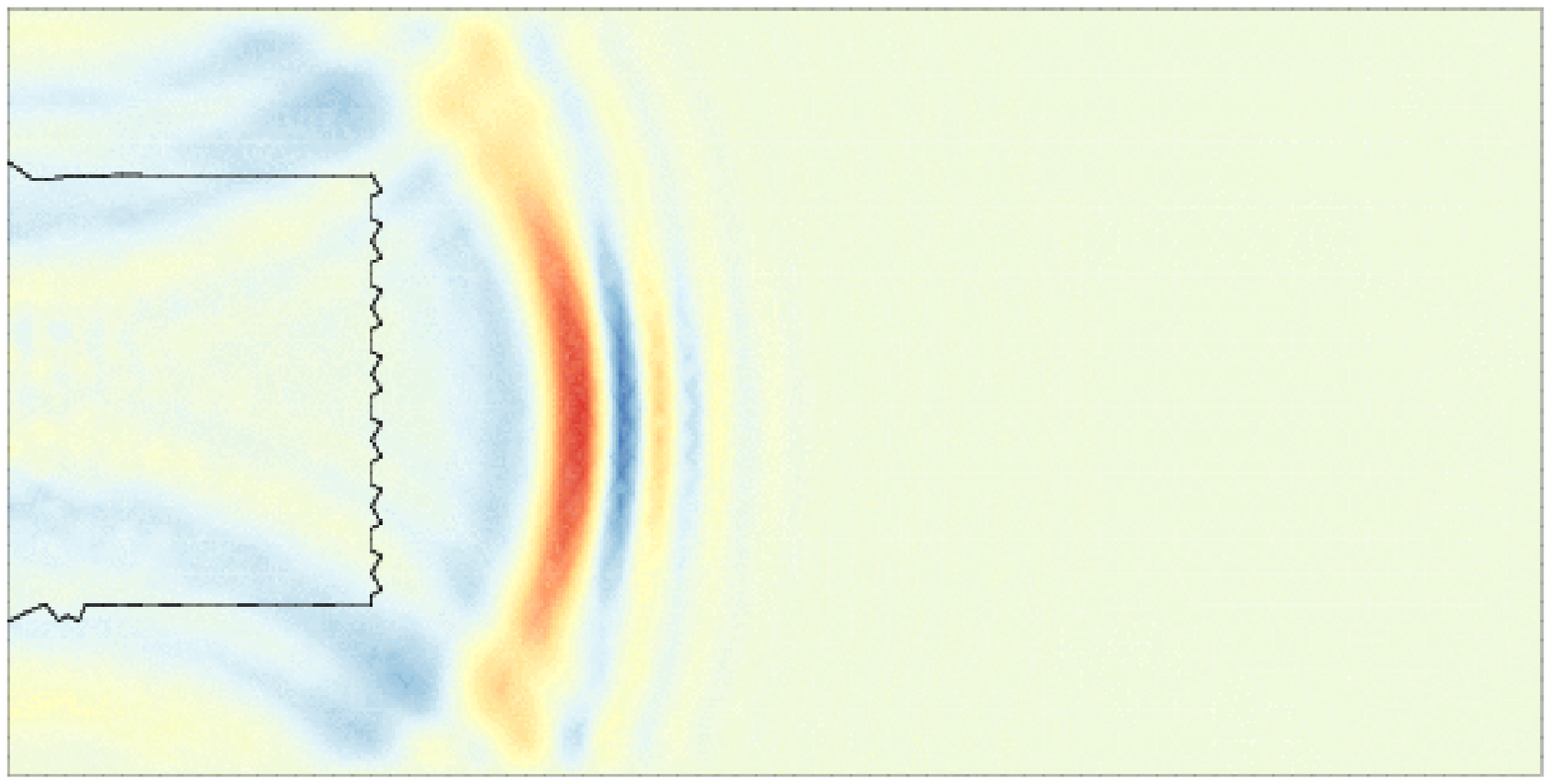}}
\subfigure[$t = 1.00$, $u_{\mathrm{min}} = -0.037$, $u_{\mathrm{max}} = 0.053$]{\includegraphics[scale = 0.3, clip]{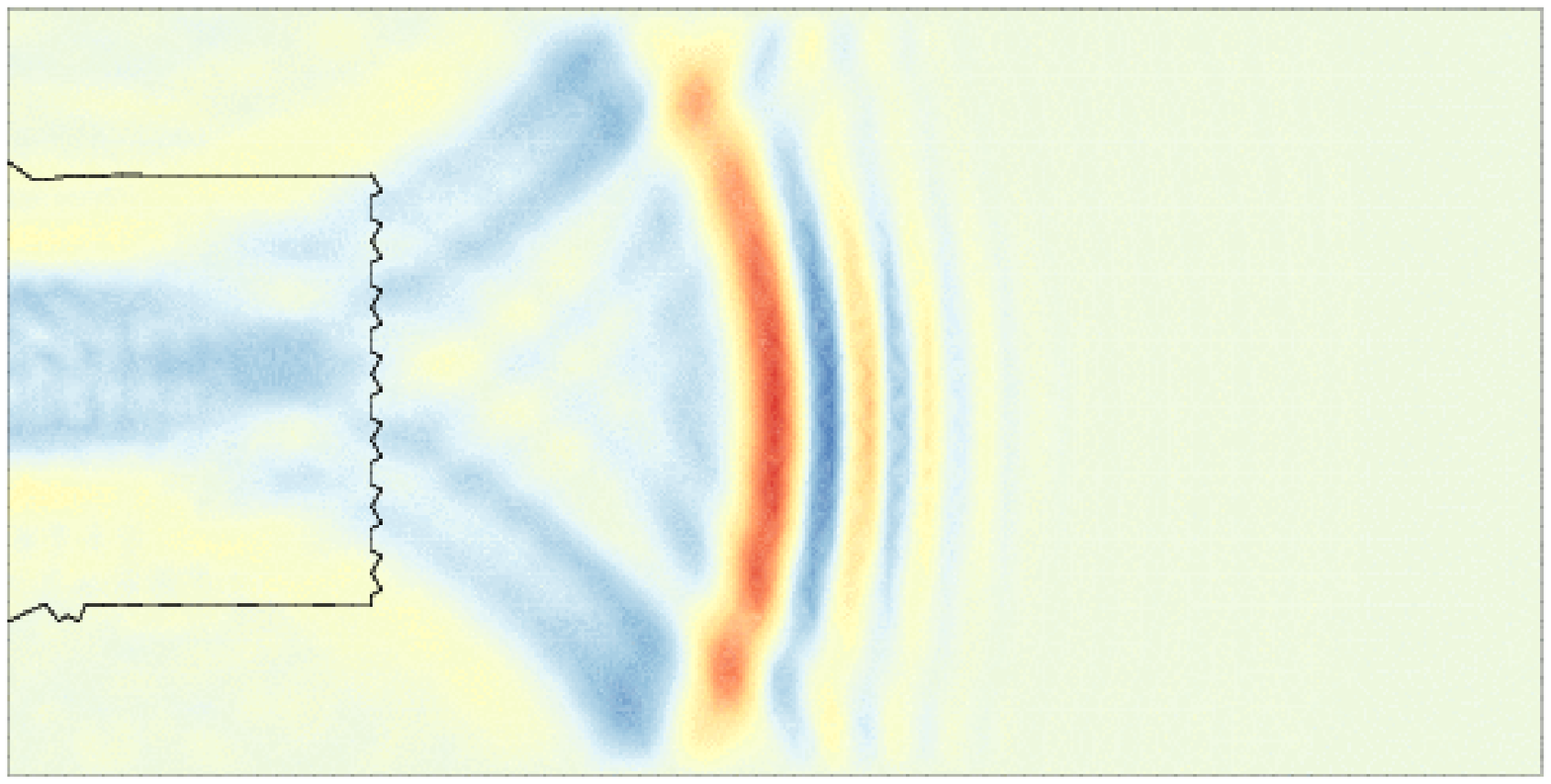}}
\subfigure[$t = 1.25$, $u_{\mathrm{min}} = -0.034$, $u_{\mathrm{max}} = 0.044$]{\includegraphics[scale = 0.3, clip]{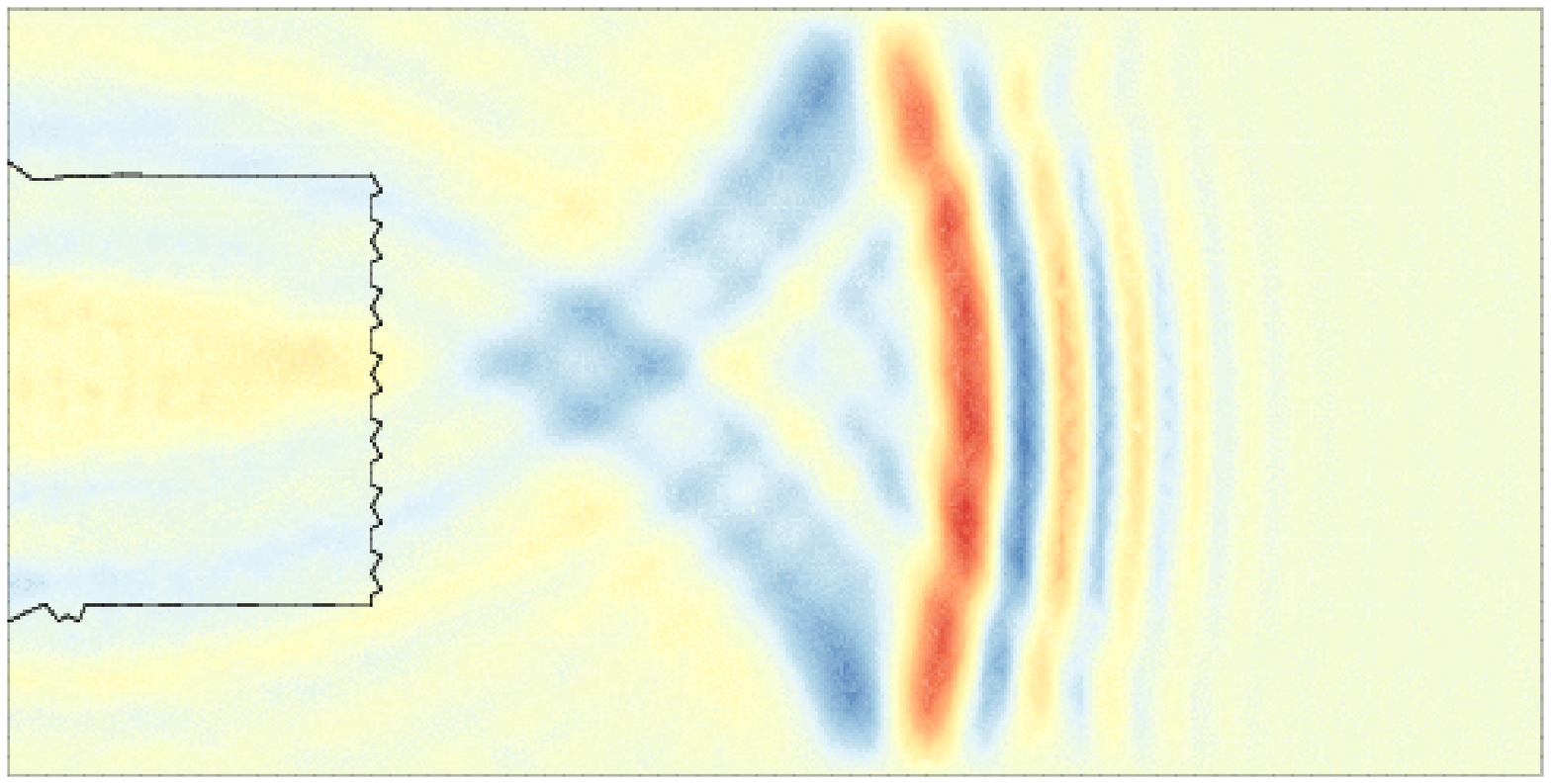}}
\subfigure[$t = 1.50$, $u_{\mathrm{min}} = -0.030$, $u_{\mathrm{max}} = 0.039$]{\includegraphics[scale = 0.3, clip]{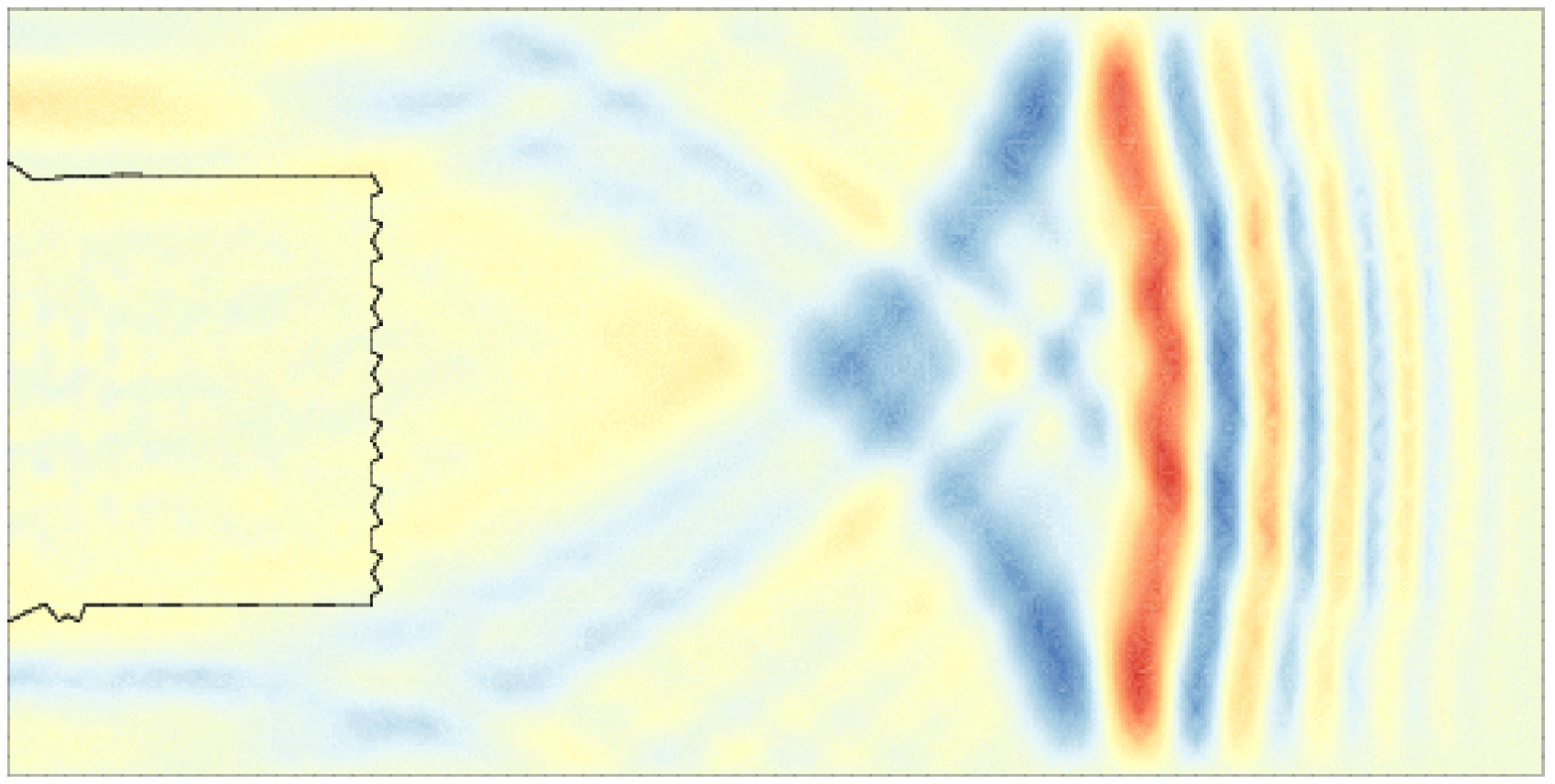}}
\subfigure[$t = 1.75$, $u_{\mathrm{min}} = -0.034$, $u_{\mathrm{max}} = 0.040$]{\includegraphics[scale = 0.3, clip]{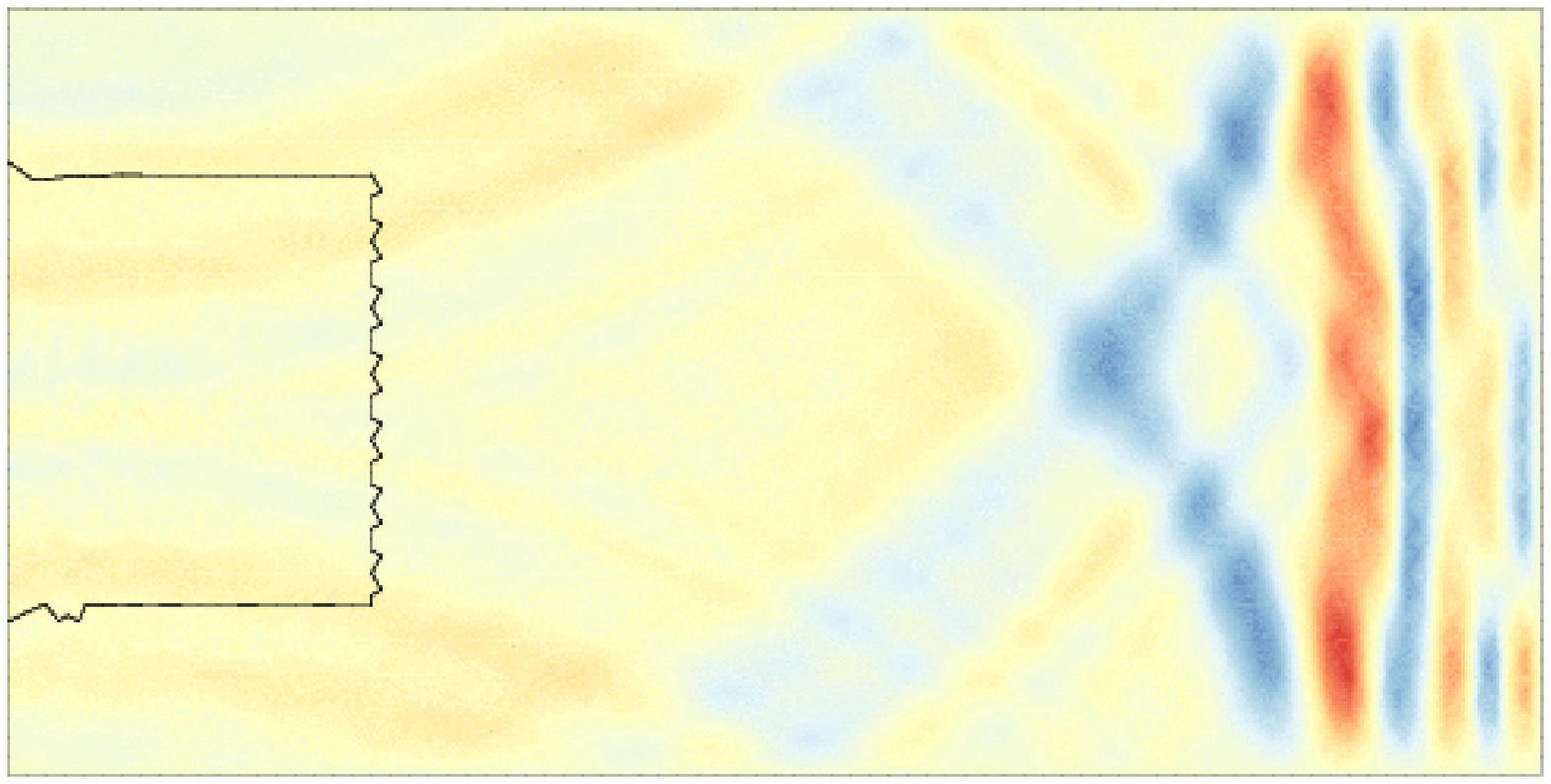}}
\subfigure[$t = 2.00$, $u_{\mathrm{min}} = -0.047$, $u_{\mathrm{max}} = 0.042$]{\includegraphics[scale = 0.3, clip]{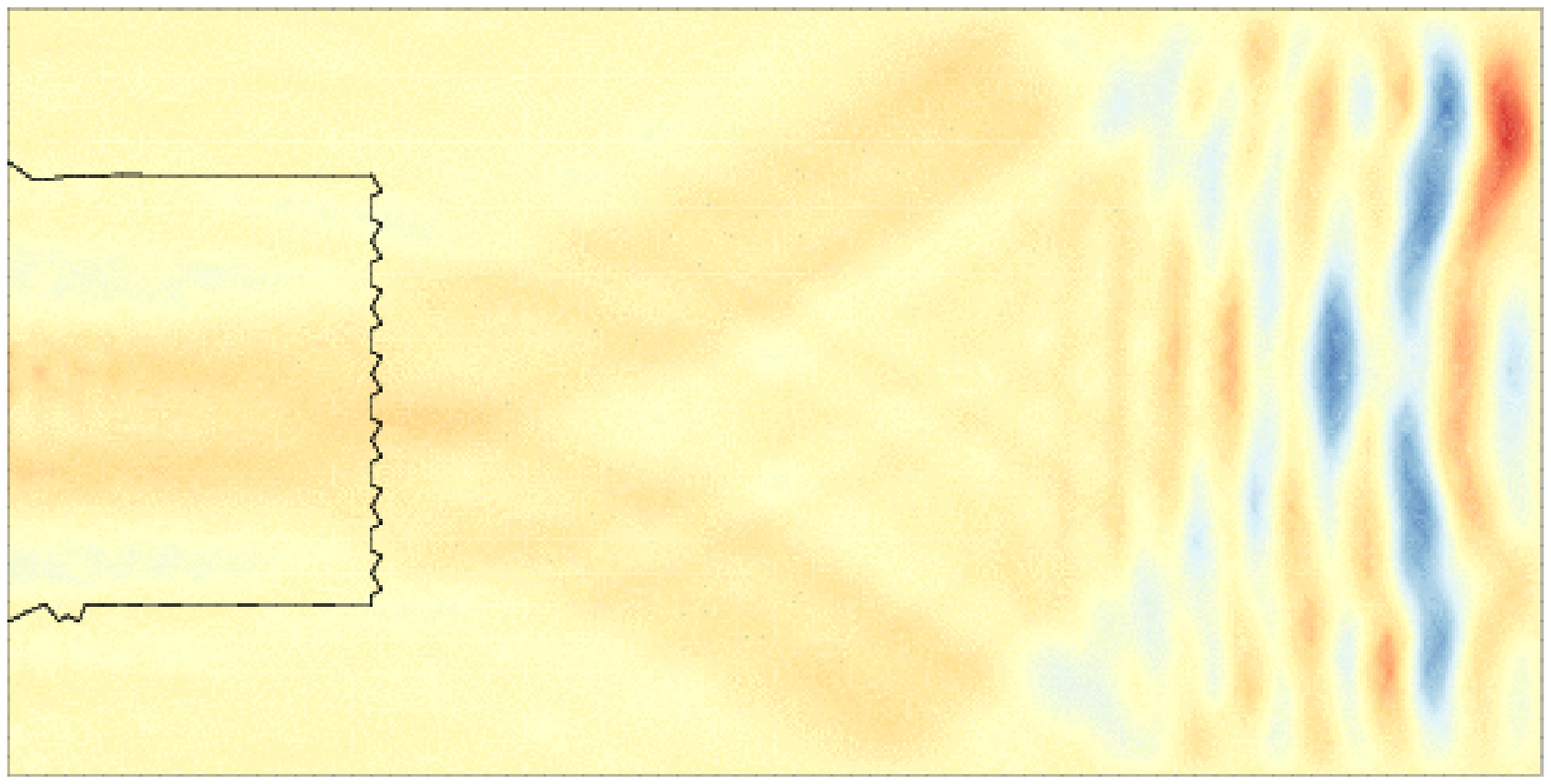}}
\subfigure[$t = 2.25$, $u_{\mathrm{min}} = -0.044$, $u_{\mathrm{max}} = 0.039$]{\includegraphics[scale = 0.3, clip]{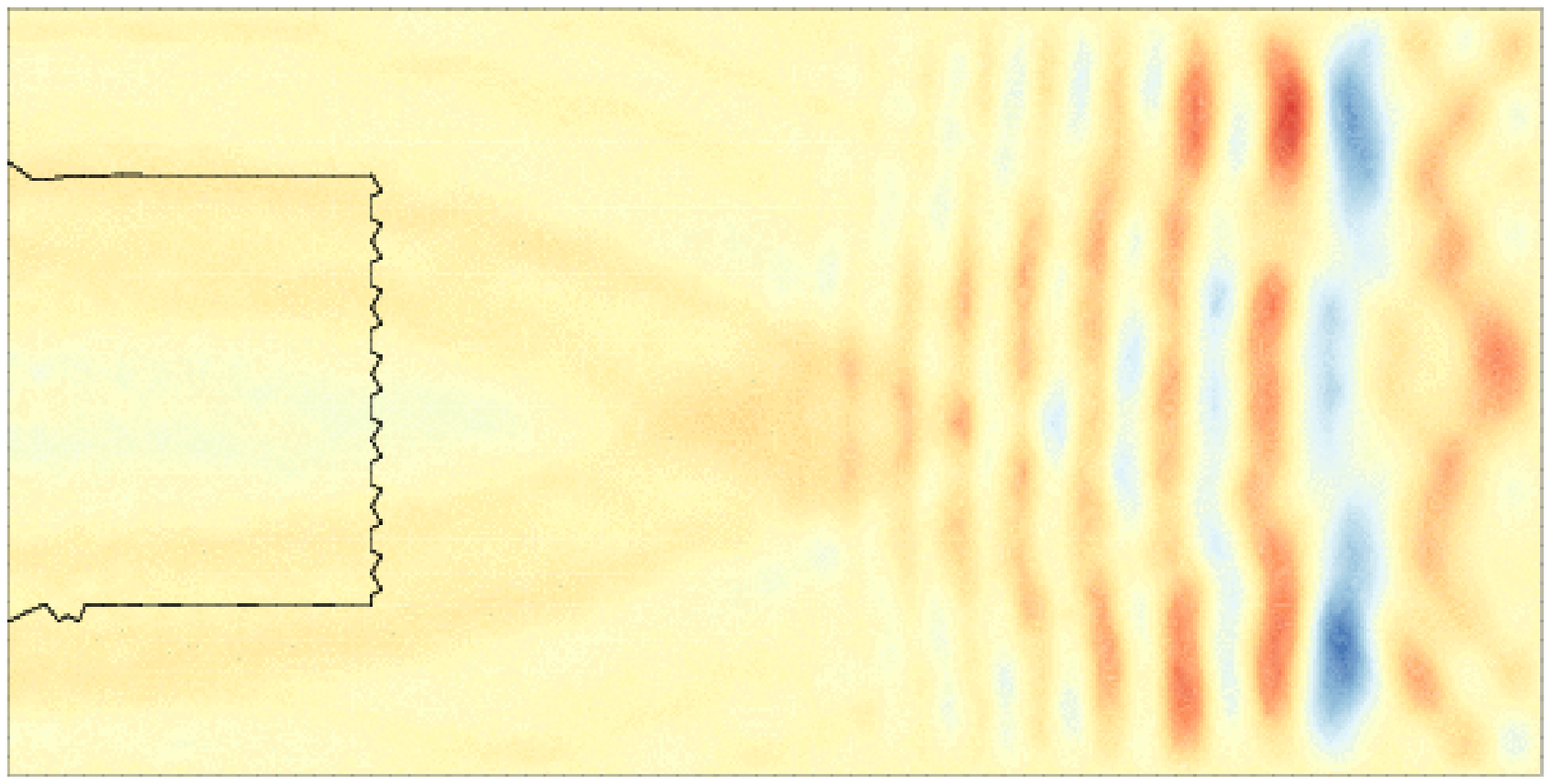}}
\subfigure[$t = 2.50$, $u_{\mathrm{min}} = -0.037$, $u_{\mathrm{max}} = 0.037$]{\includegraphics[scale = 0.3, clip]{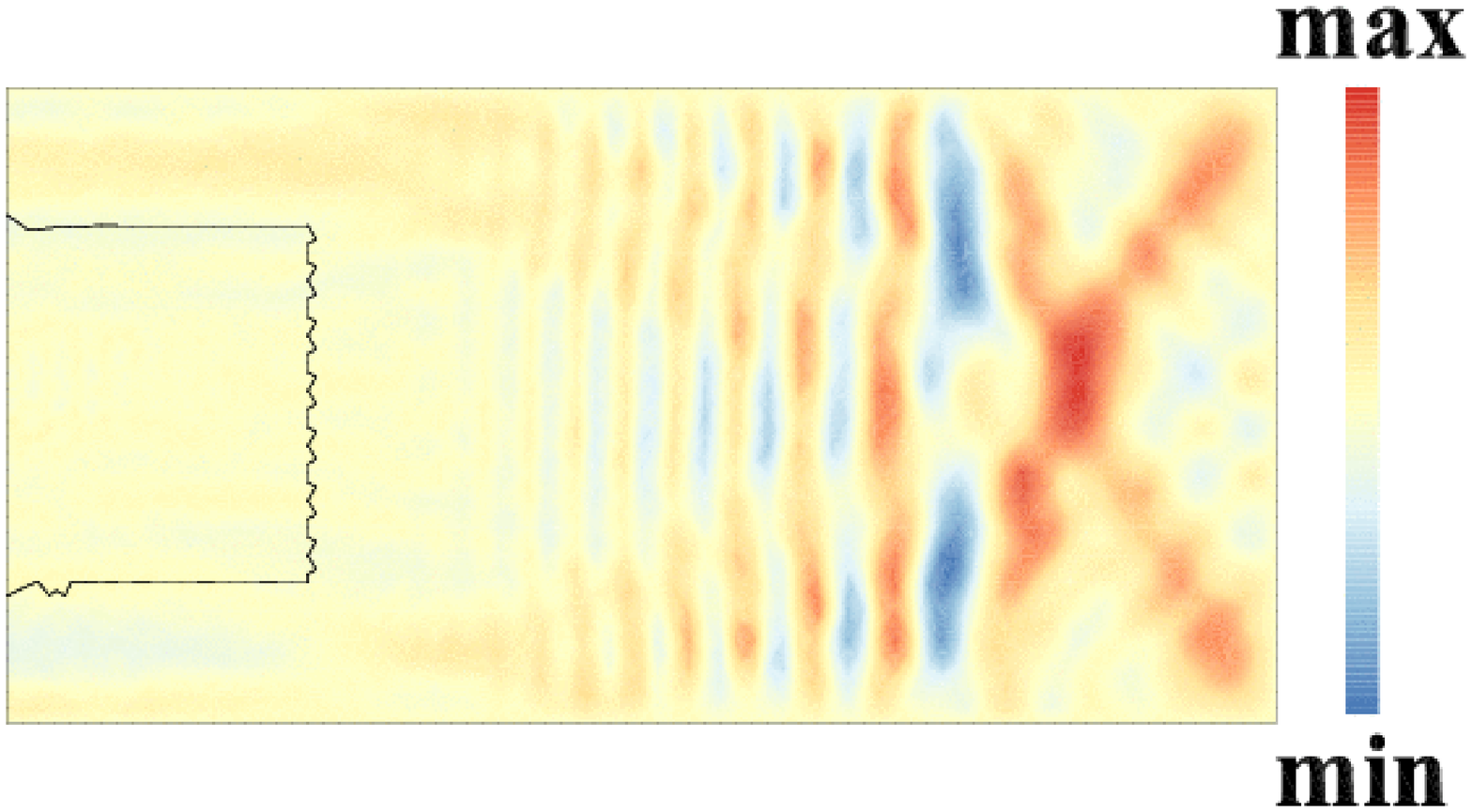}}
\caption{Two-dimensional wave propagation problem: Here, $f_0 = 5.0$, 
$\tau_{\mathrm{load}} = 0.1$, 
and $c_0 = 1$. The system time-step is $\Delta t = 10^{-4}$, and the subdomain
time-steps are $\Delta t_1 = 10^{-5}$, and $\Delta t_2 = 10^{-4}$. 
The subdomain Newmark parameters are 
$\left( \gamma_1 , \beta_1\right) = \left( 1/2, 0\right)$, and
$\left( \gamma_2 , \beta_2\right) = \left( 1/2, 1/4\right)$. 
\label{Fig:2D_Wave_Case2}}
\end{figure}

\end{document}